\documentclass[oneside]{amsart}
\usepackage{amsmath}
\usepackage{amssymb}
\usepackage{esint}
\usepackage[a4paper]{geometry}
\usepackage{hyperref}
\usepackage{enumerate}
\usepackage{amsthm}
\usepackage{dsfont}
\usepackage[foot]{amsaddr}

\theoremstyle{plain}
\newtheorem{thm}{Theorem}[section]
	\newtheorem{prop}[thm]{Proposition}
	\newtheorem{lem}[thm]{Lemma}
	\newtheorem{cor}[thm]{Corollary}

\theoremstyle{definition}

\theoremstyle{remark}
\newtheorem*{rem}{Remark}

\numberwithin{equation}{section}

\begin{document}

\title[Hydrodynamics of degenerate Fermi gases on spherical Fermi surfaces]{Hydrodynamics of degenerate Fermi gases \\ on spherical Fermi surfaces}

\author{Benjamin Anwasia}
\email{\href{mailto:benjamin.anwasia@nyu.edu}{benjamin.anwasia@nyu.edu}}

\author{Diogo Ars\'enio}
\address{Division of Science \\ New York University Abu Dhabi \\ Abu Dhabi \\ United Arab Emirates}
\email{\href{mailto:diogo.arsenio@nyu.edu}{diogo.arsenio@nyu.edu}}

\date{\today}

\begin{abstract}
	We consider the description of a Fermi gas of free electrons given by the Boltzmann--Fermi--Dirac equation, and aim at providing a precise mathematical understanding of the Fermi ground state and its first-order approximation of excited states on the Fermi sphere.

	In order to achieve that, using the framework of hydrodynamic limits in collisional kinetic theory, we identify the low-temperature regimes in which charge-density fluctuations concentrate on the Fermi sphere. In three spatial dimensions or higher, we also characterize the thermodynamic equilibra and energy spectra of fluctuations. This allows us to derive the macroscopic hydrodynamic equations describing how charge densities flow and propagate in metals, thereby providing a precise description of plasma oscillations in conductors. The two-dimensional case is fundamentally different and is handled in a companion article.
	
	Remarkably, our results establish that excited electrons and their energy can be distributed on the Fermi sphere anisotropically, which deviates from the common intuitive assumption that electrons and their energy should be distributed uniformly in all directions.
	
	The hydrodynamic regimes featured in this work are akin to the acoustic limit of the classical Boltzmann equation. However, we emphasize that our derivation holds for arbitrarily fast rates of convergence of the Knudsen number, which significantly extends the applicability of the known proofs of the classical acoustic limit. This suggest that low-temperature limits of Fermi gases provide a promising avenue of research toward a complete understanding of the compressible Euler limit.
\end{abstract}

\maketitle

\tableofcontents

\section{Introduction}

Fermions are elementary quantum particles satisfying the Pauli exclusion principle, which dictates that no two particles can occupy the same quantum state.
Accordingly, in a Fermi gas or liquid close to its absolute-zero temperature, particles cannot all accumulate near the zero-energy level. Instead, they tend to occupy their lowest available energy configuration. At low temperatures, this results in a ground state where all energy levels are filled uniformly up to a fixed energy, know as the Fermi energy or Fermi level $E_{\mathrm{F}}>0$. For the sake of simplicity, we leave particle-spin interactions out of this discussion.

This behavior is effectively captured by the Fermi--Dirac distribution
\begin{equation*}
	f(E)=\frac{1}{1+e^{\frac{E-E_{\mathrm{F}}}{k_\mathrm{B}T}}},
\end{equation*}
which describes the energy spectrum of an ideal Fermi gas at thermodynamic equilibrium, where $T$ is the temperature and $k_\mathrm{B}$ is Boltzmann's constant. Specifically, up to a normalization constant, the value $f(E)$ represents the mean number of particles in the gas which have energy level $E$. It is then readily seen that, as the Fermi gas reaches cold temperatures, i.e., in the regime $T\to 0$, the distribution $f(E)$ tends to a step function, with $f= 1$, if $E<E_\mathrm{F}$, and $f= 0$ if $E>E_\mathrm{F}$.

In nature, even though all matter condensates into a solid or a liquid before it reaches an absolute-zero temperature and quantum effects become significant, there are Fermi systems whose behavior is accurately described by Fermi--Dirac statistics near a ground state. For instance, Landau's theory of Fermi liquids shows how the helium isotopes $^3$He and $^4$He, in their liquid state, display a Fermi-type energy spectrum. Similarly, electrons in a metal also constitute a phase state akin to an ideal Fermi gas at low temperatures, provided one ignores other types of interactions within electrons, with other atoms and with impurities in the metal. Finally, matter at high densities can, in some suitable regimes, transform into a highly compressed plasma of electrons and nuclei. In such a high-density setting, provided the temperature is not too high, the electrons in the plasma form a degenerate Fermi gas, which means that they behave according to Fermi--Dirac statistics near a temperature of absolute zero. This is the case in white dwarf stars.

We refer to \cite{Kittel2004}, \cite{LandauVol5}, \cite{LandauVol9} and \cite{Pathria2021StatisticalMechanics} for further contextualization of the physics of Fermi systems and the Fermi--Dirac equilibrium distribution function.

Importantly, in a Fermi gas at suitably low temperatures, deviations from ground states can only occur when particles in the occupied band $E<E_\mathrm{F}$ are excited and transition into the region of available energy levels $E>E_\mathrm{F}$. This phenomenon is mostly localized and happens principally near the surface in the space of particle momenta determined by the constraint $E=E_\mathrm{F}$, which is called Fermi surface. In the simplest three-dimensional classical case where the energy of a particle of mass $m>0$ is given by its kinetic energy $\frac 12 mv^2$, with velocity $v\in\mathbb{R}^3$, the Fermi surface is a sphere with a radius given by the Fermi velocity $v_\mathrm{F}=(\frac{2E_\mathrm{F}}{m})^\frac 12$.

More generally, the behavior of electrons in a metallic crystal lattice is influenced by the periodic potential created by ions. As a result, the Fermi surface can assume intricate geometrical shapes related to the crystal structure of the metal. The study of Fermi surfaces is of fundamental significance because electrical currents are due to changes in the occupancy of states near the Fermi surface, and electrical properties of metals are therefore determined by the volume and shape of the Fermi surface.

In this work, we focus on the simple setting provided by a Fermi gas of free electrons, without worrying about the complex interactions of electrons with other particles, and we aim at providing a precise mathematical understanding of the Fermi ground state and its first-order approximation of excited states on the Fermi sphere. In order to achieve that, we consider the semiclassical description of a Fermi gas provided by the Boltzmann--Fermi--Dirac equation and, using the framework of hydrodynamic limits in collisional kinetic theory, we identify the low-temperature regimes in which charge-density fluctuations concentrate on the Fermi sphere. Furthermore, we characterize the thermodynamic equilibrium and energy spectrum of fluctuations in such regimes, which then allows us to derive the macroscopic hydrodynamic equations describing how charge densities flow and propagate in metals, thereby providing a precise description of plasma oscillations in conductors.

\section{The Boltzmann--Fermi--Dirac equation}

The nonequilibrium dynamics of an ideal gas filling the whole space and obeying Fermi--Dirac statistics can be described, in the absence of external forces, by the Boltzmann--Fermi--Dirac equation
\begin{equation}\label{eq:BFDE}
	(\partial_t+v\cdot\nabla_x)f=Q_\mathrm{FD}(f),
\end{equation}
where
\begin{equation*}
	f(t,x,v)\geq 0
\end{equation*}
is the number density of particles at time $t\in\mathbb{R}^+$, located at position $x\in\mathbb{R}^d$, with velocity $v\in\mathbb{R}^d$. Throughout this work, the integer parameter $d\geq 2$ denotes the spatial dimension.

The left-hand side above describes the temporal change of the distribution function $f$ due to the linear motion of particles between scattering events, while the right-hand side accounts for collision encounters between particles.

More specifically, the Boltzmann--Fermi--Dirac collision operator $Q_\mathrm{FD}(f)$ acts on the velocity variable $v$ only and is given by the nonlinear integral
\begin{equation}\label{eq:collision_operator}
	Q_\mathrm{FD}(f)(v)=\int_{\mathbb{R}^d\times\mathbb{S}^{d-1}}\Big(f'f_*'(1-\delta f)(1-\delta f_*)-ff_*(1-\delta f')(1-\delta f_*')\Big) b(v-v_*,\sigma) dv_*d\sigma,
\end{equation}
where we have denoted
\begin{equation*}
	f=f(t,x,v),\quad f_*=f(t,x,v_*),\quad f'=f(t,x,v'), \quad f'_*=f(t,x,v'_*),
\end{equation*}
with
\begin{equation*}
	v'=\frac{v+v_*}2+\frac{|v-v_*|}2\sigma
	\quad\text{and}\quad
	v_*'=\frac{v+v_*}2-\frac{|v-v_*|}2\sigma.
\end{equation*}
The constant $\delta>0$ is a small, fixed parameter. It can be expressed in terms of physical quantities and is proportional to the cube of Planck's constant. In particular, formally setting $\delta=0$ in \eqref{eq:BFDE} yields the classical Boltzmann equation.

The vectors $(v',v_*')$ constitute a parametrization, by $\sigma\in\mathbb{S}^{d-1}$, of all possible postcollisional velocities resulting from an elastic scattering event between two particles with precollisional velocities $(v,v_*)$. In other words, they form the family of all solutions to the system
\begin{equation*}
	\begin{aligned}
		v+v_*&= v'+v_*',
		\\
		v^2+v_*^2&=v'^2+v_*'^2,
	\end{aligned}
\end{equation*}
which expresses the conservation of momentum and energy at the microscopic level.

The cross section $b(v-v_*,\sigma)$ (or, equivalently, the collision kernel) is nonnegative and measures the likelihood of a scattering event with ingoing relative velocity $v-v_*$ and outgoing relative velocity $v'-v_*'=|v-v_*|\sigma$. In fact, it only depends on the magnitude of relative velocities and their deflection angle.
As a consequence, it can be written as a function of the magnitude $|v-v_*|=|v'-v_*'|$ and of the cosine of the deflection angle $\frac{v-v_*}{|v-v_*|}\cdot\frac{v'-v_*'}{|v'-v_*'|}$. Thus, we often write
\begin{equation*}
	b(v-v_*,\sigma)=b\left(|v-v_*|,\frac{v-v_*}{|v-v_*|}\cdot\sigma\right),
\end{equation*}
without any ambiguity. It is to be emphasized that the kernel $b$ only alters the weight given to different scattering events in the collision operator, but it is not a probability measure.

The semiclassical description provided by \eqref{eq:BFDE} posits that the transport properties of a Fermi gas or liquid can be described classically. In particular, as we have seen, this means that interactions are assumed to be elastic and that particles stream freely between scattering events according to Newton's laws. As for the quantum properties of particles, they are encoded in the Pauli blocking factors $(1-\delta f)$, $(1-\delta f_*)$, $(1-\delta f')$ and $(1-\delta f_*')$ found in the Boltzmann--Fermi--Dirac collision operator, which enforce that collisions only take particles into unoccupied states, in accordance with Pauli's exclusion principle. As a consequence, the behavior of the gas is conditioned by its trend to Fermi--Dirac equilibria. We discuss the equilibrium states of \eqref{eq:BFDE} in more detail below.

This description works well in a gas of free electrons where interactions between electrons and other elements in the material are neglected, i.e., where the behavior of the dilute gas is dominated by electron--electron interactions. Furthermore, these electron-electron interactions need to remain relatively weak so that the Boltzmann approximation of a dilute gas remains valid, and the mean-field effects of electromagnetic fields can be neglected. We refer to \cite{LandauVol10} for further details on the physical background of \eqref{eq:BFDE}. We also recommend the reading of \cite{cc90}, which we find helpful in bridging the physical with the mathematical aspects of kinetic theory.

\subsection{Dolbeault's wellposedness result}

Mathematically, the Pauli blocking term
\begin{equation*}
	(1-\delta f)
\end{equation*}
forces that the particle distribution satisfy the pointwise bound
\begin{equation}\label{pauli:bound}
	0\leq f \leq \delta^{-1}.
\end{equation}
This is a consequence of the fact that the gain term
\begin{equation*}
	Q_\mathrm{FD}^+(f)(v)=\int_{\mathbb{R}^d\times\mathbb{S}^{d-1}}f'f_*'(1-\delta f)(1-\delta f_*) b dv_*d\sigma
\end{equation*}
vanishes as $f$ approaches the saturation value $\delta^{-1}$, in the same fashion that the nonnegativeness of $f$ follows from the fact that the loss term
\begin{equation*}
	Q_\mathrm{FD}^-(f)(v)=\int_{\mathbb{R}^d\times\mathbb{S}^{d-1}}ff_*(1-\delta f')(1-\delta f_*') b dv_*d\sigma
\end{equation*}
vanishes as $f$ approaches vacuum.

This pointwise a priori bound is a central observation which allowed Dolbeault to establish the wellposedness of the Boltzmann--Fermi--Dirac equation in \cite{d94}, through a simple fixed-point argument. Specifically, for any initial datum
\begin{equation*}
	f|_{t=0}=f^\mathrm{in}\in L^\infty(\mathbb{R}^d\times\mathbb{R}^d)
\end{equation*}
satisfying \eqref{pauli:bound} almost everywhere, he showed that, provided the cross section is globally integrable, i.e.,
\begin{equation}\label{assumption:b:1}
	b(z,\sigma)\in L^1(\mathbb{R}^d\times\mathbb{S}^{d-1}),
\end{equation}
there exists a unique weak solution to \eqref{eq:BFDE} in $L^\infty(\mathbb{R}^+\times\mathbb{R}^d\times\mathbb{R}^d)$ satisfying \eqref{pauli:bound} almost everywhere. Moreover, this solution is absolutely continuous in time, with values in the Banach space of locally integrable functions in $x$ and $v$.

Dolbeault's solutions provide a convenient mathematical setting for our study of Fermi surfaces. In particular, since we are interested in low-temperature regimes where particle velocities are bound to a Fermi sphere, the hypothesis that the collision kernel $b$ is globally integrable does not reduce the generality of the asymptotic macroscopic models derived in our main results, below.

That being said, it should be noted that the well-posedness theory of \eqref{eq:BFDE} has been extended by Lions in \cite{Lions94} to a large range of physically relevant collision kernels, which are only locally integrable and satisfy the quadratic growth control
\begin{equation}\label{assumption:b:2}
	\lim_{|v|\to \infty}|v|^{-2}\int_{K\times\mathbb{S}^{d-1}}b(v-v_*,\sigma)dv_*d\sigma=0,
\end{equation}
for any compact set $K\subset\mathbb{R}^d$. This is the usual DiPerna--Lions assumption on the cross section which was used in \cite{DiPernaLions89} for the construction of renormalized solutions of the Boltzmann equation. However, in this more general context, it is unkown whether solutions to \eqref{eq:BFDE} are unique.

We presume that the results presented in this article remain valid, in some form, in the more general context provided by cross sections satisfying the above quadratic growth control assumption. In that setting, we anticipate that the corresponding proofs would then rely on renormalization techniques, as performed in the framework of hydrodynamic limits of the classical Boltzmann equation (see \cite{SR09}, for instance, for an introduction to the subject).

\subsection{Collisional symmetries and conservation laws}

It is important to observe now that the structure of the Boltzmann--Fermi--Dirac operator \eqref{eq:collision_operator} enjoys the same symmetries as the classical Boltzmann operator under the volume-preserving pre-post-collisional change of variables
\begin{equation}\label{change:prepost}
	(v,v_*,\sigma)\mapsto \left(v',v_*',\frac{v-v_*}{|v-v_*|}\right)
	=\left(\frac{v+v_*}{2}+\frac{|v-v_*|}{2}\sigma,\frac{v+v_*}{2}-\frac{|v-v_*|}{2}\sigma,\frac{v-v_*}{|v-v_*|}\right),
\end{equation}
and particle-exchange change of variables
\begin{equation}\label{change:exchange}
	(v,v_*,\sigma)\mapsto \left(v_*,v,-\sigma\right),
\end{equation}
on $\mathbb{R}^d\times\mathbb{R}^d\times\mathbb{S}^{d-1}$. These changes of variables are sometimes called collisional symmetries.

As in the case of the Boltzmann equation (see \cite{CIP94} for details on the classical Boltzmann equation), it can then be shown that the equilibrium identity
\begin{equation*}
	\int_{\mathbb{R}^d}Q_\mathrm{FD}(f)(v)\phi(v)dv=0
\end{equation*}
is solved for all suitable densities $f(v)$ if and only if $\phi(v)$ is a collision invariant, i.e., it belongs to the linear span of the set
\begin{equation*}
	\{1,v_1,\ldots,v_d,|v|^2\}.
\end{equation*}

Therefore, formally integrating \eqref{eq:BFDE} against collision invariants, we deduce the local conservation laws of mass (or continuity equation), momentum, and energy, respectively written as
\begin{align}
	\label{cons:mass}\partial_t \int_{\mathbb{R}^d}f(t,x,v) dv+\nabla_x \cdot\int_{\mathbb{R}^d}vf(t,x,v)dv&=0,
	\\
	\label{cons:momentum}\partial_t \int_{\mathbb{R}^d}vf(t,x,v) dv+\nabla_x \cdot\int_{\mathbb{R}^d}(v\otimes v)f(t,x,v)dv&=0,
	\\
	\label{cons:energy}\partial_t \int_{\mathbb{R}^d}\frac 12 |v|^2f(t,x,v) dv+\nabla_x \cdot\int_{\mathbb{R}^d}\frac 12 |v|^2vf(t,x,v)dv&=0.
\end{align}
It is then customary to introduce the macroscopic density $\rho$, the bulk (or drift) velocity $U$, the internal energy $E$, the stress (or pressure) tensor $P$ and the heat flux vector $q$, which are macroscopic observable fields defined by the moments
\begin{gather*}
	\rho= \int_{\mathbb{R}^d}f dv,
	\qquad \rho U= \int_{\mathbb{R}^d}vf dv,
	\qquad E= \int_{\mathbb{R}^d}\frac 12|v-U|^2f dv,
	\\
	P=\int_{\mathbb{R}^d}(v-U)\otimes (v-U)f dv,
	\qquad q=\int_{\mathbb{R}^d}\frac 12 |v-U|^2(v-U)f dv.
\end{gather*}
In this notation, the above conservation laws can then be recast with macroscopic variables as
\begin{align*}
	\partial_t \rho+\nabla_x \cdot(\rho U) &=0,
	\\
	\partial_t (\rho U)+\nabla_x \cdot(\rho U\otimes U+ P)&=0,
	\\
	\partial_t \left(\frac 12\rho U^2+E\right)+\nabla_x \cdot
	\left(\left(\frac 12\rho U^2+E\right)U+PU+q\right)&=0.
\end{align*}

For globally integrable cross sections, Dolbeault also showed in \cite{d94} how the unique solution to \eqref{eq:BFDE} satisfies the global conservation laws which are naturally expected from collisional kinetic equations. Specifically, if the initial datum is such that
\begin{equation}\label{initial:bound}
	(1+|x|^2+|v|^2)f^\mathrm{in}(x,v)\in L^1(\mathbb{R}^d\times\mathbb{R}^d),
\end{equation}
he showed that the corresponding solution $f(t,x,v)$ satisfies that
\begin{align*}
	\int_{\mathbb{R}^d\times\mathbb{R}^d}f(t,x,v)dxdv&=\int_{\mathbb{R}^d\times\mathbb{R}^d}f^\mathrm{in}(x,v)dxdv,
	\\
	\int_{\mathbb{R}^d\times\mathbb{R}^d}|x|^2f(t,x,v)dxdv&=\int_{\mathbb{R}^d\times\mathbb{R}^d}|x+tv|^2f^\mathrm{in}(x,v)dxdv,
	\\
	\int_{\mathbb{R}^d\times\mathbb{R}^d}|v|^2f(t,x,v)dxdv&=\int_{\mathbb{R}^d\times\mathbb{R}^d}|v|^2f^\mathrm{in}(x,v)dxdv,
\end{align*}
for all $t\in \mathbb{R}^+$, whence
\begin{equation}\label{propagation:initial:bound}
	(1+|x|^2+|v|^2)f(t,x,v)\in L^\infty_\mathrm{loc}(\mathbb{R}^+;L^1(\mathbb{R}^d\times\mathbb{R}^d)).
\end{equation}
This global bound is then sufficient to show that
\begin{equation}\label{collision_operator:quad_control}
	(1+|x|^2+|v|^2)Q_\mathrm{FD}^\pm (f)\in L^\infty_\mathrm{loc}(\mathbb{R}^+;L^1(\mathbb{R}^d\times\mathbb{R}^d)),
\end{equation}
and, consequently, that the local conservation laws of mass \eqref{cons:mass} and momentum \eqref{cons:momentum} hold in a weak sense.

Note that the validity of the conservation of energy \eqref{cons:energy} was not addressed in \cite{d94}. Nevertheless, previous methods introduced by Perthame in \cite{p89} in the context of the BGK equation apply here to deduce that
\begin{equation}\label{cubic:bound}
	|v|^3 f\in L^1_\mathrm{loc}(\mathbb{R}^+\times\mathbb{R}^d;L^1(\mathbb{R}^d)).
\end{equation}
More precisely, applying Lemma 2 from \cite{p89} with the bounds \eqref{initial:bound} and \eqref{collision_operator:quad_control} readily shows that \eqref{cubic:bound} holds true. As a consequence, one can deduce that the conservation of energy \eqref{cons:energy} holds in a weak sense, as well. However, this only seems to work well when the spatial domain is the whole Euclidean space and the solution is globally integrable in space.

With regard to the solutions constructed in \cite{Lions94} for locally integrable cross sections satisfying the quadratic growth assumption \eqref{assumption:b:2}, it remains unknown whether all exact conservation laws hold or not.
Specifically, while the continuity equation \eqref{cons:mass} remains valid in this setting, one cannot, in general, establish the validity of \eqref{cons:momentum} and \eqref{cons:energy}. The difficulty in doing so originates in the fact that the bound \eqref{collision_operator:quad_control} may no longer hold true. Instead, in addition to the continuity equation \eqref{cons:mass}, one can only show that
\begin{align*}
	\int_{\mathbb{R}^d\times\mathbb{R}^d}vf(t,x,v)dxdv&=\int_{\mathbb{R}^d\times\mathbb{R}^d}vf^\mathrm{in}(x,v)dxdv,
	\\
	\int_{\mathbb{R}^d\times\mathbb{R}^d}|x|^2f(t,x,v)dxdv&\leq \int_{\mathbb{R}^d\times\mathbb{R}^d}|x+tv|^2f^\mathrm{in}(x,v)dxdv,
	\\
	\int_{\mathbb{R}^d\times\mathbb{R}^d}|v|^2f(t,x,v)dxdv&\leq \int_{\mathbb{R}^d\times\mathbb{R}^d}|v|^2f^\mathrm{in}(x,v)dxdv,
\end{align*}
for all $t\in\mathbb{R}^+$. This is analogous to what one can expect for renormalized solutions of the classical Boltzmann equation.

\subsection{The $H$-theorem}

The Boltzmann--Fermi--Dirac equation enjoys a quantum version of the $H$-theorem, which is similar to the classical result for the Boltzmann equation. This feature is essential to our work. Formally, it is obtained by multiplying \eqref{eq:BFDE} by
\begin{equation*}
	\delta\log\left(\frac{\delta f}{1-\delta f}\right)
\end{equation*}
and then integrating in $x$ and $v$ to reach, with suitable applications of the collisional symmetries \eqref{change:prepost} and \eqref{change:exchange}, that
\begin{equation}\label{thm:H}
	H(f)(t)+ \int_0^tD(f)(s)ds = H(f^\mathrm{in}),
\end{equation}
where
\begin{equation*}
	H(f)=\int_{\mathbb{R}^d\times\mathbb{R}^d} (\delta f\log \delta f+(1-\delta f)\log(1-\delta f)) dxdv
\end{equation*}
is the entropy, and
\begin{equation}\label{entropy:dissipation}
	\begin{aligned}
		D(f)
		&=\frac 1{4\delta}\int_{\mathbb{R}^d\times\mathbb{R}^d\times\mathbb{R}^d\times\mathbb{S}^{d-1}}
		\big(
		\delta f' \delta f_{ *}'(1-\delta f)(1-\delta f_{ *})
		-\delta f \delta f_{ *}(1-\delta f')(1-\delta f_{ *}')
		\big)
		\\
		&\quad\times\log\left(
		\frac{\delta f' \delta f_{ *}'(1-\delta f)(1-\delta f_{ *})}
		{\delta f \delta f_{ *}(1-\delta f')(1-\delta f_{ *}')}
		\right)b(v-v_*,\sigma)dxdvdv_*d\sigma
	\end{aligned}
\end{equation}
is the entropy dissipation. Note that the integrand of $D(f)$ is nonnegative.

A standard argument (see Lemma 4 in \cite{d94}, for instance) readily shows that the pointwise bound \eqref{pauli:bound} on $f^\mathrm{in}$ and $f$ in combination with \eqref{initial:bound} and \eqref{propagation:initial:bound} entails that $f^\mathrm{in}\log f^\mathrm{in}\in L^1(\mathbb{R}^d\times\mathbb{R}^d)$ and $f\log f\in L^\infty_\mathrm{loc}(\mathbb{R}^+;L^1(\mathbb{R}^d\times\mathbb{R}^d))$. But \eqref{thm:H} provides more precise information than a simple $L\log L$ bound. Indeed, it gives the decay of the entropy functional $H(f)$, which, in combination with other conservation laws, will play a central role in the proofs of our main results.

For collision kernels satisfying \eqref{assumption:b:1} and initial datum with \eqref{pauli:bound} and \eqref{initial:bound}, it was established in \cite{d94} that the unique solution to \eqref{eq:BFDE} satisfies the $H$-theorem, in the sense that \eqref{thm:H} holds for all $t\in\mathbb{R}^+$. This is noteworthy. Indeed, in general, one only expects that solutions to collisional kinetic equations satisfy a weakened version of the $H$-theorem where the equal sign is replaced by an inequality. In such a case, one sometimes speaks of an entropy inequality. To be precise, in the case of the Boltzmann--Fermi--Dirac equation, the weakened form of \eqref{thm:H} reads
\begin{equation*}
	H(f)(t)+ \int_0^tD(f)(s)ds \leq H(f^\mathrm{in}),
\end{equation*}
for all $t\in\mathbb{R}^+$. For practical purposes, this is often sufficient, since the main use of the $H$-theorem is to provide a uniform control on the decay of the entropy of $f$ and its dissipation.

For renormalized solutions of the classical Boltzmann equation, the entropy inequality was established in \cite{dl91}, but it remains unknown whether or not such solutions satisfy the $H$-theorem exactly. As for the solutions to the Boltzmann--Fermi--Dirac equation constructed in \cite{Lions94}, which exist under the assumption \eqref{assumption:b:2}, the situation is similar and one only expects that they satisfy an entropy inequality.

\subsection{Fermi--Dirac equilibria and ground states}

Equilibrium states of \eqref{eq:BFDE} are defined as particle number densities $f$ for which the collision integral $Q_\mathrm{FD}(f)$ vanishes. This condition precisely translates to solutions of the relation
\begin{equation}\label{equi:relation}
	f'f_*'(1-\delta f)(1-\delta f_*)=ff_*(1-\delta f')(1-\delta f_*'),
\end{equation}
and the $H$-theorem therefore formally suggests that solutions of the Boltzmann--Fermi--Dirac equation converge toward an equilibrium state in long-time asymptotics (due to global time-integrability of the entropy dissipation).

One can show, at least formally, that solutions of \eqref{equi:relation}, with $0<f<\delta^{-1}$, are given by Fermi--Dirac distributions, which take the general form
\begin{equation*}
	M_\mathrm{FD}(v)= \frac{\delta^{-1}}{1+e^{\frac{|v-U|^2-\mu}{T}}},
\end{equation*}
where $U\in\mathbb{R}^d$, $T\in \mathbb{R}^+$ and $\mu\in\mathbb{R}$ are the bulk velocity, temperature and chemical potential, respectively, in suitable physical units. In particular, for simplicity, we assume here that we are working with a choice of units which allows us to set Boltzmann's constant $k_\mathrm{B}$ equal to one.

If $\mu>0$, further observe that the Fermi--Dirac density converges pointwise, as $T\to 0$, toward the ground state
\begin{equation*}
	F(v)=\left\{
	\begin{aligned}
		&0 && \text{if } |v-U|>\mu^\frac 12,
		\\
		&(2\delta)^{-1} && \text{if } |v-U|=\mu^\frac 12,
		\\
		&\delta^{-1} && \text{if } |v-U|<\mu^\frac 12,
	\end{aligned}
	\right.
\end{equation*}
which is also a solution of \eqref{equi:relation}.

A rigorous investigation of equilibrium states of the Boltzmann--Fermi--Dirac equation is available in \cite{d94}.

\subsection{Low-temperature hydrodynamic regimes}\label{section:low_temp_regime}

As previously mentioned, in this work, we look to describe the macroscopic evolution of a Fermi gas in a low-temperature regime near a ground state. This will be achieved by considering the hydrodynamic regime $\varepsilon\to 0$ of the nondimensionalized form of the Boltzmann--Fermi--Dirac equation
\begin{equation}\label{eq:BFDE:scaled}
	(\partial_t+v\cdot\nabla_x)f_\varepsilon=\frac 1{\varepsilon^\kappa}Q_\mathrm{FD}(f_\varepsilon),
\end{equation}
where $\varepsilon^\kappa$ represents the Knudsen number, with $\varepsilon>0$ and $\kappa>0$, for some suitable choice of initial datum $f_\varepsilon|_{t=0}=f_\varepsilon^\mathrm{in}$.

Recall that the Knudsen number is a nondimensional constant which is proportional to the mean-free path in the gas, relative to a fiducial length. Therefore, by considering an asymptotic regime where the Knudsen number vanishes, one expects to reach a continuum limit where the gas effectively becomes a fluid accurately described by a closed system of macroscopic observables. Here, we only point to \cite{SR09} for further details on the Knudsen number and hydrodynamic regimes of the Boltzmann equation in a classical setting.

We are now going to assume that the Fermi gas under consideration, with particle number density $f_\varepsilon(t,x,v)$, remains close to the normalized ground state
\begin{equation*}
	F(v)=\left\{
	\begin{aligned}
		&0 && \text{if } |{v}|>R,
		\\
		&(2\delta)^{-1} && \text{if } |{v}|=R,
		\\
		&\delta^{-1} && \text{if } |{v}|<R,
	\end{aligned}
	\right.
\end{equation*}
where the radius $R>0$ is, in fact, the Fermi velocity $v_\mathrm{F}=R$. More specifically, we consider the first-order density fluctuations $g_\varepsilon^\mathrm{in}(x,v)$ and $g_\varepsilon(t,x,v)$, defined by
\begin{equation}\label{epsilon:fluctuations}
	\begin{aligned}
		f_\varepsilon^\mathrm{in} &=M_\varepsilon+\varepsilon g_\varepsilon^\mathrm{in},
		\\
		f_\varepsilon &=M_\varepsilon+\varepsilon g_\varepsilon,
	\end{aligned}
\end{equation}
around the normalized low-temperature Fermi--Dirac distribution
\begin{equation}\label{normalized:distribution}
	M_\varepsilon(v)= \frac{1}{\delta} \Bigg(1+\exp\bigg(\frac{v^2-R^2}{\varepsilon^\tau}\bigg)\Bigg)^{-1},
\end{equation}
for some fixed parameter $\tau>0$.

Note that the temperature of $M_\varepsilon(v)$ is precisely $\varepsilon^\tau$, which is small if the Knudsen number is small. In particular, the distribution $M_\varepsilon(v)$ converges pointwise to $F(v)$ in the hydrodynamic regime $\varepsilon\to 0$ and is therefore a low-temperature approximation of this ground state. Our main results, below, provide a complete asymptotic characterization of $g_\varepsilon$ in the low-temperature hydrodynamic regime $\varepsilon\to 0$.

In order to quantify how $f_\varepsilon$ departs from the ground state $F$ and its approximate Fermi--Dirac equilibrium $M_\varepsilon$, and ensure that the fluctuation $g_\varepsilon$ is not immeasurably small or large, as $\varepsilon\to 0$, we rely on the relative entropy of $f_\varepsilon$ with respect to $M_\varepsilon$. It is defined by
\begin{equation}\label{relative:entropy:0}
	H(f_\varepsilon|M_\varepsilon)
	=\int_{\mathbb{R}^d\times\mathbb{R}^d}h(\delta f_\varepsilon,\delta M_\varepsilon)dx dv,
\end{equation}
where the nonnegative function $h$ is given by
\begin{equation}\label{entropy:function}
	h(z,a)=z\log\frac za+(1-z)\log \frac{1-z}{1-a},
\end{equation}
for all $(z,a)\in [0,1]\times (0,1)$.

Now, employing the relative entropy $H(f_\varepsilon|M_\varepsilon)$ and the entropy dissipation $D(f_\varepsilon)$ given in \eqref{entropy:dissipation}, we observe that the entropy inequality for the scaled Boltzmann--Fermi--Dirac equation \eqref{eq:BFDE:scaled} can be written as
\begin{equation}\label{entropy:inequality:scaled}
	H(f_\varepsilon|M_\varepsilon)(t)
	+\frac{1}{\varepsilon^\kappa}\int_0^tD(f_\varepsilon)(s)ds \leq
	H(f_\varepsilon^\mathrm{in}|M_\varepsilon),
\end{equation}
for all $t>0$. As previously emphasized, in the context of Dolbeault's wellposedness result, the preceding entropy inequality is, in effect, an equality, but we will not make any use of this fact.

In view of \eqref{entropy:inequality:scaled}, the size of the relative entropy of $f_\varepsilon$ and its dissipation is dominated, in the regime $\varepsilon\to 0$, by the magnitude of the relative entropy of the initial datum. We control these quantities, uniformly in $\varepsilon$, by assuming that
\begin{equation}\label{initial:relative:entropy}
	\frac 1{\varepsilon^{2-\gamma}}H(f_\varepsilon^\mathrm{in}|M_\varepsilon)
	\leq C^\mathrm{in},
\end{equation}
for some constant parameters $\gamma>0$ and $C^\mathrm{in}>0$.

Later on, we identify the numerical range for the parameters $\kappa$, $\tau$ and $\gamma$ for which the hydrodynamic regime $\varepsilon\to 0$ yields a closed limiting system, thereby providing a complete macroscopic description of the asymptotic behavior of electrons near Fermi spheres.

We emphasize that the ratio of the drift velocity's fiducial magnitude to the Fermi velocity in a free electron gas forms a nondimensional parameter akin to the Mach number in acoustics. Moreover, in hydrodynamic regimes where the fluctuation $g_\varepsilon$ remains of finite magnitude, the fluctuation coefficient $\varepsilon$ showing up in \eqref{epsilon:fluctuations} can be interpreted as the Mach number. This follows from the formal asymptotic computation
\begin{equation*}
	\mathrm{(Mach\ number)}\sim
	\frac{\left|\int_{\mathbb{R}^d}f_\varepsilon v dv\right|}{\left(\int_{\mathbb{R}^d}f_\varepsilon v^2 dv\right)^\frac 12}
	=\frac{\varepsilon\left|\int_{\mathbb{R}^d}g_\varepsilon v dv\right|}{\left(\int_{\mathbb{R}^d}f_\varepsilon v^2 dv\right)^\frac 12}
	\sim\frac{\varepsilon\left|\int_{\mathbb{R}^d}g_\varepsilon v dv\right|}{\left(\int_{\mathbb{R}^d}M_\varepsilon v^2 dv\right)^\frac 12}
	\sim\varepsilon.
\end{equation*}
Now, the Fermi velocity in a free electron gas is typically $10^6\ \mathrm{m/s}$ at ambient conditions, a sizeable fraction of the speed of light. In contrast, the electron drift velocity with a typical applied current is much slower, around $10^{-4}\ \mathrm{m/s}$. (We refer to \cite[Chapter 6]{Kittel2004} for further details on these numerical values.) This leads to infinitesimally small Mach numbers, directly justifying the relevance of considering the regime $\varepsilon\to 0$ in \eqref{epsilon:fluctuations}.

Finally, some remarks on the implications of the entropy bounds \eqref{entropy:inequality:scaled} and \eqref{initial:relative:entropy} on the behavior of solutions at infinity are in order. Indeed, a global control on the relative entropy implies that $f_\varepsilon$ and $M_\varepsilon$ must be close to each other at infinity. Thus, since the solution $f_\varepsilon$ is not expected to approach vacuum, it can no longer be globally integrable and the initial bound \eqref{initial:bound} needs to be relaxed into
\begin{equation*}
	(1+|v|^2)f^\mathrm{in}_\varepsilon(x,v)\in L^1_\mathrm{loc}(\mathbb{R}^d;L^1(\mathbb{R}^d)).
\end{equation*}
Accordingly, for a fixed value $\varepsilon>0$, the a priori global bounds \eqref{propagation:initial:bound} and \eqref{collision_operator:quad_control} need to be replaced with
\begin{equation*}
	(1+|v|^2)f_\varepsilon(t,x,v)\in L^\infty(\mathbb{R}^+;L^1_\mathrm{loc}(\mathbb{R}^d;L^1(\mathbb{R}^d)))
\end{equation*}
and
\begin{equation*}
	(1+|v|^2)Q_\mathrm{FD}^\pm (f_\varepsilon)\in L^\infty(\mathbb{R}^+;L^1_\mathrm{loc}(\mathbb{R}^d;L^1(\mathbb{R}^d))),
\end{equation*}
which is sufficient to establish the validity of the local conservation laws of mass \eqref{cons:mass} and momentum \eqref{cons:momentum} in a weak sense.

\section{Historical context and the acoustic limit}\label{acoustic:limit}

Hydrodynamic limits of collisional kinetic equations are a central subject of modern mathematical research, with most results focusing on the rigorous derivation of macroscopic fluid mechanical models from the Boltzmann equation. Pioneering contributions by Bardos, Golse and Levermore in \cite{bgl91,bgl93} established an influential research program. This work notably led to the first fully rigorous derivation of the incompressible Navier--Stokes system from the Boltzmann equation in \cite{gsr04}, subsequently advanced by \cite{lm10} and \cite{a12}.

At the lowest-order hydrodynamic regime, a Boltzmann gas formally approaches the compressible Euler system as the Knudsen number vanishes and it reaches statistical equilibrium. However, a rigorous understanding of the connection between this macroscopic gas dynamics system and the Boltzmann equation remains challenging and incomplete (see \cite{SR09} for further details).

Alternatively, one can characterize the first-order hydrodynamics of particle-density fluctuations around a normalized equilibrium state. Asymptotically, this approach linearizes the compressible Euler system around a constant solution, describing the propagation of sound waves. This hydrodynamic regime is therefore known as the acoustic limit of the Boltzmann equation.

Specifically, the acoustic regime considers fluctuations
\begin{equation*}
	f_\varepsilon=M+\varepsilon g_\varepsilon,
\end{equation*}
around a fixed normalized Maxwellian equilibrium state $M$, of the scaled Boltzmann equation
\begin{equation*}
	(\partial_t+v\cdot\nabla_x)f_\varepsilon=\frac 1{\varepsilon^\kappa}Q_\mathrm{B}(f_\varepsilon),
\end{equation*}
where $Q_\mathrm{B}(f_\varepsilon)$ is the classical Boltzmann collision operator (obtained by setting $\delta=0$ in \eqref{eq:collision_operator}) and $\varepsilon^\kappa$ is the Knudsen number. The rigorous derivation of this acoustic limit was fully achieved in the works \cite{bgl98,bgl00,gl02,jlm10}. These studies demonstrate that, under natural hypotheses and for $\kappa\leq 2$, renormalized solutions of the Boltzmann equation converge to the expected acoustic-wave system as $\varepsilon\to 0$. The case where $\kappa>2$ remains open and exceptionally challenging, as its resolution would constitute a significant step toward a general understanding of the compressible Euler limit of the Boltzmann equation.

In the context of the Boltzmann--Fermi--Dirac equation's hydrodynamic regimes described in Section \ref{section:low_temp_regime}, the classical acoustic limit corresponds to fixing the values of the parameters $\tau$ and $\gamma$ to zero. In this setting, where the global temperature of the gas remains bounded away from its absolute zero, it should be possible to adapt the proofs from \cite{gl02,jlm10} establishing the acoustic limit of the classical Boltzmann equation to its quantum version for fermions, provided $\kappa\leq 2$.

The asymptotic regime investigated in this work notably diverges from the classical acoustic limit. Here, we analyze fluctuations around an equilibrium that approaches a ground state, effectively characterizing an asymptotically degenerate gas whose global temperature reaches its absolute zero in the limit. This is achieved by considering positive values for the parameter $\tau$.

A key feature of our main results, below, is the demonstration of fluctuation convergence to a macroscopic wave system for parameters satisfying
\begin{equation*}
	0<\gamma=\tau< 1
	\qquad\text{and}\qquad
	\kappa>2\tau.
\end{equation*}
Remarkably, the restriction $\kappa\leq 2$, which seems crucial for the classical acoustic limit, is entirely lifted in our setting, meaning that there is no upper bound on the numerical range of $\kappa$. This unique characteristic of low-temperature regimes suggests it may be a promising avenue for progress toward understanding the compressible Euler limit of the Boltzmann--Fermi--Dirac equation.

It is also worth noting that results concerning hydrodynamic limits of the Boltzmann equation for fermions are still sparse. Formal investigations of the hydrodynamic regimes of \eqref{eq:BFDE} were first presented in \cite{z15,z15thesis}. Subsequent rigorous justifications of certain regimes were achieved in \cite{jxz22,jz24jde,jz24cmp,h25} within smooth and local settings near equilibria. However, these prior results do not explore low-temperature regimes and thus do not asymptotically exhibit quantum features of fermions.

\section{Main results}

The results in this section contain our main contributions to the study of degenerate Fermi gases. They only hold in dimensions $d\geq 3$. The two-dimensional case is fundamentally different and is treated in our companion article \cite{aa25}.

\subsection{Macroscopic dynamics of a degenerate Fermi gas}

Our first main theorem fully characterizes the thermodynamic equilibrium and macroscopic dynamics of a degenerate free electron gas.

\begin{thm}\label{thm:main}
	Assume that $d\geq 3$ and take a cross-section such that
	\begin{equation*}
		b(z,\sigma)\in L^1\cap L^\infty(\mathbb{R}^d\times\mathbb{S}^{d-1})\cap C(B(0,2R)\times\mathbb{S}^{d-1})
	\end{equation*}
	and
	\begin{equation*}
		b(z,\sigma)>0\quad\text{on }B(0,2R)\times\mathbb{S}^{d-1}.
	\end{equation*}
	Let
	\begin{equation*}
		f_\varepsilon^\mathrm{in}(x,v)\in L^\infty(\mathbb{R}^d\times\mathbb{R}^d),
		\quad\text{with }0\leq f_\varepsilon^\mathrm{in}\leq\delta^{-1},
	\end{equation*}
	be a family of initial data such that the initial relative entropy bound \eqref{initial:relative:entropy} holds uniformly in $\varepsilon>0$. Consider the unique solution
	\begin{equation*}
		f_\varepsilon(t,x,v)\in L^\infty(\mathbb{R}^+\times\mathbb{R}^d\times\mathbb{R}^d),
		\quad\text{with }0\leq f_\varepsilon\leq\delta^{-1},
	\end{equation*}
	of the Boltzmann--Fermi--Dirac equation \eqref{eq:BFDE:scaled} with initial datum $f_\varepsilon^\mathrm{in}$ and such that the scaled entropy inequality \eqref{entropy:inequality:scaled} holds.
	Further suppose that
	\begin{equation*}
		0<\gamma=\tau< 1
		\qquad\text{and}\qquad
		\kappa>2\tau,
	\end{equation*}
	and consider the density fluctuations $g_\varepsilon^\mathrm{in}(x,v)$ and $g_\varepsilon(t,x,v)$ given in \eqref{epsilon:fluctuations}.
	
	Then, as $\varepsilon\to 0$, up to extraction of a subsequence, the family of fluctuations $g_\varepsilon$ converges in the weak* topology of $\mathcal{M}_\mathrm{loc}(\mathbb{R}^+\times\mathbb{R}^d\times\mathbb{R}^d)$ toward a limit point
	\begin{equation}\label{concentration:characterization}
		\mu(t,x,v)=g\left(t,x,R\frac{v}{|v|}\right)dt\otimes dx\otimes \delta_{\partial B(0,R)}(v),
	\end{equation}
	where the density $g(t,x,\omega)$ belongs to $L^\infty\big(dt;L^2\big(\mathbb{R}^d\times\partial B(0,R)\big)\big)$.
	Similarly, up to further extraction of a subsequence, the family of initial fluctuations $g_\varepsilon^\mathrm{in}$ converges in the weak* topology of $\mathcal{M}_\mathrm{loc}(\mathbb{R}^d\times\mathbb{R}^d)$ toward a limit point
	\begin{equation*}
		\mu^\mathrm{in}(x,v)=g^\mathrm{in}\left(x,R\frac{v}{|v|}\right) dx\otimes \delta_{\partial B(0,R)}(v),
	\end{equation*}
	where $g^\mathrm{in}(x,\omega)\in L^2\big(\mathbb{R}^d\times\partial B(0,R)\big)$.
	
	Moreover, the limiting density $g(t,x,\omega)$ takes the form of an infinitesimal distribution
	\begin{equation}\label{infinitesimal_FD_distribution}
		g(t,x,\omega)=\rho(t,x)+U(t,x)\cdot\omega,
	\end{equation}
	where the density $\rho(t,x)$ and the velocity field $U(t,x)$ belong to $L^\infty\big(dt;L^2\big(\mathbb{R}^d\big)\big)$ and solve the wave system
	\begin{equation}\label{plasma:wave:system}
		\left\{
		\begin{aligned}
			\textstyle
			\partial_t \rho(t,x)+\frac{R^2}{d}
			\nabla_x \cdot U(t,x)&=0,
			\\
			\textstyle
			\partial_t U(t,x)
			+\nabla_x \rho(t,x)&=0,
		\end{aligned}
		\right.
	\end{equation}
	with initial data
	\begin{equation}\label{plasma:wave:system:data}
		\rho|_{t=0}=\rho^\mathrm{in}=\fint_{\partial B(0,R)}g^\mathrm{in}(x,\omega)d\omega
		\quad\text{and}\quad
		U|_{t=0}=U^\mathrm{in}=\frac{d}{R^2}\fint_{\partial B(0,R)}\omega g^\mathrm{in}(x,\omega)d\omega.
	\end{equation}
\end{thm}

\begin{rem}
	The previous statement can be strengthened by exploiting the notion of entropic convergence introduced in \cite{bgl93} and the fact that the limiting wave system \eqref{plasma:wave:system} is well-posed in $L^2$. Specifically, assuming that the initial data converges entropically, in the sense that
	\begin{equation*}
		g_\varepsilon^\mathrm{in}\rightharpoonup^*
		g^\mathrm{in}\left(x,R\frac{v}{|v|}\right) dx\otimes \delta_{\partial B(0,R)}(v),
	\end{equation*}
	with
	\begin{equation*}
		g^\mathrm{in}(x,\omega)=\rho^\mathrm{in}+U^\mathrm{in}\cdot \omega,
	\end{equation*}
	and
	\begin{equation*}
		\frac 1{\varepsilon^{2-\gamma}}H(f_\varepsilon^\mathrm{in}|M_\varepsilon)
		\to
		R\delta^2\int_{\mathbb{R}^d\times\partial B(0,R)}(g^\mathrm{in})^2(x,\omega)dxd\omega,
	\end{equation*}
	one can show that the convergence of fluctuations $g_\varepsilon$ toward $\mu$ holds in a suitable strong topology in $(t,x,\omega)$, without extraction of subsequences. This is analogous to the improvement of weak limit theorems into strong ones performed in \cite{gl02}.
\end{rem}

Before proving the theorem, let us discuss the physical relevance of this result. As previously explained, in a metal, electrons form a quantum mechanical system, a degenerate Fermi gas, even at room temperature. Their internal kinetic energy is not determined by the temperature, but rather by the Pauli exclusion principle, which gives rise to a zero-temperature pressure known as the Fermi pressure. The characteristic speed associated with this quantum pressure is the Fermi velocity $v_\mathrm{F}$, which we denote by $R$ in the above statement, for simplicity and clarity. Only electrons propagating at a speed which is close to $v_\mathrm{F}$ can interact with other electrons and get excited into a different state. This behavior is effectively captured in the above theorem by showing the convergence of fluctuations $g_\varepsilon$ toward a measure concentrated on the Fermi sphere $\{|v|=v_\mathrm{F}\}$.

Theorem \ref{thm:main} goes further into the description of electron dynamics at the Fermi sphere by characterizing the thermodynamic equilibrium of limiting fluctuations as an infinitesimal Fermi--Dirac distribution in \eqref{infinitesimal_FD_distribution}. While the convergence toward an infinitesimal distribution is reminiscent of the classical acoustic limit of the Boltzmann equation, it is remarkable in that it establishes an anisotropic velocity distribution of electrons on a Fermi surface. This is significant, for it departs from the common physical intuition that particles on a spherical Fermi surface should be distributed uniformly since their behavior is dictated by the Fermi--Dirac density, which is isotropic.

Additionally, the wave system \eqref{plasma:wave:system} describes the evolution of the macroscopic density and velocity field of electrons found in \eqref{infinitesimal_FD_distribution}, which establishes the propagation of electron-plasma waves, also known as Langmuir waves following the pioneering work of Tonks and Langmuir in \cite{tl29}. These are longitudinal waves in the electron-charge density that propagate through the plasma at a characteristic speed. Generally, the restoring forces that drive these oscillations are not just due to electrostatic interactions within the plasma (which are neglected in our investigation in order to focus on the quantum mechanical properties of Fermi gases) but also include the pressure of the gas (much like the internal pressure in a fluid is responsible for sound waves, which propagate at a characteristic thermal velocity).

The electron-plasma waves system \eqref{plasma:wave:system} effectively describes the propagation phenomenon of waves due to the Fermi pressure. Specifically, a direct computation readily shows that a plane wave
\begin{equation*}
	\begin{pmatrix}
		\rho \\ U
	\end{pmatrix}(t,x)
	=
	\begin{pmatrix}
		\rho_0 \\ U_0
	\end{pmatrix}
	e^{i(\tau t+\xi\cdot x)},
\end{equation*}
with $\rho_0,\tau\in\mathbb{R}\setminus\{0\}$ and $U_0,\xi\in\mathbb{R}^d$, is a solution of \eqref{plasma:wave:system} if and only if
\begin{equation*}
	\rho_0^2=\frac{R^2}{d}U_0^2,
	\quad
	\tau^2=\frac{R^2}{d}\xi^2
	\quad\text{and}\quad
	\tau U_0 
	+\rho_0\xi =0.
\end{equation*}
This establishes that electron-plasma waves propagate with a characteristic speed of
\begin{equation*}
	\frac{R}{\sqrt d}=\frac{v_\mathrm{F}}{\sqrt d},
\end{equation*}
which is a constant multiple of the Fermi velocity. This is consistent with the Bohm--Gross dispersion relation for longitudinal plasma waves in a degenerate electron gas (see \cite{bg49,Shukla:2010,Zhu_2012}), which predicts that in a small-wavelength regime, where electrostatic forces are negligible, the phase and group velocities of plasma waves are directly proportional to the Fermi velocity.

We move on now to the justification of our first main theorem.

\begin{proof}[Proof of Theorem \ref{thm:main}]
	The proof is lengthy and intricate. To ensure clarity, its various components are detailed in subsequent sections.
	
	Thus, the extraction of uniform bounds from the relative entropy inequality and ensuing weak compactness estimates are the subject of Section \ref{section:bounds_and_compactness}. In particular, the convergence of fluctuations $g_\varepsilon$ and initial data $g_\varepsilon^\mathrm{in}$, up to extraction of subsequences, toward measures concentrated on a spherical Fermi surface with densities $g(t,x,\omega)$ and $g^\mathrm{in}(x,\omega)$, respectively, is established in Proposition \ref{prop:limit:characterization}, where it is also shown that, for almost every $t\geq 0$,
	\begin{equation*}
		\int_{\mathbb{R}^d\times\partial B(0,R)}g^2(t,x,\omega)dxd\omega \leq
		\frac{C^{\mathrm{in}}}{R\delta^2}
		\quad\text{and}\quad
		\int_{\mathbb{R}^d\times\partial B(0,R)}(g^\mathrm{in})^2(x,\omega)dxd\omega \leq
		\frac{C^{\mathrm{in}}}{R\delta^2},
	\end{equation*}
	with the constant $C^{\mathrm{in}}$ given in \eqref{initial:relative:entropy}.
	It is further established therein that
	\begin{equation*}
		\int_{\mathbb{R}^+\times\mathbb{R}^d\times\mathbb{R}^d}
		g_\varepsilon(t,x,v)\varphi(t,x,v)dtdxdv
		\to
		\int_{\mathbb{R}^+\times\mathbb{R}^d\times\partial B(0,R)}
		g(t,x,\omega)\varphi(t,x,\omega)dtdxd\omega,
	\end{equation*}
	and
	\begin{equation*}
		\int_{\mathbb{R}^d\times\mathbb{R}^d}
		g_\varepsilon^\mathrm{in}(x,v)\varphi(0,x,v)dxdv
		\to
		\int_{\mathbb{R}^d\times\partial B(0,R)}
		g^\mathrm{in}(x,\omega)\varphi(0,x,\omega)dxd\omega,
	\end{equation*}
	as $\varepsilon\to 0$, for any continuous function $\varphi(t,x,v)$ which is compactly supported in $(t,x)$ and satisfies the quadratic growth condition
	\begin{equation*}
		\frac{\varphi(t,x,v)}{1+v^2}\in L^\infty\big(\mathbb{R}^+\times\mathbb{R}^d\times\mathbb{R}^d\big).
	\end{equation*}
	The weak convergence of fluctuations against unbounded test functions with quadratic growth is essential in the derivation of the wave system \eqref{plasma:wave:system}, below.
	
	Then, Section \ref{section:relaxation_and_equilibria} focuses on the consequences of the entropy dissipation bound and the relaxation process toward thermodynamic equilibria. In particular, the fundamental characterization of limiting densities given in \eqref{infinitesimal_FD_distribution} is found in Proposition \ref{prop:instrumental}.
	
	Now, by gathering the preceding properties, we derive the limiting wave system \eqref{plasma:wave:system}. To that end, we start from the macroscopic conservation laws
	\begin{equation*}
		\begin{aligned}
			\partial_t \int_{\mathbb{R}^d}g_\varepsilon(t,x,v) dv+\nabla_x \cdot\int_{\mathbb{R}^d}vg_\varepsilon(t,x,v)dv&=0,
			\\
			\partial_t \int_{\mathbb{R}^d}vg_\varepsilon(t,x,v) dv+\nabla_x \cdot\int_{\mathbb{R}^d}(v\otimes v)g_\varepsilon(t,x,v)dv&=0,
		\end{aligned}
	\end{equation*}
	which can be written in weak form as
	\begin{equation*}
		\begin{aligned}
			\int_{\mathbb{R}^d\times\mathbb{R}^d}g_\varepsilon^\mathrm{in}(x,v) \varphi(0,x)dxdv
			+\int_{\mathbb{R}^+\times\mathbb{R}^d\times\mathbb{R}^d}g_\varepsilon(t,x,v) (\partial_t+v\cdot \nabla_x) \varphi(t,x) dtdxdv
			&=0,
			\\
			\int_{\mathbb{R}^d\times\mathbb{R}^d}vg_\varepsilon^\mathrm{in}(x,v) \varphi(0,x)dxdv
			+\int_{\mathbb{R}^+\times\mathbb{R}^d\times\mathbb{R}^d}vg_\varepsilon(t,x,v) (\partial_t+v\cdot \nabla_x) \varphi(t,x) dtdxdv
			&=0,
		\end{aligned}
	\end{equation*}
	for any test function $\varphi(t,x)\in C_c^1(\mathbb{R}^+\times\mathbb{R}^d)$.
	
	We can then pass to the limit and deduce that
	\begin{equation*}
		\begin{aligned}
			\int_{\mathbb{R}^d\times\partial B(0,R)}g^\mathrm{in}(x,\omega) \varphi(0,x)dxd\omega
			+\int_{\mathbb{R}^+\times\mathbb{R}^d\times\partial B(0,R)}g(t,x,\omega) (\partial_t+\omega\cdot \nabla_x) \varphi(t,x) dtdxd\omega
			&=0,
			\\
			\int_{\mathbb{R}^d\times\partial B(0,R)}\omega g^\mathrm{in}(x,\omega) \varphi(0,x)dxd\omega
			+\int_{\mathbb{R}^+\times\mathbb{R}^d\times\partial B(0,R)}\omega g(t,x,\omega) (\partial_t+\omega\cdot \nabla_x) \varphi(t,x) dtdxd\omega
			&=0.
		\end{aligned}
	\end{equation*}
	Finally, a straightforward computation yields that
	\begin{equation*}
		\begin{aligned}
			\int_{\mathbb{R}^d}\left(\fint_{\partial B(0,R)} g^\mathrm{in}(x,\omega)d\omega\right) \varphi(0,x)dx
			+\int_{\mathbb{R}^+\times\mathbb{R}^d}\rho(t,x) \partial_t \varphi(t,x) dtdx&
			\\
			+\frac{R^2}{d}\int_{\mathbb{R}^+\times\mathbb{R}^d}U(t,x) \cdot \nabla_x \varphi(t,x) dtdx
			&=0,
			\\
			\int_{\mathbb{R}^d}\left(\fint_{\partial B(0,R)} \omega g^\mathrm{in}(x,\omega)d\omega\right) \varphi(0,x)dx
			+\frac{R^2}{d}\int_{\mathbb{R}^+\times\mathbb{R}^d}U(t,x) \partial_t \varphi(t,x) dtdx&
			\\
			+\frac{R^2}{d}\int_{\mathbb{R}^+\times\mathbb{R}^d}\rho(t,x) \nabla_x \varphi(t,x) dtdx
			&=0,
		\end{aligned}
	\end{equation*}
	which is precisely the weak form of the wave system \eqref{plasma:wave:system} with the initial datum given in \eqref{plasma:wave:system:data}. The proof of the theorem is thereby complete.
\end{proof}

\subsection{Energy spectrum of a degenerate Fermi gas}

The previous result characterized the asymptotic fluctuations $g_\varepsilon(t,x,v)$ within a specific hydrodynamic regime. We are now able to further describe the energy spectrum of a degenerate free electron gas, or Fermi gas, in the same asymptotic regime near the Fermi sphere. This is the focus of the next theorem.

\begin{thm}\label{thm:main:2}
	Under the hypotheses of Theorem \ref{thm:main}, we assume that there is a convergent subsequence of fluctuations $g_\varepsilon(t,x,v)$ such that \eqref{concentration:characterization}, \eqref{infinitesimal_FD_distribution}, \eqref{plasma:wave:system} and \eqref{plasma:wave:system:data} hold for some limit points $\mu(t,x,v)$ and $g(t,x,\omega)$. Further suppose that
	\begin{equation*}
		0<\gamma=\tau< \frac 23
		\qquad\text{and}\qquad
		2\tau<\kappa<2-\tau.
	\end{equation*}
	
	Then, as $\varepsilon\to 0$, up to further extraction of a subsequence, for any continuous test function $\varphi(t,x,u,v)$, compactly supported in $(t,x)$, with the growth condition
	\begin{equation}\label{spectrum:test:function:growth}
		\frac{\varphi(t,x,u,v)}{1+|u|}\in L^\infty\big(\mathbb{R}^+\times\mathbb{R}^d\times\mathbb{R}\times\mathbb{R}^d\big),
	\end{equation}
	we have the limits
	\begin{equation}\label{spectrum:convergence}
		\begin{aligned}
			\int_{\mathbb{R}^+\times\mathbb{R}^d\times\mathbb{R}^d}
			g_\varepsilon(t,x,v)\varphi & \left(t,x,\frac{v^2-R^2}{\varepsilon^\tau},v\right)dtdxdv
			\\
			&\to
			\int_{\mathbb{R}^+\times\mathbb{R}^d\times\mathbb{R}\times\partial B(0,R)}
			\widetilde g(t,x,u,\omega)\varphi(t,x,u,\omega)dtdxdud\omega,
		\end{aligned}
	\end{equation}
	and
	\begin{equation}\label{spectrum:convergence:initial}
		\begin{aligned}
			\int_{\mathbb{R}^d\times\mathbb{R}^d}
			g_\varepsilon^\mathrm{in}(x,v)\varphi & \left(0,x,\frac{v^2-R^2}{\varepsilon^\tau},v\right)dxdv
			\\
			&\to
			\int_{\mathbb{R}^d\times\mathbb{R}\times\partial B(0,R)}
			\widetilde g^\mathrm{in}(x,u,\omega)\varphi(0,x,u,\omega)dxdud\omega,
		\end{aligned}
	\end{equation}
	where the limit points satisfy that
	\begin{equation*}
		\widetilde g(t,x,u,\omega)\in L^\infty\left(dt;L^2\left(\cosh^2\left(\frac u2\right)dxdud\omega\right)\right)
	\end{equation*}
	and
	\begin{equation*}
		\widetilde g^\mathrm{in}(x,u,\omega)\in L^2\left(\cosh^2\left(\frac u2\right)dxdud\omega\right).
	\end{equation*}
	
	Moreover, the initial limiting density satisfies that
	\begin{equation*}
		\int_{\mathbb{R}}\widetilde g^\mathrm{in}(x,u,\omega)du=g^\mathrm{in}(x,\omega),
	\end{equation*}
	while the limiting density $\widetilde g(t,x,u,\omega)$ takes the form of an infinitesimal distribution
	\begin{equation}\label{limit:fluctuation:dilated}
		\widetilde g(t,x,u,\omega)=\frac{\rho(t,x)+U(t,x)\cdot\omega+E(x)u}{4\cosh^2(\frac u2)},
	\end{equation}
	where the density $\rho(t,x)$ and the velocity field $U(t,x)$ belong to $L^\infty\big(dt;L^2\big(\mathbb{R}^d\big)\big)$ and coincide with the coefficients from \eqref{infinitesimal_FD_distribution}. In particular, they solve the wave system \eqref{plasma:wave:system}, with initial data \eqref{plasma:wave:system:data}. Finally, the energy density $E(x)$ is independent of time, it belongs to $L^2\big(\mathbb{R}^d\big)$ and it is given by
	\begin{equation*}
		\frac{\pi^2}{3}R^{d-1}|\mathbb{S}^{d-1}|E(x)
		=\int_{\mathbb{R}\times\partial B(0,R)}u\widetilde g^\mathrm{in}(x,u,\omega)dud\omega.
	\end{equation*}
\end{thm}

\begin{rem}
	Observe that Theorem \ref{thm:main:2} imposes the additional parameter restriction $\kappa+\tau<2$, which is absent from Theorem \ref{thm:main}. This is a consequence of the need to control conservation defects in the proof below. It is interesting to note that the formal regime $\tau=0$ leads to the condition $\kappa<2$. This restriction then matches the parameter range for the classical acoustic limit of the Boltzmann equation, as discussed in Section \ref{acoustic:limit}. At this stage, it remains unclear whether the $\kappa+\tau<2$ restriction can be improved.
\end{rem}

In order to adequately interpret the preceding result, we introduce the dilated fluctuations
\begin{equation}\label{dilated:fluctuations}
	\widetilde g_\varepsilon (t,x,u,\omega)=
	\left\{
	\begin{aligned}
		&\frac{\varepsilon^\tau|v|^{d-2}}{2R^{d-1}}
		g_\varepsilon(t,x,v) &&\text{if }u\geq -\frac{R^2}{\varepsilon^\tau},
		\\
		&0 &&\text{if }u< -\frac{R^2}{\varepsilon^\tau},
	\end{aligned}
	\right.
\end{equation}
for all $(t,x,u,\omega)\in\mathbb{R}^+\times\mathbb{R}^d\times\mathbb{R}\times\partial B(0,R)$, and similarly for the initial data, where
\begin{equation}\label{spectrum:variables}
	v=\left(1+\frac{\varepsilon^\tau}{R^2}u\right)^\frac 12 \omega,
\end{equation}
or, equivalently,
\begin{equation*}
	u=\frac{v^2-R^2}{\varepsilon^\tau}, \qquad \omega=R\frac{v}{|v|}.
\end{equation*}
Thus, the variable $u$ represents energy levels of electrons: values less than zero $\{u<0\}$ are inside the Fermi sphere, a value zero $\{u=0\}$ is on the surface, and values greater than zero $\{u>0\}$ are outside it.

Observe that the preceding change of variables acts on infinitesimal volumes according to the identity
\begin{equation*}
	dv=\frac{\varepsilon^\tau}{2R}\left(1+\frac{\varepsilon^\tau}{R^2}u\right)^\frac{d-2}{2}dud\omega,
\end{equation*}
or, equivalently,
\begin{equation}\label{dilation:change:variable}
	\frac{2R^{d-1}}{\varepsilon^\tau |v|^{d-2}}dv=dud\omega.
\end{equation}
In particular, it holds that
\begin{equation*}
	\int_{\mathbb{R}^d}g_\varepsilon(t,x,v)\psi(v)dv=\int_{\mathbb{R}\times\partial B(0,R)}\widetilde g_\varepsilon(t,x,u,\omega)\psi(v)dud\omega,
\end{equation*}
for any suitable test function $\psi$.

It is judicious to also introduce the dilations of $f_\varepsilon(t,x,v)$ and $M_\varepsilon(v)$ by
\begin{equation*}
	\widetilde f_\varepsilon (t,x,u,\omega)=
	\left\{
	\begin{aligned}
		&\frac{|v|^{d-2}}{2R^{d-1}}
		f_\varepsilon(t,x,v) &&\text{if }u\geq -\frac{R^2}{\varepsilon^\tau},
		\\
		&0 &&\text{if }u< -\frac{R^2}{\varepsilon^\tau},
	\end{aligned}
	\right.
\end{equation*}
and
\begin{equation*}
	\begin{aligned}
		\widetilde M_\varepsilon (u)&=
		\left\{
		\begin{aligned}
			&\frac{|v|^{d-2}}{2R^{d-1}}M_\varepsilon(v) &&\text{if }u\geq -\frac{R^2}{\varepsilon^\tau},
			\\
			&0 &&\text{if }u< -\frac{R^2}{\varepsilon^\tau},
		\end{aligned}
		\right.
		\\
		&=
		\left\{
		\begin{aligned}
			&\frac{|v|^{d-2}}{2R^{d-1}\delta(1+e^u)} &&\text{if }u\geq -\frac{R^2}{\varepsilon^\tau},
			\\
			&0 &&\text{if }u< -\frac{R^2}{\varepsilon^\tau}.
		\end{aligned}
		\right.
	\end{aligned}
\end{equation*}
It then holds that
\begin{equation*}
	\widetilde f_\varepsilon=\widetilde M_\varepsilon+\varepsilon^{1-\tau} \widetilde g_\varepsilon.
\end{equation*}
This readily shows that, at the lowest order, the dilated densities $\widetilde f_\varepsilon$ converge in a suitable sense toward the energy distribution
\begin{equation*}
	\widetilde M(u)=\frac{1}{2R\delta(1+e^u)},
\end{equation*}
with $u\in\mathbb{R}$. At the next order, by virtue of the above theorem, the asymptotic energy distribution in the Fermi gas is determined by the limiting dilated fluctuation $\widetilde g$ given in \eqref{limit:fluctuation:dilated}. In other words, in the hydrodynamic regime $\varepsilon\to 0$, the dilated particle number densities $\widetilde f_\varepsilon$ are well-approximated, at least in a weak sense, by the distributions
\begin{equation*}
	\frac{1}{2R\delta(1+e^u)}+\varepsilon^{1-\tau}\widetilde g.
\end{equation*}
This strikingly demonstrates an anisotropic asymptotic energy distribution near the Fermi surface.

Another noteworthy feature of Theorem \ref{thm:main:2} resides in the asymptotic emergence of a persistent energy structure. This is readily seen in \eqref{limit:fluctuation:dilated}, where the energy density coefficient $E(x)$ is independent of time and fully determined by the initial datum. In particular, this shows that any perturbation of the energy distribution of a degenerate Fermi gas persists forever or, at least, on a suitable time-scale compatible with the hydrodynamic regime under consideration.

We can now proceed to the proof of our second main theorem.

\begin{proof}[Proof of Theorem \ref{thm:main:2}]
	To ensure a smooth and accessible reading, we provide only an outline of the main steps of the convergence proof here. The more technical and complex details of our arguments are deferred to later sections.
	
	We build now upon the proof of Theorem \ref{thm:main}. In particular, we begin by assuming that we have extracted convergent subsequences of fluctuations $g_\varepsilon$ and $g_\varepsilon^\mathrm{in}$, which satisfy all properties listed in the statement of Theorem \ref{thm:main}.
	
	Then, we consider the dilated fluctuations $\widetilde g_\varepsilon$ and $\widetilde g_\varepsilon^\mathrm{in}$. Their uniform boundedness and weak relative compactness in
	\begin{equation*}
		L^1_\mathrm{loc}(dtdx;L^1((1+|u|)dud\omega))
		\quad\text{and}\quad
		L^1_\mathrm{loc}(dx;L^1((1+|u|)dud\omega)),
	\end{equation*}
	respectively, is established in Proposition \ref{prop:limit:characterization:dilated}. It is also shown therein, up to another extraction of subsequences, that we may assume their weak convergence toward respective limit points
	\begin{equation*}
		\widetilde g(t,x,u,\omega)\in L^\infty\left(dt;L^2\left(\cosh^2\left(\frac u2\right)dxdud\omega\right)\right)
	\end{equation*}
	and
	\begin{equation*}
		\widetilde g^\mathrm{in}(x,u,\omega)\in L^2\left(\cosh^2\left(\frac u2\right)dxdud\omega\right),
	\end{equation*}
	such that
	\begin{equation}\label{moment:relation}
		\int_{\mathbb{R}}\widetilde g(t,x,u,\omega)du=g(t,x,\omega)
		\qquad\text{and}\qquad
		\int_{\mathbb{R}}\widetilde g^\mathrm{in}(x,u,\omega)du=g^\mathrm{in}(x,\omega).
	\end{equation}
	Finally, considering a continuous test function $\varphi(t,x,u,v)$, compactly supported in $t$ and $x$, with the growth condition \eqref{spectrum:test:function:growth}, it is further demonstrated in Proposition \ref{prop:limit:characterization:dilated} that
	\begin{equation*}
		\begin{aligned}
			\int_{\mathbb{R}^+\times\mathbb{R}^d\times\mathbb{R}\times\partial B(0,R)}
			&\widetilde g_\varepsilon(t,x,u,\omega)\varphi(t,x,u,v)dtdxdud\omega
			\\
			&\to
			\int_{\mathbb{R}^+\times\mathbb{R}^d\times\mathbb{R}\times\partial B(0,R)}
			\widetilde g(t,x,u,\omega)\varphi(t,x,u,\omega)dtdxdud\omega,
		\end{aligned}
	\end{equation*}
	and
	\begin{equation*}
		\int_{\mathbb{R}^d\times\mathbb{R}\times\partial B(0,R)}
		\widetilde g_\varepsilon^\mathrm{in}(x,u,\omega)\varphi(0,x,u,v)dxdud\omega
		\to
		\int_{\mathbb{R}^d\times\mathbb{R}\times\partial B(0,R)}
		\widetilde g^\mathrm{in}(x,u,\omega)\varphi(0,x,u,\omega)dxdud\omega,
	\end{equation*}
	as $\varepsilon\to 0$, which establishes the convergences \eqref{spectrum:convergence} and \eqref{spectrum:convergence:initial}.
	
	Next, establishing the thermodynamic equilibrium of $\widetilde g$ is subtle and intricate, and relies specifically on the control of the entropy dissipation bound. This characterization is given in Proposition \ref{prop:instrumental:dilated}, which shows that the equilibrium state is
	\begin{equation}\label{limiting:infinitesimal:distribution:energy}
		\widetilde g(t,x,u,\omega)=\frac{\rho(t,x)+U(t,x)\cdot\omega+E(t,x)u}{4\cosh^2(\frac u2)}.
	\end{equation}
	At this stage, in view of \eqref{moment:relation}, it is clear that the density $\rho(t,x)$ and the velocity field $U(t,x)$ match the coefficients found in the statement of Theorem \ref{thm:main}. In particular, they solve the wave system \eqref{plasma:wave:system}. Therefore, there only remains to show that the energy density $E(t,x)\in L^\infty (dt;L^2(\mathbb{R}^d))$ is independent of time.
	
	To that end, we perform an asymptotic analysis of an approximate energy conservation law for the scaled Boltzmann--Fermi--Dirac equation. Specifically, we start from the equation
	\begin{equation*}
		\partial_t \int_{\mathbb{R}^d}\frac{v^2-R^2}{\varepsilon^\tau}g_\varepsilon \chi dv+\nabla_x \cdot\int_{\mathbb{R}^d}\frac{v^2-R^2}{\varepsilon^\tau}vg_\varepsilon \chi dv=
		\frac{1}{\varepsilon^{\kappa+1}}\int_{\mathbb{R}^d} \frac{v^2-R^2}{\varepsilon^\tau}Q_\mathrm{FD}(f_\varepsilon)\chi dv,
	\end{equation*}
	where
	\begin{equation}\label{cutoff:function:0}
		\chi(v)=\mathds{1}_{\left\{\frac{|v^2-R^2|}{\varepsilon^\tau}< -\log\varepsilon^\alpha\right\}},
	\end{equation}
	for some fixed parameter $\alpha>0$, to be determined later. In its weak form, this equation becomes
	\begin{equation}\label{weak:form:energy:law}
		\begin{aligned}
			\int_{\mathbb{R}^d\times\mathbb{R}\times\partial B(0,R)} &
			u\widetilde g_\varepsilon^\mathrm{in}(x,u,\omega)\chi(v) \varphi(0,x)dxdud\omega
			\\
			&+\int_{\mathbb{R}^+\times\mathbb{R}^d\times\mathbb{R}\times\partial B(0,R)}
			u\widetilde g_\varepsilon(t,x,u,\omega)\chi(v) \partial_t \varphi(t,x) dtdxdud\omega
			\\
			&+\int_{\mathbb{R}^+\times\mathbb{R}^d}F_\varepsilon(t,x)\cdot \nabla_x \varphi(t,x) dtdx
			=-\int_{\mathbb{R}^+\times\mathbb{R}^d}D_\varepsilon(t,x) \varphi(t,x) dtdx,
		\end{aligned}
	\end{equation}
	for any test function $\varphi(t,x)\in C_c^1(\mathbb{R}^+\times\mathbb{R}^d)$, where we have introduced the fluxes
	\begin{equation*}
		F_\varepsilon(t,x)=
		\int_{\mathbb{R}^d}\frac{v^2-R^2}{\varepsilon^\tau}v g_\varepsilon(t,x,v)\chi(v) dv
		=
		\int_{\mathbb{R}\times\partial B(0,R)}uv \widetilde g_\varepsilon(t,x,u,\omega)\chi(v) dud\omega
	\end{equation*}
	and the conservation defects
	\begin{equation*}
		D_\varepsilon(t,x)=
		\frac{1}{\varepsilon^{\kappa+1}}\int_{\mathbb{R}^d} \frac{v^2-R^2}{\varepsilon^\tau}Q_\mathrm{FD}(f_\varepsilon)(t,x,v)\chi(v)dv.
	\end{equation*}
	
	It is reasonable to expect formally that $D_\varepsilon$ vanish as $\varepsilon\to 0$, because the cutoff $\chi$ converges to a constant function in the variable $u$. However, showing this with full rigor is nontrivial. Moreover, it requires that $\alpha$ be large enough, and that $\kappa+\tau<2$. Therefore, for clarity, we defer the removal of conservation defects to Section \ref{removal:defects}, where it is established in Lemma \ref{lemma:removal} that
	\begin{equation*}
		D_\varepsilon \to 0,
	\end{equation*}
	in the strong topology of $L^1_\mathrm{loc}(dtdx)$.
	
	Now, we claim that
	\begin{equation*}
		F_\varepsilon \rightharpoonup 0,
	\end{equation*}
	in the weak topology of $L^1_\mathrm{loc}(dtdx)$. Momentarily assuming the validity of this assertion, we can readily pass to the limit in \eqref{weak:form:energy:law} to infer that
	\begin{equation*}
		\int_{\mathbb{R}\times\partial B(0,R)}
		u\widetilde g^\mathrm{in}(x,u,\omega)dud\omega
		=\int_{\mathbb{R}\times\partial B(0,R)}
		u\widetilde g(t,x,u,\omega) dud\omega.
	\end{equation*}
	Then, since $\widetilde g$ reaches the infinitesimal equilibrium described in \eqref{limiting:infinitesimal:distribution:energy}, we conclude, by virtue of the identity
	\begin{equation*}
		\int_{\mathbb{R}}\frac{u^2}{4\cosh^2\left(\frac u2\right)}du=\frac{\pi^2}3,
	\end{equation*}
	that
	\begin{equation*}
		\int_{\mathbb{R}\times\partial B(0,R)}
		u\widetilde g^\mathrm{in}(x,u,\omega)dud\omega
		=\frac{\pi^2}{3}R^{d-1}|\mathbb{S}^{d-1}|E(t,x),
	\end{equation*}
	for all $t>0$, which establishes that the energy density $E$ is time-independent.
	
	We are now going to demonstrate that the fluxes vanish. To that end, we first decompose $F_\varepsilon$ into the sum of
	\begin{equation*}
		F_\varepsilon^1=
		\int_{\mathbb{R}\times\partial B(0,R)} u \omega \widetilde g_\varepsilon \chi dud\omega
	\end{equation*}
	and
	\begin{equation*}
		F_\varepsilon^2=
		\int_{\mathbb{R}\times\partial B(0,R)} u\left(\left(1+\frac{\varepsilon^\tau}{R^2}u\right)^\frac 12-1\right) \omega \widetilde g_\varepsilon \chi dud\omega.
	\end{equation*}
	It is then readily seen, employing \eqref{limiting:infinitesimal:distribution:energy}, that
	\begin{equation*}
		F_\varepsilon^1\rightharpoonup
		\int_{\mathbb{R}\times\partial B(0,R)} u \omega \widetilde g dud\omega=0,
	\end{equation*}
	in the weak topology of $L^1_\mathrm{loc}(dtdx)$, whereas
	\begin{equation*}
		\begin{aligned}
			|F_\varepsilon^2|&\leq
			\int_{\mathbb{R}\times\partial B(0,R)} |u|\left|\left(1+\frac{\varepsilon^\tau}{R^2}u\right)^\frac 12-1\right|
			|\widetilde g_\varepsilon| \chi dud\omega
			\\
			&\leq
			\frac{\varepsilon^\tau}{R^2}
			\int_{\mathbb{R}\times\partial B(0,R)} u^2|\widetilde g_\varepsilon| \chi dud\omega=O(\varepsilon^\tau|\log\varepsilon|)_{L^1_\mathrm{loc}(dtdx)}.
		\end{aligned}
	\end{equation*}
	All in all, we have shown that $F_\varepsilon$ tends to zero, as $\varepsilon\to 0$, in the weak topology of locally integrable functions. This completes the proof of the theorem.
\end{proof}

\section{Uniform bounds and weak compactness}\label{section:bounds_and_compactness}

This section is devoted to the extraction of quantitative bounds from the relative entropy inequality and their compactness consequences. The starting point of our analysis is the relative entropy inequality, which not only ensures the nonlinear stability of the kinetic dynamics with respect to the limiting fluid behavior, but also provides the quantitative control required to rigorously analyze the small-$\varepsilon$ limit of the Boltzmann--Fermi--Dirac equation.

More specifically, we derive here precise bounds on the fluctuations $g_\varepsilon$ which capture how the fluctuations scale with respect to the parameters $\gamma$ and $\tau$, and in particular show that $g_\varepsilon$ is increasingly concentrated near the Fermi sphere $\{|v|=R\}$ as $\varepsilon \to 0$.

\subsection{Controls from the relative entropy bound}

Our goal now consists in extracting a uniform control of fluctuations from the relative entropy bound
\begin{equation}\label{relative:entropy:1}
	\frac 1{\varepsilon^{2-\gamma}}H(f_\varepsilon|M_\varepsilon)(t)
	\leq C^\mathrm{in},
\end{equation}
for some $\gamma>0$, $C^\mathrm{in}>0$ and all $t\geq 0$,

Observe that the relative entropy $H(f_\varepsilon|F)$, obtained by replacing $M_\varepsilon$ by $F$ in \eqref{relative:entropy:0}, is not well defined, which is why we rely on the approximation $M_\varepsilon$ of the ground state $F$ to measure the entropy of the densities $f_\varepsilon$ at low temperatures.

We begin by establishing a collection of uniform estimates that follow directly from the relative entropy bound. These will serve as the foundation for the compactness and asymptotic analysis carried out in the critical regime.

\begin{prop}\label{prop:relative:entropy}
	Consider a family of density distributions $0\leq f_\varepsilon(t,x,v)\leq\delta^{-1}$, with $\varepsilon>0$, such that the relative entropy bound \eqref{relative:entropy:1} holds uniformly in $\varepsilon$, with some fixed parameter $\gamma>0$, where the Fermi--Dirac distribution $M_\varepsilon$ defined in \eqref{normalized:distribution} has a given parameter value $\tau>0$.
	Further consider the density fluctuations $g_\varepsilon(t,x,v)$ defined in \eqref{epsilon:fluctuations}.
	
	Then, as $\varepsilon\to 0$:
	\begin{itemize}
		
		\item
		If $\gamma+\tau\leq 2$, then any subsequence of fluctuations $g_\varepsilon$ satisfies the bounds
		\begin{equation}\label{entropy:bound:1}
			g_\varepsilon
			=
			O(\varepsilon^{\frac{\tau-\gamma}2})_{L^\infty(dt;L^1_\mathrm{loc}(dx;L^1(dv)))}
		\end{equation}
		and
		\begin{equation}\label{entropy:bound:2}
			g_\varepsilon|v^2-R^2|
			=
			O(\varepsilon^{\frac{3\tau-\gamma}2})_{L^\infty(dt;L^1_\mathrm{loc}(dx;L^1(dv)))}.
		\end{equation}
		
		\item
		If $\gamma+\tau< 2$, then any subsequence of fluctuations $g_\varepsilon$ satisfies a concentrated tightness property around the sphere $\{|v|=R\}$, in the sense that, for any compact subset $K\subset\mathbb{R}^d$, it holds that
		\begin{equation}\label{concentrated:tightness}
			\lim_{\lambda\to\infty}\sup_{\varepsilon>0,t>0}\int_{K\times\mathbb{R}^d}\varepsilon^{\frac{\gamma-3\tau}2}\big|g_\varepsilon(v^2-R^2)\big|
			\mathds{1}_{\left\{|v^2-R^2|\geq \lambda \varepsilon^\tau\right\}}dxdv
			=0.
		\end{equation}
		
		\item
		For any choice of parameters $\gamma>0$ and $\tau>0$, any subsequence of fluctuations $g_\varepsilon$ satisfies the bound
		\begin{equation}\label{fluctuation:H:2}
			\sup_{t>0}\int_{\mathbb{R}^d\times\mathbb{R}^d}\left(g_\varepsilon\right)^2
			\max\left\{2,\varepsilon^{-\tau}\left|v^2-R^2\right|\right\}dx dv
			\lesssim \varepsilon^{-\gamma}.
		\end{equation}
	\end{itemize}
\end{prop}

The proof of the above proposition hinges upon appropriate applications of a convexity inequality established in the next lemma, which is a particular case of a generalized Young inequality. Thus, for the sake of clarity, we defer the proof of Proposition \ref{prop:relative:entropy} until after the proof of Lemma \ref{lemma:inequality:Young}, below.

The idea of using Young inequalities to extract suitable controls of fluctuations from relative entropy bounds goes back to the seminal work \cite{bgl93} on hydrodynamic limits (see Section 3 therein). We also refer to Appendix B of \cite{asr19} for a brief introduction to Young inequalities related to hydrodynamic limits.

\begin{lem}\label{lemma:inequality:Young}
	For all $a\in (0,1)$, $z\in [-a,1-a]$, $y\in\mathbb{R}$ and $\alpha>0$, with $\alpha\varepsilon^{2-\gamma}\leq 1$, it holds that
	\begin{equation}\label{inequality:Young:0}
		|zy|
		\leq
		\frac{1}{\alpha\varepsilon^{2-\gamma}}h(a+z,a)
		+a(1-a)\alpha \varepsilon^{2-\gamma}w(y),
	\end{equation}
	where we have denoted $w(y)=e^y+e^{-y}-2$.
\end{lem}

\begin{proof}[Proof of Lemma \ref{lemma:inequality:Young}]
	For any fixed $0<a<1$, we introduce the function
	\begin{equation*}
		h_1(z)=h(a+z,a),
	\end{equation*}
	which is defined on the domain $z\in [-a,1-a]$, and its convex conjugate (also called the Legendre transform)
	\begin{equation*}
		h_1^*(y)=\sup_{z\in [-a,1-a]}\big(zy-h_1(z)\big),
	\end{equation*}
	which leads to the generalized Young inequality
	\begin{equation*}
		zy\leq h_1(z)+h_1^*(y).
	\end{equation*}
	
	More explicitly, a straightforward optimization procedure applied to the function
	\begin{equation*}
		z\mapsto zy-h_1(z),
	\end{equation*}
	for any fixed $y\in\mathbb{R}$, readily shows that its global maximum is attained at the critical point
	\begin{equation*}
		z_*=\frac{a(1-a)(e^y-1)}{1+a(e^y-1)},
	\end{equation*}
	with the critical value
	\begin{equation*}
		h_1^*(y)=z_*y-h_1(z_*)=\log\big(1+a(e^y-1)\big)-ay.
	\end{equation*}
	In particular, the convex conjugate $h_1^*(y)$ is finite for all values $y\in\mathbb{R}$, and we obtain the explicit Young inequality
	\begin{equation}\label{inequality:Young:3}
		yz\leq
		\left[(a+z)\log\left(1+\frac za\right)+(1-a-z)\log\left(1-\frac z{1-a}\right)\right]
		+\left[\log\big(1+a(e^y-1)\big)-ay\right],
	\end{equation}
	for all $z\in [-a,1-a]$ and $y\in\mathbb{R}$.
	
	Next, observing that
	\begin{equation*}
		\begin{aligned}
			{h_1^*}''(y)
			&=\frac{a(1-a)e^y}{(1+a(e^y-1))^2}
			=\frac{a(1-a)e^{-y}}{(1+(1-a)(e^{-y}-1))^2}
			\\
			&\leq a(1-a)e^{|y|}\leq a(1-a)(e^y+e^{-y}),
		\end{aligned}
	\end{equation*}
	we deduce that
	\begin{equation}\label{estimate:crude}
		h_1^*(y)\leq a(1-a)(e^{|y|}-|y|-1)\leq a(1-a)(e^y+e^{-y}-2),
	\end{equation}
	for all $a\in (0,1)$ and $y\in\mathbb{R}$.
	It therefore holds that
	\begin{equation}\label{inequality:Young:1}
		zy\leq h_1(z)+h_1^*(y)\leq h_1(z)+a(1-a)w(y),
	\end{equation}
	for every $z\in [-a,1-a]$ and $y\in\mathbb{R}$.
	
	Now, computing that
	\begin{equation*}
		\frac{\partial}{\partial\lambda}\left(\frac{w(\lambda y)}{\lambda^2}\right)
		=\frac{\lambda y(e^{\lambda y}-e^{-\lambda y})-2(e^{\lambda y}+e^{-\lambda y}-2)}{\lambda^3}\geq 0,
	\end{equation*}
	for all $y\in\mathbb{R}$ and $\lambda>0$, we deduce that $\lambda^{-2}w(\lambda y)$ is increasing in $\lambda$, which implies that
	\begin{equation}\label{inequality:Young:2}
		w(\lambda y)\leq \lambda^2w(y),
	\end{equation}
	for every $\lambda\in [0,1]$ and $y\in\mathbb{R}$.
	
	All in all, combining the functional properties \eqref{inequality:Young:1} and \eqref{inequality:Young:2}, we find that
	\begin{equation*}
		\begin{aligned}
			|zy\alpha|&\leq \frac{1}{\varepsilon^{2-\gamma}}\left(h_1(z)
			+h_1^*\left(\varepsilon^{2-\gamma}|y\alpha|\frac{z}{|z|}\right)\right)
			\\
			&\leq
			\frac{1}{\varepsilon^{2-\gamma}}h_1(z)+
			a(1-a)\varepsilon^{2-\gamma}\alpha^2 w(y),
		\end{aligned}
	\end{equation*}
	for all $z\in [-a,1-a]$ (with $z\neq 0$), $y\in\mathbb{R}$ and $0<\alpha\leq \varepsilon^{\gamma-2}$. Since the case $z=0$ is trivial, the proof of the lemma is complete.
\end{proof}

\begin{proof}[Proof of Proposition \ref{prop:relative:entropy}]
	Setting
	\begin{equation*}
		a=\delta M_\varepsilon,
		\qquad
		z=\delta\varepsilon g_\varepsilon,
		\qquad
		y=1
		\qquad
		\text{and}
		\qquad
		\alpha=\varepsilon^{\frac{\gamma-\tau}2-1}
	\end{equation*}
	in \eqref{inequality:Young:0},
	with $\gamma+\tau\leq 2$, so that $\alpha\varepsilon^{2-\gamma}\leq 1$, yields that
	\begin{equation*}
		\varepsilon^{\frac{\gamma-\tau}2}\delta| g_\varepsilon|
		\leq
		\frac{1}{\varepsilon^{2-\gamma}}h(\delta f_\varepsilon,\delta M_\varepsilon)
		+\delta M_\varepsilon(1-\delta M_\varepsilon)
		\varepsilon^{-\tau} w(1).
	\end{equation*}
	Now, a straightforward integration in spherical coordinates shows that
	\begin{equation*}
		\begin{aligned}
			\int_{\mathbb{R}^d}\delta M_\varepsilon(1-\delta M_\varepsilon)dv
			&=
			\int_{\mathbb{R}^d}
			\frac {dv}{\left(1+e^{\frac{v^2-R^2}{\varepsilon^\tau}}\right)
			\left(1+e^{-\frac{v^2-R^2}{\varepsilon^\tau}}\right)}
			\\
			&=
			|\mathbb{S}^{d-1}|\int_0^\infty
			\frac {r^{d-1}dr}{\left(1+e^{\frac{v^2-R^2}{\varepsilon^\tau}}\right)
			\left(1+e^{-\frac{v^2-R^2}{\varepsilon^\tau}}\right)}
			\\
			&=
			\frac{\varepsilon^\tau |\mathbb{S}^{d-1}|}2\int_{-\frac{R^2}{\varepsilon^\tau}}^\infty
			\frac {(R^2+\varepsilon^\tau u)^{\frac{d-2}2}du}{\left(1+e^{u}\right)
			\left(1+e^{-u}\right)}=O(\varepsilon^\tau),
		\end{aligned}
	\end{equation*}
	while the relative entropy bound \eqref{relative:entropy:1} gives that
	\begin{equation*}
		\frac{1}{\varepsilon^{2-\gamma}}h(\delta f_\varepsilon,\delta M_\varepsilon)=
		O(1)_{L^\infty(dt;L^1(dxdv))}.
	\end{equation*}
	Therefore, decomposing the fluctuations into
	\begin{equation*}
		g_\varepsilon=
		g_\varepsilon\mathds{1}_{\left\{\frac{1}{\varepsilon^{2-\gamma}}h(\delta f_\varepsilon,\delta M_\varepsilon)
		\geq
		\delta M_\varepsilon(1-\delta M_\varepsilon)
		\varepsilon^{-\tau}\right\}}
		+
		g_\varepsilon\mathds{1}_{\left\{\frac{1}{\varepsilon^{2-\gamma}}h(\delta f_\varepsilon,\delta M_\varepsilon)
		<
		\delta M_\varepsilon(1-\delta M_\varepsilon)
		\varepsilon^{-\tau}\right\}}
	\end{equation*}
	allows us to conclude that
	\begin{equation*}
		g_\varepsilon
		=
		O(\varepsilon^{\frac{\tau-\gamma}2})_{L^\infty(dt;L^1(dxdv))}
		+
		O(\varepsilon^{\frac{\tau-\gamma}2})_{L^\infty(dtdx;L^1(dv))},
	\end{equation*}
	which establishes the first bound \eqref{entropy:bound:1}.
	
	Similarly, setting
	\begin{equation*}
		y=\frac{|v^2-R^2|}{2\varepsilon^\tau}
	\end{equation*}
	instead of $y=1$ into \eqref{inequality:Young:0}, we arrive at the estimate
	\begin{equation*}
		\varepsilon^{\frac{\gamma-3\tau}2}\delta
		\big|g_\varepsilon(v^2-R^2)\big|
		\leq
		\frac{2}{\varepsilon^{2-\gamma}}h(\delta f_\varepsilon,\delta M_\varepsilon)
		+2\delta M_\varepsilon(1-\delta M_\varepsilon)
		\varepsilon^{-\tau}
		w\left(\frac{v^2-R^2}{2\varepsilon^\tau}\right).
	\end{equation*}
	As before, a direct calculation establishes that
	\begin{equation}\label{spherical:calculation:2}
		\int_{\mathbb{R}^d}\delta M_\varepsilon(1-\delta M_\varepsilon)
		w\left(\frac{v^2-R^2}{2\varepsilon^\tau}\right)dv
		=
		\frac{\varepsilon^\tau |\mathbb{S}^{d-1}|}2\int_{-\frac{R^2}{\varepsilon^\tau}}^\infty
		\frac {(R^2+\varepsilon^\tau u)^{\frac{d-2}2}w(\frac u2)du}{\left(1+e^{u}\right)
		\left(1+e^{-u}\right)}=O(\varepsilon^\tau),
	\end{equation}
	which, taking the relative entropy bound \eqref{relative:entropy:1} into account, leads to the control
	\begin{equation*}
		g_\varepsilon|v^2-R^2|
		=
		O(\varepsilon^{\frac{3\tau-\gamma}2})_{L^\infty(dt;L^1(dxdv))}
		+
		O(\varepsilon^{\frac{3\tau-\gamma}2})_{L^\infty(dtdx;L^1(dv))},
	\end{equation*}
	provided that $\gamma+\tau\leq 2$, thereby completing the justification of \eqref{entropy:bound:2}.

	We move on now to the proof of the concentrated tightness around $\{|v|=R\}$. To that end, taking $\lambda>0$ and setting
	\begin{equation*}
		a=\delta M_\varepsilon,
		\qquad
		z=\delta\varepsilon g_\varepsilon,
		\qquad
		y=\frac{|v^2-R^2|}{2\varepsilon^\tau}
		\qquad
		\text{and}
		\qquad
		\alpha=e^{\frac \lambda 8}\varepsilon^{\frac{\gamma-\tau}2-1}
	\end{equation*}
	in \eqref{inequality:Young:0},
	with $\gamma+\tau< 2$, so that $\alpha\varepsilon^{2-\gamma}\leq 1$ provided $\varepsilon$ is small enough, we obtain that
	\begin{equation*}
		\varepsilon^{\frac{\gamma-3\tau}2}\delta
		\big|g_\varepsilon(v^2-R^2)\big|
		\leq
		\frac{2}{e^{\frac \lambda 8}\varepsilon^{2-\gamma}}h(\delta f_\varepsilon,\delta M_\varepsilon)
		+2\delta M_\varepsilon(1-\delta M_\varepsilon)
		e^{\frac \lambda 8}\varepsilon^{-\tau}
		w\left(\frac{v^2-R^2}{2\varepsilon^\tau}\right).
	\end{equation*}
	Then, we observe that a simple adaptation of \eqref{spherical:calculation:2} yields that
	\begin{equation*}
		\begin{aligned}
			\int_{\mathbb{R}^d}\delta M_\varepsilon (1-\delta M_\varepsilon)
			&w\left(\frac{v^2-R^2}{2\varepsilon^\tau}\right)
			\mathds{1}_{\left\{|v^2-R^2|\geq \lambda \varepsilon^\tau\right\}}dv
			\\
			&=
			\frac{\varepsilon^\tau |\mathbb{S}^{d-1}|}2\int_{-\frac{R^2}{\varepsilon^\tau}}^\infty
			\frac {(R^2+\varepsilon^\tau u)^{\frac{d-2}2}w(\frac u2)\mathds{1}_{\{|u|\geq\lambda\}}du}{\left(1+e^{u}\right)
			\left(1+e^{-u}\right)}=O(e^{-\frac\lambda 4}\varepsilon^\tau).
		\end{aligned}
	\end{equation*}
	It therefore follows that
	\begin{equation*}
		g_\varepsilon|v^2-R^2|
		\mathds{1}_{\left\{|v^2-R^2|\geq \lambda \varepsilon^\tau\right\}}
		=
		O(e^{-\frac\lambda 8}\varepsilon^{\frac{3\tau-\gamma}2})_{L^\infty(dt;L^1(dxdv))}
		+
		O(e^{-\frac\lambda 8}\varepsilon^{\frac{3\tau-\gamma}2})_{L^\infty(dtdx;L^1(dv))},
	\end{equation*}
	provided that $e^{\frac \lambda 8}\leq \varepsilon^{\frac{\gamma+\tau}2-1}$.
	
	Considering now a subsequence of fluctuations $g_\varepsilon$, as $\varepsilon\to 0$, the preceding estimate allows us to conclude that, for any compact subset $K\subset\mathbb{R}^d$ and any small value $\rho>0$, there exists a large number $\lambda>0$ such that
	\begin{equation*}
		\sup_{t>0}\int_{K\times\mathbb{R}^d}\varepsilon^{\frac{\gamma-3\tau}2}\big|g_\varepsilon(v^2-R^2)\big|
		\mathds{1}_{\left\{|v^2-R^2|\geq \lambda \varepsilon^\tau\right\}}dxdv
		<\rho,
	\end{equation*}
	uniformly in $\varepsilon>0$, which yields the concentrated tightness estimate \eqref{concentrated:tightness}.

	Next, in order to establish the quadratic bound \eqref{fluctuation:H:2}, we first show that
	\begin{equation}\label{lower:bound:2}
		h(z,a)\geq \frac{\log\left(\frac{1-a}a\right)}{1-2a}(z-a)^2\geq 2(z-a)^2,
	\end{equation}
	whenever $a\neq \frac 12$. Perhaps, the simplest way to justify \eqref{lower:bound:2} consists in studying the minima of the smooth function
	\begin{equation*}
		H(z)=\frac{h(z,a)}{(z-a)^2},
	\end{equation*}
	on $z\in (0,1)$, for any fixed $a\neq\frac 12$. Indeed, a direct computation shows that
	\begin{equation*}
		H'(z)=\frac{h(a,z)-h(z,a)}{(z-a)^3}
		\quad\text{and}\quad
		H''(z)=\frac{r(z)}{(z-a)^4},
	\end{equation*}
	with $r(z)=2h(z,a)-4h(a,z)+\frac{(z-a)^2}{z(1-z)}$. Since $r(a)=r'(a)=r''(a)=0$ and $r''(z)>0$, for all $z\neq a$, we deduce that both $r(z)$ and $H''(z)$ are nonnegative functions. Therefore, we conclude that $H(z)$ attains its unique global minimum at the critical point $z=1-a$. Then, the critical value
	\begin{equation*}
		H(1-a)=\frac{\log\left(\frac{1-a}a\right)}{1-2a}\geq 2
	\end{equation*}
	yields the optimal constant claimed in \eqref{lower:bound:2}.
	
	Now, observe that
	\begin{equation*}
		\frac{\log\left(\frac{1-\delta M_\varepsilon}{\delta M_\varepsilon}\right)}{1-2\delta M_\varepsilon}
		=\frac{\exp\left(\frac{v^2-R^2}{\varepsilon^\tau}\right)+1}{\exp\left(\frac{v^2-R^2}{\varepsilon^\tau}\right)-1}
		\left(\frac{v^2-R^2}{\varepsilon^\tau}\right)
		\geq \max\left\{2,\left|\frac{v^2-R^2}{\varepsilon^\tau}\right|\right\},
	\end{equation*}
	where we employed the elementary inequality $\frac{e^z+1}{e^z-1}z\geq\max\{2,|z|\}$, for all $z\in\mathbb{R}$. Therefore, incorporating \eqref{lower:bound:2} in the definition \eqref{relative:entropy:0} of the relative entropy provides the estimate
	\begin{equation*}
		\delta^2\varepsilon^{2}\int_{\mathbb{R}^d\times\mathbb{R}^d}\left(g_\varepsilon\right)^2
		\max\left\{2,\varepsilon^{-\tau}\left|v^2-R^2\right|\right\}
		dx dv
		\leq H(f_\varepsilon|M_\varepsilon)(t),
	\end{equation*}
	which completes the proof of the proposition.
\end{proof}

It is apparent from the bound \eqref{entropy:bound:1} established in Proposition \ref{prop:relative:entropy} that the relative entropy bound \eqref{relative:entropy:1} does not yield a meaningful control of fluctuations in the case $\gamma>\tau$, whereas a choice of parameters $\gamma<\tau$ leads to a trivial asymptotic regime. Therefore, from now on, we are going to focus mainly on the asymptotic regime provided by the case $0<\gamma=\tau\leq 1$. Later on, we will further restrict our attention to the case $0<\gamma=\tau< 1$.

Thus, the next result explores the convergence of density fluctuations satisfying a uniform relative entropy bound in the case $0<\gamma=\tau\leq 1$, and provides a preliminary characterization of their limit points. The full characterization of these limit points will only be possible later, in Section \ref{section:relaxation_and_equilibria}, by exploiting the entropy dissipation bound.

\begin{prop}\label{prop:limit:characterization}
	Consider a family of density distributions $0\leq f_\varepsilon(t,x,v)\leq\delta^{-1}$, with $\varepsilon>0$, such that the relative entropy bound \eqref{relative:entropy:1} holds uniformly in $\varepsilon$, with some fixed parameter $\gamma>0$, where the Fermi--Dirac distribution $M_\varepsilon$ defined in \eqref{normalized:distribution} has a given parameter value $\tau>0$.
	Further suppose that
	\begin{equation*}
		0<\gamma=\tau\leq 1
	\end{equation*}
	and consider the density fluctuations $g_\varepsilon(t,x,v)$ defined in \eqref{epsilon:fluctuations}.
	
	Then, as $\varepsilon\to 0$:
	\begin{itemize}
		\item
		Any subsequence of fluctuations $g_\varepsilon$ is uniformly bounded in
		\begin{equation*}
			L^\infty\big(dt;L^1_\mathrm{loc}\big(dx;L^1\big((1+v^2)dv\big)\big)\big),
		\end{equation*}
		and weakly* relatively compact in the space of locally finite Radon measures
		\begin{equation*}
			\mathcal{M}_\mathrm{loc}\big(\mathbb{R}^+\times\mathbb{R}^d\times\mathbb{R}^d\big).
		\end{equation*}
		
		\item
		For any continuous function $\varphi(t,x,v)$ such that
		\begin{equation}\label{quadratic:bound:1}
			\frac{\varphi(t,x,v)}{1+v^2}\in L^\infty\big(\mathbb{R}^+\times\mathbb{R}^d\times\mathbb{R}^d\big)
		\end{equation}
		and any subsequence of fluctuations $g_\varepsilon$, the corresponding subsequence
		\begin{equation*}
			\int_0^\infty g_\varepsilon(t,x,r\sigma)\varphi(t,x,r\sigma)r^{d-1}dr,
		\end{equation*}
		with $(t,x,\sigma)\in\mathbb{R}^+\times\mathbb{R}^d\times\mathbb{S}^{d-1}$,
		is weakly relatively compact in the space of locally integrable functions
		\begin{equation*}
			L^1_\mathrm{loc}\big(\mathbb{R}^+\times\mathbb{R}^d\times\mathbb{S}^{d-1}\big).
		\end{equation*}
		
		\item
		If $\mu$ is a weak* limit point in $\mathcal{M}_\mathrm{loc}\big(\mathbb{R}^+\times\mathbb{R}^d\times\mathbb{R}^d\big)$ of the family of fluctuations $g_\varepsilon$, then it can be characterized as
		\begin{equation}\label{measure:characterization:1}
			\mu(t,x,v)=g\left(t,x,R\frac{v}{|v|}\right)dt\otimes dx\otimes \delta_{\partial B(0,R)}(v),
		\end{equation}
		where the density $g(t,x,\omega)$ belongs to $L^\infty\big(\mathbb{R}^+;L^2\big(\mathbb{R}^d\times\partial B(0,R)\big)\big)$ and satisfies, for almost every $t\geq 0$,
		\begin{equation}\label{square:integrable:0}
			\int_{\mathbb{R}^d\times\partial B(0,R)}g^2(t,x,\omega)dxd\omega \leq
			\frac{C^{\mathrm{in}}}{R\delta^2}.
		\end{equation}
		Furthermore, up to extraction of a subsequence, it holds that
		\begin{equation}\label{limit:fluctuations}
			\int_{\mathbb{R}^+\times\mathbb{R}^d\times\mathbb{R}^d}
			g_\varepsilon(t,x,v)\varphi(t,x,v)dtdxdv
			\to
			\int_{\mathbb{R}^+\times\mathbb{R}^d\times\partial B(0,R)}
			g(t,x,\omega)\varphi(t,x,\omega)dtdxd\omega,
		\end{equation}
			as $\varepsilon\to 0$, for any continuous function $\varphi(t,x,v)$ which is compactly supported in $(t,x)$ and satisfies the quadratic growth condition \eqref{quadratic:bound:1}.
	\end{itemize}
\end{prop}

\begin{proof}
	Note first that the uniform boundedness of the family of fluctuations $g_\varepsilon$ in
	\begin{equation*}
		L^\infty\big(dt;L^1_\mathrm{loc}\big(dx;L^1\big((1+v^2)dv\big)\big)\big)
		\subset \mathcal{M}_\mathrm{loc}(\mathbb{R}^+\times\mathbb{R}^d\times\mathbb{R}^d)
	\end{equation*}
	follows directly from the combination of \eqref{entropy:bound:1} and \eqref{entropy:bound:2}. In particular, we conclude that the fluctuations $g_\varepsilon$ are weakly* compact in the space of locally finite Radon measures, for the closed unit ball in $\mathcal{M}(K)$, for any compact domain $K\subset \mathbb{R}^+\times\mathbb{R}^d\times\mathbb{R}^d$, is compact in the weak* topology, by the Banach--Alaoglu Theorem. Further recall that compactness and sequential compactness are equivalent in the weak* topology of the unit ball of $\mathcal{M}(K)=C(K)'$, because $C(K)$ is separable.
	
	Next, we consider any continuous function $\varphi(t,x,v)$ with the quadratic growth condition \eqref{quadratic:bound:1} and write
	\begin{equation}\label{ray:1}
		G_\varepsilon(t,x,\sigma)=\int_0^\infty g_\varepsilon(t,x,r\sigma)\varphi(t,x,r\sigma)r^{d-1}dr.
	\end{equation}
	Observe that
	\begin{equation*}
		\int_{\mathbb{S}^{d-1}}\big|G_\varepsilon(t,x,\sigma)\big|d\sigma
		\leq \left\|\frac{\varphi(v)}{1+v^2}\right\|_{L^\infty}\left\|g_\varepsilon(t,x,v)\right\|_{L^1\left((1+v^2)dv\right)},
	\end{equation*}
	whereby $G_\varepsilon$ is locally integrable. Moreover, it can be decomposed into
	\begin{equation*}
		G_\varepsilon =G_\varepsilon^1+G_\varepsilon^2+G_\varepsilon^3,
	\end{equation*}
	where $G_\varepsilon^1$, $G_\varepsilon^2$ and $G_\varepsilon^3$ are obtained by restricting the domain of integration in \eqref{ray:1} to
	\begin{equation*}
		\{|r^2-R^2|<\lambda\varepsilon^\tau\},
		\quad
		\{\lambda\varepsilon^\tau\leq |r^2-R^2|<1\}
		\quad\text{and}\quad
		\{|r^2-R^2|\geq 1\},
	\end{equation*}
	respectively, for some $0<\lambda<\varepsilon^{-\tau}$.
	
	Then, we estimate that
	\begin{equation*}
		\begin{aligned}
			\int_{\mathbb{R}^d\times\mathbb{S}^{d-1}}\big|G_\varepsilon^1(t,x,\sigma)\big|^2dxd\sigma
			&\lesssim
			\int_{\mathbb{R}^d\times\mathbb{S}^{d-1}}
			\left(\int_{\sqrt{R^2-\lambda\varepsilon^\tau}}^{\sqrt{R^2+\lambda\varepsilon^\tau}} |g_\varepsilon(t,x,r\sigma)|dr\right)^2
			dxd\sigma
			\\
			&\lesssim\lambda\varepsilon^{\tau}
			\int_{\mathbb{R}^d\times\mathbb{S}^{d-1}}
			\left(\int_0^\infty |g_\varepsilon(t,x,r\sigma)|^2r^{d-1}dr\right)
			dxd\sigma
			\\
			&=\lambda\varepsilon^\tau\int_{\mathbb{R}^d\times\mathbb{R}^d}g_\varepsilon^2 dxdv\lesssim\lambda,
		\end{aligned}
	\end{equation*}
	 where we employed \eqref{fluctuation:H:2}. In particular, we deduce that $G^1_\varepsilon$ remains uniformly bounded in $L^\infty(dt;L^2(dxd\sigma))$, for any fixed value $\lambda$.
	
	Further notice, for any compact set $K\subset \mathbb{R}^d$, that \eqref{entropy:bound:2} yields that
	\begin{equation}\label{weak:limit:1}
		\begin{aligned}
			\int_{K\times\mathbb{S}^{d-1}}\big|G_\varepsilon^2(t,x,\sigma)\big| dxd\sigma
			&\lesssim
			\int_{K\times\{\left|v^2-R^2\right|\geq\lambda\varepsilon^\tau\}}
			|g_\varepsilon(t,x,v)|
			dxdv
			\\
			&\lesssim
			\int_{K\times\{\left|v^2-R^2\right|\geq\lambda\varepsilon^\tau\}}
			|g_\varepsilon(t,x,v)|\frac{|v^2-R^2|}{\lambda\varepsilon^\tau}
			dxdv
			\lesssim \lambda^{-1},
		\end{aligned}
	\end{equation}
	and
	\begin{equation}\label{weak:limit:2}
		\begin{aligned}
			\int_{K\times\mathbb{S}^{d-1}}\big|G_\varepsilon^3(t,x,\sigma)\big| dxd\sigma
			&\lesssim
			\int_{K\times\{\left|v^2-R^2\right|\geq 1\}}
			|g_\varepsilon(t,x,v)|(1+v^2)
			dxdv
			\\
			&\lesssim
			\int_{K\times\{\left|v^2-R^2\right|\geq 1\}}
			|g_\varepsilon(t,x,v)|\left|v^2-R^2\right|
			dxdv\lesssim \varepsilon^\tau,
		\end{aligned}
	\end{equation}
	as $\varepsilon\to 0$.
	
	All in all, for any compact set $K\subset\mathbb{R}^+\times\mathbb{R}^d\times\mathbb{S}^{d-1}$ and any measurable set $E\subset K$, we see that
	\begin{equation*}
		\int_E\big|G_\varepsilon(t,x,\sigma)\big|dtdxd\sigma
		\lesssim \lambda^\frac 12|E|^\frac 12+\lambda^{-1}+\varepsilon^\tau,
	\end{equation*}
	which implies that $G_\varepsilon$ is uniformly integrable (or equi-integrable) on any compact domain $K$. Therefore, by the Dunford--Pettis compactness criterion, we conclude that any subsequence $G_\varepsilon$ is weakly relatively compact in $L^1(K)$.

	We turn now to the characterization of weak* limit points $\mu\in\mathcal{M}_\mathrm{loc}(\mathbb{R}^+\times\mathbb{R}^d\times\mathbb{R}^d)$ of the family of fluctuations $g_\varepsilon$, as $\varepsilon\to 0$. To that end, up to extraction of a subsequence, we may assume that
	\begin{equation}\label{weak:limit:0}
		\begin{aligned}
			\int_{\mathbb{R}^+\times\mathbb{R}^d\times\mathbb{S}^{d-1}}G_\varepsilon(t,x,\sigma)dtdxd\sigma
			&=
			\int_{\mathbb{R}^+\times\mathbb{R}^d\times\mathbb{R}^d}g_\varepsilon(t,x,v)\varphi(t,x,v)dtdxdv
			\\
			&\to
			\int_{\mathbb{R}^+\times\mathbb{R}^d\times\mathbb{R}^d}\varphi(t,x,v)d\mu(t,x,v),
		\end{aligned}
	\end{equation}
	as $\varepsilon\to 0$,
	for all compactly supported continuous functions $\varphi\in C_c(\mathbb{R}^+\times\mathbb{R}^d\times\mathbb{R}^d)$.
	
	Next, we further introduce
	\begin{equation}\label{ray:2}
		\widetilde G_\varepsilon(t,x,\sigma)=\varphi(t,x,R\sigma)\int_0^\infty g_\varepsilon(t,x,r\sigma)\chi(r)r^{d-1}dr,
	\end{equation}
	for any continuous function $\varphi(t,x,v)$ which satisfies \eqref{quadratic:bound:1} and any continuous cutoff function
	\begin{equation*}
		\mathds{1}_{\{|r^2-R^2|< \frac 12\}}\leq\chi(r)\leq\mathds{1}_{\{|r^2-R^2|< 1\}},
	\end{equation*}
	and decompose
	\begin{equation*}
		\widetilde G_\varepsilon =\widetilde G_\varepsilon^1+\widetilde G_\varepsilon^2,
	\end{equation*}
	where $\widetilde G_\varepsilon^1$ and $\widetilde G_\varepsilon^2$ are obtained by restricting the domain of integration in \eqref{ray:2} to
	\begin{equation*}
		\{|r^2-R^2|<\lambda\varepsilon^\tau\}
		\quad\text{and}\quad
		\{\lambda\varepsilon^\tau\leq |r^2-R^2|\},
	\end{equation*}
	respectively, for some $0<\lambda<\frac 1{2\varepsilon^\tau}$.
	
	Now, for any continuous function $\varphi(t,x,v)$ which is compactly supported in $(t,x)$ and satisfies the quadratic growth condition \eqref{quadratic:bound:1}, we find that
	\begin{equation*}
		\begin{aligned}
			\int_{\mathbb{R}^+\times\mathbb{R}^d\times\mathbb{S}^{d-1}}
			&\left|G_\varepsilon^1(t,x,\sigma)
			-\widetilde G_\varepsilon^1(t,x,\sigma)\right|
			dtdxd\sigma
			\\
			&\lesssim
			\sup_{\{|v^2-R^2|<\lambda\varepsilon^\tau\}}
			\left|\varphi(t,x,v)-\varphi\left(t,x,R\frac{v}{|v|}\right)\right|
			\left\|g_\varepsilon\right\|_{L^1_\mathrm{loc}\left(dtdx;L^1(dv)\right)},
		\end{aligned}
	\end{equation*}
	which implies that
	\begin{equation}\label{strong:limit:1}
		G_\varepsilon^1-\widetilde G_\varepsilon^1
		\to 0, \quad\text{in }L^1(dtdxd\sigma),
	\end{equation}
	as $\varepsilon\to 0$, by uniform continuity of $\varphi$ on compact domains. Moreover, the remaining term $\widetilde G_\varepsilon^2$ satisfies an estimate similar to \eqref{weak:limit:1}.
	
	Therefore, by combining \eqref{weak:limit:1}, \eqref{weak:limit:2} and \eqref{strong:limit:1}, we conclude that
	\begin{equation*}
		G_\varepsilon-\widetilde G_\varepsilon
		\to 0, \quad\text{in }L^1(dtdxd\sigma),
	\end{equation*}
	whence, by \eqref{weak:limit:0},
	\begin{equation*}
		\int_{\mathbb{R}^+\times\mathbb{R}^d\times\mathbb{S}^{d-1}}\widetilde G_\varepsilon(t,x,\sigma)dtdxd\sigma
		\to
		\int_{\mathbb{R}^+\times\mathbb{R}^d\times\mathbb{R}^d}
		\varphi\left(t,x,R\frac v{|v|}\right)\chi(|v|)d\mu(t,x,v)
	\end{equation*}
	and
	\begin{equation*}
		\begin{aligned}
			\int_{\mathbb{R}^+\times\mathbb{R}^d\times\mathbb{S}^{d-1}}
			G_\varepsilon(t,x,\sigma)dtdxd\sigma
			\to&
			\int_{\mathbb{R}^+\times\mathbb{R}^d\times\mathbb{R}^d}
			\varphi\left(t,x,R\frac v{|v|}\right)\chi(|v|)d\mu(t,x,v)
			\\
			&=\int_{\mathbb{R}^+\times\mathbb{R}^d\times\mathbb{R}^d}\varphi(t,x,v)d\mu(t,x,v),
		\end{aligned}
	\end{equation*}
	as $\varepsilon\to 0$, for any continuous function $\varphi(t,x,v)$ which is compactly supported in $(t,x)$ and satisfies the quadratic growth condition \eqref{quadratic:bound:1}.
	
	Next, since, by the preceding steps, the corresponding subsequence
	\begin{equation*}
		\int_0^\infty g_\varepsilon(t,x,r\sigma)\chi(r)r^{d-1}dr
	\end{equation*}
	is weakly relatively compact in $L^1_\mathrm{loc}(\mathbb{R}^+\times\mathbb{R}^d\times\mathbb{S}^{d-1})$, we infer the existence of
	\begin{equation*}
		g(t,x,\omega)\in L^1_\mathrm{loc}(\mathbb{R}^+\times\mathbb{R}^d\times \partial B(0,R))
	\end{equation*}
	such that
	\begin{equation*}
		\begin{aligned}
			\int_{\mathbb{R}^+\times\mathbb{R}^d\times\mathbb{S}^{d-1}}\widetilde G_\varepsilon(t,x,\sigma)dtdxd\sigma
			\to &
			\int_{\mathbb{R}^+\times\mathbb{R}^d\times\mathbb{S}^{d-1}}
			\varphi\left(t,x,R\sigma\right)R^{d-1}g(t,x,R\sigma)dtdxd\sigma
			\\
			&=
			\int_{\mathbb{R}^+\times\mathbb{R}^d\times\partial B(0,R)}
			\varphi\left(t,x,\omega\right)g(t,x,\omega)dtdxd\omega.
		\end{aligned}
	\end{equation*}
	It therefore follows that
	\begin{equation*}
		\int_{\mathbb{R}^+\times\mathbb{R}^d\times\mathbb{R}^d}\varphi(t,x,v)d\mu(t,x,v)
		=
		\int_{\mathbb{R}^+\times\mathbb{R}^d\times\partial B(0,R)}
		\varphi\left(t,x,\omega\right)g(t,x,\omega)dtdxd\omega,
	\end{equation*}
	for all suitable test functions $\varphi(t,x,v)$, which establishes the characterization \eqref{measure:characterization:1} with a locally integrable density $g(t,x,\omega)$.
	
	There only remains to show that, in fact, the density $g(t,x,\omega)$ is square-integrable and satisfies the bound \eqref{square:integrable:0}.
	
	To that end, we set
	\begin{equation*}
		z=\delta\varepsilon g_\varepsilon,
		\quad
		a=\delta M_\varepsilon
		\quad\text{and}\quad
		y=\varepsilon^{1-\tau}\lambda \varphi\left(t,x,R\frac v{|v|}\right),
	\end{equation*}
	for some $\varphi(t,x,\omega)\in C_c(\mathbb{R}^+\times\mathbb{R}^d\times\partial B(0,R))$ and $\lambda>0$, in \eqref{inequality:Young:3}.
	Then, integrating the resulting inequality in $t$, $x$ and $v$, we obtain, for any $0\leq t_1\leq t_2$, that
	\begin{equation*}
		\begin{aligned}
			\delta\lambda\int_{[t_1,t_2]\times \mathbb{R}^d \times\mathbb{R}^d} & g_\varepsilon\varphi dtdxdv
			\\
			&\leq
			\frac 1{\varepsilon^{2-\tau}}\int_{[t_1,t_2]\times \mathbb{R}^d \times\mathbb{R}^d}h(\delta f_\varepsilon,\delta M_\varepsilon)dtdxdv
			\\
			&\quad +\frac 1{\varepsilon^{2-\tau}}\int_{[t_1,t_2]\times \mathbb{R}^d \times\mathbb{R}^d}\left[\log\big(1+\delta M_\varepsilon(e^{\varepsilon^{1-\tau}\lambda\varphi}-1)\big)-\delta M_\varepsilon\varepsilon^{1-\tau}\lambda\varphi\right]dtdxdv
			\\
			&\leq (t_2-t_1)C^\mathrm{in}+\int_{[t_1,t_2]\times \mathbb{R}^d\times\mathbb{S}^{d-1}} I_\varepsilon\big(\lambda \varphi(t,x,R\sigma)\big)
			dtdxd\sigma,
		\end{aligned}
	\end{equation*}
	where
	\begin{equation*}
		\begin{aligned}
			I_\varepsilon(y)&=\frac 1{\varepsilon^{2-\tau}}\int_0^\infty
			\left[\log\left(1+\frac{e^{\varepsilon^{1-\tau} y}-1}{1+e^{\frac{r^2-R^2}{\varepsilon^\tau}}}\right)
			-\frac{\varepsilon^{1-\tau}y}{1+e^{\frac{r^2-R^2}{\varepsilon^\tau}}}\right]
			r^{d-1}dr
			\\
			&=\frac 1{2\varepsilon^{2(1-\tau)}}
			\int_{-\frac {R^2}{\varepsilon^\tau}}^\infty
			\left[\log\left(1+\frac{e^{\varepsilon^{1-\tau} y}-1}{1+e^{u}}\right)
			-\frac{\varepsilon^{1-\tau} y}{1+e^{u}}\right]
			(R^2+\varepsilon^\tau u)^\frac{d-2}2du.
		\end{aligned}
	\end{equation*}
	Thus, letting $\varepsilon\to 0$ yields that
	\begin{equation}\label{limit:entropy:1}
		\begin{aligned}
			\delta\lambda \int_{[t_1,t_2]\times \mathbb{R}^d \times\partial B(0,R)} & g(t,x,\omega)\varphi(t,x,\omega) dtdxd\omega
			\\
			&\leq
			(t_2-t_1)C^\mathrm{in}
			+\liminf_{\varepsilon\to 0}\int_{[t_1,t_2]\times \mathbb{R}^d\times\mathbb{S}^{d-1}} I_\varepsilon\big(\lambda \varphi(t,x,R\sigma)\big)
			dtdxd\sigma,
		\end{aligned}
	\end{equation}
	by weak* convergence.
	
	Next, we recall, from \eqref{estimate:crude} and \eqref{inequality:Young:2}, the crude estimate
	\begin{equation*}
		\log(1+a(e^{\rho y}-1))-a\rho y\leq a(1-a)(e^{\rho y}+e^{-\rho y}-2)
		\leq a(1-a)\rho^2(e^{ y}+e^{- y}-2),
	\end{equation*}
	for all $a\in[0,1]$, $y\in\mathbb{R}$ and $\rho\in [0,1]$. It then follows, for any $u\geq-\frac {R^2}{\varepsilon^\tau}$ and $0<\varepsilon\leq 1$, that
	\begin{equation}\label{dominated:estimate}
		\frac 1{\varepsilon^{2(1-\tau)}}\left[\log\left(1+\frac{e^{\varepsilon^{1-\tau} y}-1}{1+e^{u}}\right)
		-\frac{\varepsilon^{1-\tau} y}{1+e^{u}}\right]
		\leq
		\frac{2e^{u}}{(1+e^{u})^2}
		(e^{|y|}-1)
		\leq 2e^{-|u|}(e^{|y|}-1),
	\end{equation}
	which implies that the integrand of $I_\varepsilon\big(\lambda \varphi(t,x,R\sigma)\big)$ is dominated by the integrable function
	\begin{equation*}
		e^{-|u|}(R^2+\varepsilon^\tau|u|)^{\frac{d-2}2}\left(e^{\lambda \left|\varphi(t,x,R\sigma)\right|}-1\right).
	\end{equation*}
	Therefore, evaluating that
	\begin{equation}\label{special:limit}
		\lim_{\rho\to 0}
		\frac{\log\left(1+a(e^{\rho y}-1)\right)
		-a\rho y}{\rho^2}
		=a(1-a)\frac{y^2}2,
	\end{equation}
	for any $a\in [0,1]$ and $y\in\mathbb{R}$, we conclude, in the case $0<\tau<1$, by the Dominated Convergence Theorem, that
	\begin{equation*}
		\begin{aligned}
			&\lim_{\varepsilon\to 0}
			\int_{[t_1,t_2]\times \mathbb{R}^d\times\mathbb{S}^{d-1}} I_\varepsilon\big(\lambda \varphi(t,x,R\sigma)\big)
			dtdxd\sigma
			\\
			&=\int_{[t_1,t_2]\times \mathbb{R}^d\times\mathbb{S}^{d-1}}
			\frac{R^{d-2}}{4}\lambda^2\varphi(t,x,R\sigma)^2\left(\int_{-\infty}^\infty
			\frac{e^u}{(1+e^u)^2}du\right)
			dtdxd\sigma
			\\&=
			\frac{\lambda^2}{4R}\int_{[t_1,t_2]\times \mathbb{R}^d\times\partial B(0,R)}
			\varphi(t,x,\omega)^2dtdxd\omega.
		\end{aligned}
	\end{equation*}
	As for the case $\tau=1$, the Dominated Convergence Theorem yields that
	\begin{equation*}
		\begin{aligned}
			&\lim_{\varepsilon\to 0}\int_{[t_1,t_2]\times \mathbb{R}^d\times\mathbb{S}^{d-1}} I_\varepsilon\big(\lambda \varphi(t,x,R\sigma)\big)
			dtdxd\sigma
			\\&=
			(2R)^{-1}\int_{[t_1,t_2]\times \mathbb{R}^d\times\partial B(0,R)}
			\left(\int_{-\infty}^\infty
			\left[\log\left(1+\frac{e^{\lambda \varphi(t,x,\omega)}-1}{1+e^{u}}\right)
			-\frac{\lambda \varphi(t,x,\omega)}{1+e^{u}}\right]
			du\right)dtdxd\omega
			\\&=
			\frac{\lambda^2}{4R}\int_{[t_1,t_2]\times \mathbb{R}^d\times\partial B(0,R)}
			\varphi(t,x,\omega)^2dtdxd\omega,
		\end{aligned}
	\end{equation*}
	where we employed the remarkable identity
	\begin{equation*}
		\int_{-\infty}^\infty
		\left[\log\left(1+\frac{e^y-1}{1+e^{u}}\right)
		-\frac{y}{1+e^{u}}\right]
		du=
		\int_0^1\left[\log(1+a(e^y-1))-ay\right]\frac{da}{a(1-a)}
		=\frac{y^2}2,
	\end{equation*}
	for all $y\in\mathbb{R}$, which readily follows from the observation that
	\begin{equation*}
		\frac{d}{dy}\int_0^1\left[\log(1+a(e^y-1))-ay\right]\frac{da}{a(1-a)}
		=
		\int_0^1\frac{e^y-1}{1+a(e^y-1)}da
		=y.
	\end{equation*}
	(Note that the differentiation under the integral sign is easily justified by the continuous differentiability of the integrand.)
	
	All in all, incorporating the previous limits into \eqref{limit:entropy:1}, we conclude that
	\begin{equation*}
		\begin{aligned}
			\delta \int_{[t_1,t_2]\times \mathbb{R}^d \times\partial B(0,R)} & g(t,x,\omega)\varphi(t,x,\omega) dtdxd\omega
			\\
			&\leq
			\lambda^{-1}(t_2-t_1)C^\mathrm{in}
			+\frac{\lambda}{4R}\int_{[t_1,t_2]\times \mathbb{R}^d\times\partial B(0,R)}
			\varphi(t,x,\omega)^2dtdxd\omega,
		\end{aligned}
	\end{equation*}
	which yields, by optimizing the value of $\lambda>0$, that
	\begin{equation*}
		\delta \int_{[t_1,t_2]\times \mathbb{R}^d \times\partial B(0,R)}
		g(t,x,\omega)\varphi(t,x,\omega) dtdxd\omega
		\leq
		\left(\frac{(t_2-t_1)C^\mathrm{in}}{R}\right)^\frac 12
		\|\varphi\|_{L^2(dtdxd\omega)}.
	\end{equation*}
	Finally, taking the supremum over all $\varphi$, we conclude that
	\begin{equation*}
		\int_{[t_1,t_2]\times \mathbb{R}^d \times\partial B(0,R)}
		g^2(t,x,\omega) dtdxd\omega
		\leq
		\frac{(t_2-t_1)C^\mathrm{in}}{R\delta^2},
	\end{equation*}
	which, by the arbitrariness of $t_1$ and $t_2$, gives \eqref{square:integrable:0} and completes the proof of the proposition.
\end{proof}

The preceding proposition established an interesting asymptotic phenomenon which differs from the usual behavior of solutions to the Boltzmann equation in classical settings of hydrodynamic limits. Specifically, Proposition \ref{prop:limit:characterization} shows that the collection of fluctuations $g_\varepsilon$ is uniformly bounded in $L^1$ in $x$ and $v$, while the limiting density $\mu$ in \eqref{measure:characterization:1} is square-integrable in $x$, but not in $v$. In fact, the density $\mu$ is a singular measure in velocity.

In the classical setting previously studied rigorously in \cite[Section 3]{bgl93}, it can be shown that the limiting fluctuations are square-integrable in both variables $x$ and $v$ (see Proposition 3.1, therein). Moreover, by exploiting the control of the entropy dissipation, one can further establish that the limiting fluctuations even take the form of a smooth infinitesimal Maxwellian distribution in $v$ (see Proposition 3.8, therein).

This is considerably different in the quantum setting provided by Fermi grounds states. Indeed, in this case, since, as shown in Proposition \ref{prop:limit:characterization}, fluctuations tend to concentrate on the sphere $\{|v|=R\}$, we see that limiting densities can never exhibit an asymptotically smooth behavior in their velocity variable.

It is to be emphasized that classical results on hydrodynamic limits of Boltzmann equations typically do rely on the asymptotic velocity smoothness of fluctuations. Thus, for example, major progress in the study of the incompressible Navier--Stokes regime of the Boltzmann equation was achieved in \cite{gsr04}, where it was noticed that some uniform integrability of fluctuations in velocity could be transferred to the other variables (see Section 3.2, therein). In fact, this strategy of transfer of uniform integrability is now a cornerstone of proofs of hydrodynamic limits in the incompressible Navier--Stokes regime. For instance, it is featured in \cite[Chapter 7]{asr19}, where it is discussed in detail. The velocity regularity of fluctuations also plays a central role in \cite{b15}, and subsequent related works in smooth settings near equilibria, where the incompressible Navier--Stokes regime of the Boltzmann equation is studied in a high regularity setting.

All in all, it is apparent that the singularity of quantum effects near Fermi ground states brings a new dimension to the study of hydrodynamic limits and the promise of fascinating mathematical challenges.

\subsection{Optimality of constants}

It is possible to show that the constants in \eqref{square:integrable:0} are the best possible ones. Indeed, let us suppose, under the assumptions of Proposition \ref{prop:limit:characterization}, that any limiting density $g(t,x,\sigma)$, as given in \eqref{measure:characterization:1}, satisfies that
\begin{equation}\label{square:integrable:1}
	\int_{\mathbb{R}^d\times\partial B(0,R)}g^2(t,x,\omega)dxd\omega \leq
	\lambda C^\mathrm{in},
\end{equation}
for some independent constant $\lambda>0$. Our goal is to show that $\delta^2 R \lambda\geq 1$.

We treat the case $\tau=1$ first. To that end, we consider the family of fluctuations
\begin{equation*}
	g_\varepsilon(t,x,v)=\delta^{-1}\varepsilon^{-1}
	\frac {e^{\frac{v^2-R^2}{\varepsilon}}}{1+e^{\frac{v^2-R^2}{\varepsilon}}}
	\mathds{1}_{\{x\in K,\ |v^2-R^2|\leq A\varepsilon\}},
\end{equation*}
where $A>0$ is fixed with respect to $\varepsilon$ and $K\subset\mathbb{R}^d$ is an arbitrary compact set. Then, we verify that
\begin{equation*}
	g_\varepsilon\rightharpoonup^*
	g\left(t,x,R\frac{v}{|v|}\right)dt\otimes dx\otimes \delta_{\partial B(0,R)}(v),
\end{equation*}
as $\varepsilon\to 0$, where
\begin{equation*}
	g=\frac{A}{2\delta R}\mathds{1}_{K}(x).
\end{equation*}
In particular, we compute that
\begin{equation*}
	\int_{\mathbb{R}^d\times\partial B(0,R)}g^2(t,x,\omega)dxd\omega=
	\frac {A^2}{4\delta^2} R^{d-3}|K\times \mathbb{S}^{d-1}|.
\end{equation*}

Then, writing $f_\varepsilon=M_\varepsilon+\varepsilon g_\varepsilon$ and noticing that $\delta f_\varepsilon\equiv 1$ on $\{x\in K,\ |v^2-R^2|\leq A\varepsilon\}$, we see that
\begin{equation*}
	\begin{aligned}
		\frac 1\varepsilon H\big(\delta f_\varepsilon|\delta M_\varepsilon\big)
		&=\frac 1\varepsilon\int_{K\times \{|v^2-R^2|\leq A\varepsilon\}}
		h(1,\delta M_\varepsilon)dx dv
		=\frac 1\varepsilon\int_{K\times \{|v^2-R^2|\leq A\varepsilon\}}
		|\log(\delta M_\varepsilon)|dx dv
		\\
		&=\frac{|K\times \mathbb{S}^{d-1}|}\varepsilon
		\int_{\sqrt{R^2-A\varepsilon}}^{\sqrt{R^2+A\varepsilon}} \log \left(1+e^{\frac{r^2-R^2}{\varepsilon}}\right) r^{d-1}dr
		\\
		&=\frac{|K\times \mathbb{S}^{d-1}|}2
		\int_{-A}^{A} \log\left(1+e^{u}\right) (R^2+\varepsilon u)^\frac{d-2}2du
		\\
		&\stackrel{\varepsilon\to 0}{\longrightarrow}
		|K\times \mathbb{S}^{d-1}|\frac{R^{d-2}}{2}
		\int_{-A}^{A}
		\log(1+e^{u})du.
	\end{aligned}
\end{equation*}
Therefore, applying \eqref{square:integrable:1} to the family of fluctuations $g_\varepsilon$ yields that
\begin{equation*}
	\frac {A^2}{4\delta^2} R^{d-3}|K\times \mathbb{S}^{d-1}|
	\leq \lambda
	|K\times \mathbb{S}^{d-1}|\frac{R^{d-2}}{2}
	\int_{-A}^{A}
	\log(1+e^{u})du,
\end{equation*}
which, upon noticing that
\begin{equation*}
	\lim_{A\to\infty}A^{-2}\int_{-A}^{A}
	\log(1+e^{u})du=\frac 12,
\end{equation*}
gives that $\delta^2 R\lambda\geq 1$.
This establishes the optimality of \eqref{square:integrable:0}, in the case $\tau=1$.

Moving on to the case $0<\tau<1$, we consider another family of fluctuations
\begin{equation*}
	g_\varepsilon(t,x,v)=\delta^{-1}\varepsilon^{-\tau}
	Q\left(\frac{v^2-R^2}{\varepsilon^\tau}\right)
	\mathds{1}_{\{x\in K,\ |v^2-R^2|\leq A\varepsilon^\tau\}},
\end{equation*}
where $Q(u)$ is a continuous profile to be determined later on, $A>0$ is fixed with respect to $\varepsilon$ and $K\subset\mathbb{R}^d$ is an arbitrary compact set. We can now verify that
\begin{equation*}
	g_\varepsilon\rightharpoonup^*
	g\left(t,x,R\frac{v}{|v|}\right)dt\otimes dx\otimes \delta_{\partial B(0,R)}(v),
\end{equation*}
as $\varepsilon\to 0$, where
\begin{equation*}
	g=\frac{\int_{-A}^AQ(u)du}{2\delta R}\mathds{1}_{K}(x),
\end{equation*}
and
\begin{equation*}
	\int_{\mathbb{R}^d\times\partial B(0,R)}g^2(t,x,\omega)dxd\omega=
	\frac {R^{d-3}}{4\delta^2}|K\times \mathbb{S}^{d-1}|
	\left(\int_{-A}^AQ(u)du\right)^2.
\end{equation*}

Then, observing that
\begin{equation*}
	\lim_{z\to 0}\frac{h(a+z,a)}{z^2}=\frac 1{2a(1-a)},
\end{equation*}
for any $a\in (0,1)$, we evaluate that
\begin{equation*}
	\begin{aligned}
		\frac 1{\varepsilon^{2-\tau}} H\big(\delta f_\varepsilon|\delta M_\varepsilon\big)
		&=\frac 1{\varepsilon^{2-\tau}}
		\int_{K\times \{|v^2-R^2|\leq A\varepsilon^\tau\}}
		h\left(\delta M_\varepsilon+\varepsilon^{1-\tau}
		Q\left(\frac{v^2-R^2}{\varepsilon^\tau}\right),
		\delta M_\varepsilon\right)dx dv
		\\
		&=\frac {|K\times \mathbb{S}^{d-1}|}{2\varepsilon^{2(1-\tau)}}
		\int_{-A}^A
		h\left(\frac{1}{1+e^u}+\varepsilon^{1-\tau}Q(u),
		\frac{1}{1+e^u}\right)(R^2+\varepsilon^\tau u)^\frac{d-2}2 du
		\\
		&\stackrel{\varepsilon\to 0}{\longrightarrow}
		|K\times \mathbb{S}^{d-1}|\frac{R^{d-2}}{4}
		\int_{-A}^{A}
		\frac{(1+e^u)^2}{e^u}Q(u)^2 du.
	\end{aligned}
\end{equation*}
This time, applying \eqref{square:integrable:1} to the family of fluctuations $g_\varepsilon$ yields that
\begin{equation*}
	\frac {R^{d-3}}{4\delta^2}|K\times \mathbb{S}^{d-1}|
	\left(\int_{-A}^AQ(u)du\right)^2
	\leq\lambda
	|K\times \mathbb{S}^{d-1}|\frac{R^{d-2}}{4}
	\int_{-A}^{A}
	\frac{(1+e^u)^2}{e^u}Q(u)^2 du,
\end{equation*}
which, upon setting $Q(u)=\frac{e^u}{(1+e^u)^2}$ and letting $A\to\infty$, leads to
\begin{equation*}
	\int_{-\infty}^\infty\frac{e^u}{(1+e^u)^2}du
	\leq\lambda
	\delta^2 R.
\end{equation*}
Finally, evaluating that
\begin{equation*}
	\int_{-\infty}^\infty\frac{e^u}{(1+e^u)^2}du=1
\end{equation*}
shows that $\lambda \delta^2 R\geq 1$. This completes the justification of the optimality of \eqref{square:integrable:0} in the remaining range $0<\tau<1$.

\subsection{Renormalized fluctuations}

As a way of refining our understanding of the asymptotic square-integrability of densities, we introduce now the renormalized fluctuations $\phi_\varepsilon(t,x,v)$ and $\psi_\varepsilon(t,x,v)$ of order $\varepsilon$ defined by
\begin{equation}\label{epsilon:renormalized:fluctuations}
	\sqrt{\delta f_\varepsilon}=\sqrt{\delta M_\varepsilon}\left(1+\frac{\delta\varepsilon}2 \phi_\varepsilon\right)
	\qquad\text{and}\qquad
	\sqrt{1-\delta f_\varepsilon}=\sqrt{1-\delta M_\varepsilon}\left(1-\frac{\delta\varepsilon}2 \psi_\varepsilon\right).
\end{equation}
The next result establishes simple uniform bounds on $\phi_\varepsilon$ and $\psi_\varepsilon$, thereby showing the significance of these renormalized fluctuations. Similar square-root-renormalization techniques are often employed in the study of hydrodynamic limits and we refer to \cite[Chapter 5]{asr19} for an example of the use of square-root renormalizations in a classical setting.

\begin{lem}\label{lemma:squareroot:relative:entropy}
	Consider a family of density distributions $0\leq f_\varepsilon(t,x,v)\leq\delta^{-1}$, with $\varepsilon>0$, such that the relative entropy bound \eqref{relative:entropy:1} holds uniformly in $\varepsilon$, with some fixed parameter $\gamma>0$, where the Fermi--Dirac distribution $M_\varepsilon$ defined in \eqref{normalized:distribution} has a given parameter value $\tau>0$.
	Further consider the density fluctuations $\phi_\varepsilon(t,x,v)$ and $\psi_\varepsilon(t,x,v)$ defined in \eqref{epsilon:renormalized:fluctuations}.
	
	Then, as $\varepsilon\to 0$, any subsequences of renormalized fluctuations $\phi_\varepsilon$ and $\psi_\varepsilon$ satisfy the bounds
	\begin{equation*}
		\phi_\varepsilon=O(\varepsilon^{-\frac\gamma 2})_{L^\infty(dt;L^2(M_\varepsilon dxdv))}
	\end{equation*}
	and
	\begin{equation*}
		\psi_\varepsilon=O(\varepsilon^{-\frac\gamma 2})_{L^\infty(dt;L^2((1-\delta M_\varepsilon) dxdv))}.
	\end{equation*}
\end{lem}

\begin{proof}
	The justification of the bounds on $\phi_\varepsilon$ and $\psi_\varepsilon$ will follow from another careful analysis of the function $h(z,a)$ given in \eqref{entropy:function}. More precisely, we consider now the decomposition
	\begin{equation*}
		h(z,a)=ah_0\left(\frac za\right)+(1-a)h_0\left(\frac{1-z}{1-a}\right),
	\end{equation*}
	where the nonnegative function
	\begin{equation*}
		h_0(y)=y\log y-y+1
	\end{equation*}
	is defined for all $y\in[0,\infty)$.
	
	Next, we employ the elementary inequality (see \cite[Appendix B]{asr19} for a brief justification)
	\begin{equation*}
		h_0(y)\geq \left(\sqrt y - 1\right)^2
	\end{equation*}
	to deduce that
	\begin{equation*}
		h(z,a)\geq \left(\sqrt z-\sqrt a\right)^2+\left(\sqrt {1-z}-\sqrt {1-a}\right)^2.
	\end{equation*}
	It then follows that
	\begin{equation*}
		\begin{aligned}
			H(f_\varepsilon|M_\varepsilon)
			&=\int_{\mathbb{R}^d\times\mathbb{R}^d}h(\delta f_\varepsilon,\delta M_\varepsilon)dx dv
			\\
			&\geq
			\int_{\mathbb{R}^d\times\mathbb{R}^d}(\sqrt{\delta f_\varepsilon}-\sqrt{\delta M_\varepsilon})^2
			+(\sqrt{1-\delta f_\varepsilon}-\sqrt{1-\delta M_\varepsilon})^2dx dv.
		\end{aligned}
	\end{equation*}
	In particular, in view of the relative entropy bound \eqref{relative:entropy:1}, we have now reached the uniform estimate
	\begin{equation*}
		\frac{\delta^2\varepsilon^\gamma}{4}
		\int_{\mathbb{R}^d\times\mathbb{R}^d}
		\delta M_\varepsilon \phi_\varepsilon^2
		+ (1-\delta M_\varepsilon) \psi_\varepsilon^2
		dx dv
		\leq \frac 1{\varepsilon^{2-\gamma}}H(f_\varepsilon|M_\varepsilon)
		\leq C^\mathrm{in},
	\end{equation*}
	for almost all $t\geq 0$, which completes the proof of the lemma.
\end{proof}

By combining \eqref{epsilon:fluctuations} with \eqref{epsilon:renormalized:fluctuations}, observe that
\begin{equation*}
	g_\varepsilon=\delta M_\varepsilon\phi_\varepsilon+\frac{\delta\varepsilon}4 \delta M_\varepsilon\phi_\varepsilon^2
	=\left(1-\delta M_\varepsilon\right)\psi_\varepsilon-\frac{\delta\varepsilon}4 \left(1-\delta M_\varepsilon\right)\psi_\varepsilon^2
\end{equation*}
and, therefore,
\begin{equation}\label{renormalized:decomposition}
	\begin{aligned}
		g_\varepsilon
		&=\left(1-\delta M_\varepsilon\right)g_\varepsilon + \delta M_\varepsilon g_\varepsilon
		\\
		&=\delta M_\varepsilon\left(1-\delta M_\varepsilon\right)
		\left(\phi_\varepsilon+\psi_\varepsilon\right)
		+\frac{\delta\varepsilon}4 \delta M_\varepsilon\left(1-\delta M_\varepsilon\right)
		\left(\phi_\varepsilon^2-\psi_\varepsilon^2\right).
	\end{aligned}
\end{equation}
Moreover, employing that
\begin{equation}\label{basic:asymptotics}
	\begin{aligned}
		\int_{\mathbb{R}^d}\big(\delta M_\varepsilon(1-\delta M_\varepsilon)\big)^\alpha dv
		&=|\mathbb{S}^{d-1}|\int_0^\infty \frac{e^{\alpha\frac{r^2-R^2}{\varepsilon^\tau}}r^{d-1}}{\left(1+e^{\frac{r^2-R^2}{\varepsilon^\tau}}\right)^{2\alpha}}dr
		\\
		&=\frac{\varepsilon^\tau|\mathbb{S}^{d-1}|}{2^{1+2\alpha}}\int_{-\frac{R^2}{\varepsilon^\tau}}^\infty \frac{(R^2+\varepsilon^\tau u)^\frac{d-2}2}{\cosh^{2\alpha}\left(\frac u2\right)}du
		=O(\varepsilon^{\tau}),
	\end{aligned}
\end{equation}
for any $\alpha>0$, to deduce that
\begin{equation*}
	\delta M_\varepsilon\left(1-\delta M_\varepsilon\right)
	=O(\varepsilon^\frac\tau p)_{L^p(dv)},
\end{equation*}
for any $0< p\leq \infty$, we infer, by virtue of Lemma \ref{lemma:squareroot:relative:entropy}, that
\begin{equation}\label{bound:extra}
	\begin{aligned}
		\Big\|(\delta M_\varepsilon(1-\delta M_\varepsilon))&^{\frac 12+\alpha}
		\left(\phi_\varepsilon+\psi_\varepsilon\right)\Big\|_{L^\infty (dt;L^2(dx;L^q(dv)))}
		\\
		&\leq
		\|(\delta M_\varepsilon\left(1-\delta M_\varepsilon\right))^\alpha
		\|_{L^{\frac{2q}{2-q}}(dv)}
		\left\|\phi_\varepsilon+\psi_\varepsilon\right\|_{L^\infty (dt;L^2(\delta M_\varepsilon\left(1-\delta M_\varepsilon\right)dxdv))}
		\\
		&\lesssim \varepsilon^{\tau\left(\frac 1q-\frac 12\right)-\frac\gamma 2},
	\end{aligned}
\end{equation}
for all $1\leq q\leq 2$ and $\alpha> 0$, or $q=2$ and $\alpha=0$.

Therefore, we conclude that the preceding decomposition implies the asymptotic control
\begin{equation}\label{asymptotic:control}
	\begin{aligned}
		g_\varepsilon&= \sqrt{\delta M_\varepsilon(1-\delta M_\varepsilon)}
		O(\varepsilon^{-\frac{\gamma}2})_{L^\infty (dt;L^2(dxdv))}
		+O(\varepsilon^{1-\gamma})_{L^\infty (dt;L^1(dxdv))}
		\\
		&= O(\varepsilon^{\frac{\tau-\gamma}2})_{L^\infty (dt;L^2(dx;L^1(dv)))}
		+O(\varepsilon^{1-\gamma})_{L^\infty (dt;L^1(dxdv))},
	\end{aligned}
\end{equation}
for all $\gamma>0$ and $\tau>0$, as a direct consequence of the relative entropy bound \eqref{relative:entropy:1}.

Thus, in the case $0<\gamma=\tau<1$, which is of interest to us, we see now that the fluctuations $g_\varepsilon$ are uniformly bounded in $L^\infty (dt;L^2(dx;L^1(dv)))$ up to a remainder of order $\varepsilon^{1-\gamma}$ in $L^\infty (dt;L^1(dxdv))$. Moreover, by virtue of Proposition \ref{prop:limit:characterization}, if $\mu$ is a weak* limit point in $\mathcal{M}_\mathrm{loc}$ of a converging subsequence of fluctuations $g_\varepsilon$, then
\begin{equation*}
	\delta M_\varepsilon\phi_\varepsilon,
	\qquad
	\left(1-\delta M_\varepsilon\right)\psi_\varepsilon
	\qquad\text{and}\qquad
	\delta M_\varepsilon\left(1-\delta M_\varepsilon\right)\left(\phi_\varepsilon+\psi_\varepsilon\right)
\end{equation*}
also converge toward $\mu$, at least in the weak* topology of locally finite Radon measures.

\bigskip

The following corollary extends the bounds from the previous lemma to other $L^p$ spaces, by exploiting the fact that fluctuations are bounded pointwise, but not uniformly so.

\begin{cor}\label{cor:squareroot:relative:entropy}
	Under the same assumptions as Lemma \ref{lemma:squareroot:relative:entropy}, any subsequences of renormalized fluctuations $\phi_\varepsilon$ and $\psi_\varepsilon$ satisfy the bounds
	\begin{equation*}
		\sqrt{\delta M_\varepsilon}\phi_\varepsilon
		=O(\varepsilon^{\frac {2-\gamma}p-1})_{L^\infty(dt;L^p(dxdv))}
	\end{equation*}
	and
	\begin{equation*}
		\sqrt{1-\delta M_\varepsilon}\psi_\varepsilon
		=O(\varepsilon^{\frac {2-\gamma}p-1})_{L^\infty(dt;L^p(dxdv))},
	\end{equation*}
	for all $2\leq p\leq\infty$.
\end{cor}

\begin{proof}
	Notice that $\sqrt{\delta M_\varepsilon}\phi_\varepsilon$ and $\sqrt{1-\delta M_\varepsilon}\psi_\varepsilon$ take values in $[-\frac 2{\delta\varepsilon},\frac 2{\delta\varepsilon}]$, whereby
	\begin{equation*}
		\varepsilon\sqrt{\delta M_\varepsilon}\phi_\varepsilon
		\qquad\text{and}\qquad
		\varepsilon\sqrt{1-\delta M_\varepsilon}\psi_\varepsilon
	\end{equation*}
	remain bounded in $L^\infty(dtdxdv)$, uniformly in $\varepsilon$. Then, combining this pointwise bound with the bounds from Lemma \ref{lemma:squareroot:relative:entropy} yields that
	\begin{equation*}
		\begin{aligned}
			\left\|\left(\sqrt{\delta M_\varepsilon}\phi_\varepsilon,\sqrt{1-\delta M_\varepsilon}\psi_\varepsilon\right)\right\|_{L^\infty(dt;L^p(dxdv))}
			&\leq
			\left\|\left(\sqrt{\delta M_\varepsilon}\phi_\varepsilon,\sqrt{1-\delta M_\varepsilon}\psi_\varepsilon\right)\right\|_{L^\infty(dt;L^2(dxdv))}^\frac 2p
			\\
			&\quad\times\left\|\left(\sqrt{\delta M_\varepsilon}\phi_\varepsilon,\sqrt{1-\delta M_\varepsilon}\psi_\varepsilon\right)\right\|_{L^\infty(dt;L^\infty(dxdv))}^{1-\frac 2p}
			\\
			&\lesssim \varepsilon^{-\frac\gamma p+\frac 2p -1},
		\end{aligned}
	\end{equation*}
	which completes the justification of the corollary.
\end{proof}

\subsection{Dilated fluctuations}

The preceding analysis of the relative entropy has provided bounds on the fluctuations $g_\varepsilon(t,x,v)$. We are now going to derive corresponding bounds on the dilated fluctuations $\widetilde g_\varepsilon (t,x,u,\omega)$, as defined in \eqref{dilated:fluctuations}. These new bounds are crucial for understanding the asymptotic energy spectrum of a degenerate Fermi gas.

\begin{prop}\label{prop:relative:entropy:dilated}
	Consider a family of density distributions $0\leq f_\varepsilon(t,x,v)\leq\delta^{-1}$, with $\varepsilon>0$, such that the relative entropy bound \eqref{relative:entropy:1} holds uniformly in $\varepsilon$, with some fixed parameter $\gamma>0$, where the Fermi--Dirac distribution $M_\varepsilon$ defined in \eqref{normalized:distribution} has a given parameter value $\tau>0$.
	Further consider the density fluctuations $g_\varepsilon(t,x,v)$ defined in \eqref{epsilon:fluctuations}, and its dilations $\widetilde g_\varepsilon (t,x,u,\omega)$ given in \eqref{dilated:fluctuations}.
	
	Then, as $\varepsilon\to 0$:
	\begin{itemize}
		
		\item
		If $\gamma+\tau\leq 2$, then any subsequence of dilated fluctuations $\widetilde g_\varepsilon$ satisfies the bound
		\begin{equation}\label{entropy:bound:1:dilated}
			(1+|u|)\widetilde g_\varepsilon
			=
			O(\varepsilon^{\frac{\tau-\gamma}2})_{L^\infty(dt;L^1_\mathrm{loc}(dx;L^1(dud\omega)))}.
		\end{equation}
		
		\item
		If $\gamma+\tau< 2$, then any subsequence of dilated fluctuations $\widetilde g_\varepsilon$ is tight in the energy variable, in the sense that, for any compact subset $K\subset\mathbb{R}^d$, it holds that
		\begin{equation}\label{concentrated:tightness:dilated}
			\lim_{\lambda\to\infty}\sup_{\varepsilon>0,t>0}\int_{K\times\mathbb{R}\times\partial B(0,R)}\varepsilon^{\frac{\gamma-\tau}2}\big|\widetilde g_\varepsilon\big||u|
			\mathds{1}_{\left\{|u|\geq \lambda \right\}}dxdud\omega
			=0.
		\end{equation}
		
		\item
		For any choice of parameters $\gamma>0$ and $\tau>0$, any subsequence of dilated fluctuations $\widetilde g_\varepsilon$ satisfies the bound
		\begin{equation}\label{fluctuation:H:2:dilated}
			\sup_{t>0}\int_{\mathbb{R}^d\times\mathbb{R}\times\partial B(0,R)}\left(\widetilde g_\varepsilon\right)^2
			\frac {1+|u|}{(R^2+\varepsilon^\tau u)^\frac{d-2}{2}} dx dud\omega
			\lesssim \varepsilon^{\tau-\gamma}.
		\end{equation}
	\end{itemize}
\end{prop}

\begin{proof}
	The change of variable formula \eqref{dilation:change:variable} allows us to deduce the bounds \eqref{entropy:bound:1:dilated}, \eqref{concentrated:tightness:dilated} and \eqref{fluctuation:H:2:dilated} directly from Proposition \ref{prop:relative:entropy}, through a careful calculation. This completes the proof.
\end{proof}

The uniform bounds from the previous proposition allows us to infer weak compactness estimates on subsequences of dilated fluctuations. This is presented in the next result.

\begin{prop}\label{prop:limit:characterization:dilated}
	Consider a family of density distributions $0\leq f_\varepsilon(t,x,v)\leq\delta^{-1}$, with $\varepsilon>0$, such that the relative entropy bound \eqref{relative:entropy:1} holds uniformly in $\varepsilon$, with some fixed parameter $\gamma>0$, where the Fermi--Dirac distribution $M_\varepsilon$ defined in \eqref{normalized:distribution} has a given parameter value $\tau>0$.
	Further suppose that
	\begin{equation*}
		0<\gamma=\tau< 1
	\end{equation*}
	and consider the density fluctuations $g_\varepsilon(t,x,v)$ defined in \eqref{epsilon:fluctuations}, and its dilations $\widetilde g_\varepsilon (t,x,u,\omega)$ given in \eqref{dilated:fluctuations}.
	
	In accordance with Proposition \ref{prop:limit:characterization}, as $\varepsilon\to 0$, up to extraction of a subsequence, the family of fluctuations $g_\varepsilon$ converges in the weak* topology of $\mathcal{M}_\mathrm{loc}\big(\mathbb{R}^+\times\mathbb{R}^d\times\mathbb{R}^d\big)$ toward a limit point
	\begin{equation*}
		\mu(t,x,v)=g\left(t,x,R\frac{v}{|v|}\right)dt\otimes dx\otimes \delta_{\partial B(0,R)}(v),
	\end{equation*}
	where the density $g(t,x,\omega)$ belongs to $L^\infty\big(dt;L^2\big(\mathbb{R}^d\times\partial B(0,R)\big)\big)$.
	
	Then, as $\varepsilon\to 0$:
	\begin{itemize}
		\item
		Any subsequence of dilated fluctuations $(1+|u|)\widetilde g_\varepsilon$ is uniformly bounded in
		\begin{equation*}
			L^\infty\big(dt;L^1_\mathrm{loc}\big(dx;L^1\big(dud\omega\big)\big)\big),
		\end{equation*}
		and weakly relatively compact in the space
		\begin{equation*}
			L^1_\mathrm{loc}\big(dtdx;L^1\big(dud\omega\big)\big).
		\end{equation*}
		
		\item
		If $\widetilde g$ is a weak limit point in $L^1_\mathrm{loc}\big(dtdx;L^1\big((1+|u|)dud\omega\big)\big)$ of the family of dilated fluctuations $\widetilde g_\varepsilon$, then it is related to the limiting density $g$ through the identity
		\begin{equation}\label{link:dilated}
			\int_{\mathbb{R}}\widetilde g(t,x,u,\omega)du=g(t,x,\omega).
		\end{equation}
		Moreover, it satisfies, for almost every $t\geq 0$,
		\begin{equation}\label{square:integrable:0:dilated}
			\int_{\mathbb{R}^d \times\mathbb{R}\times\partial B(0,R)}
			\left(\widetilde g(t,x,u,\omega)\cosh\left(\frac u2\right)\right)^2
			dxdud\omega
			\leq
			\frac{C^\mathrm{in}}{4R\delta^2},
		\end{equation}
		and, up to extraction of a subsequence, it holds that
		\begin{equation}\label{convergence:fluctuations:dilated}
			\begin{aligned}
				\int_{\mathbb{R}^+\times\mathbb{R}^d\times\mathbb{R}\times\partial B(0,R)}
				&\widetilde g_\varepsilon(t,x,u,\omega)\varphi(t,x,u,v)dtdxdud\omega
				\\
				&\to
				\int_{\mathbb{R}^+\times\mathbb{R}^d\times\mathbb{R}\times\partial B(0,R)}
				\widetilde g(t,x,u,\omega)\varphi(t,x,u,\omega)dtdxdud\omega,
			\end{aligned}
		\end{equation}
			as $\varepsilon\to 0$, for any continuous function $\varphi(t,x,u,v)$ which is compactly supported in $(t,x)$ and satisfies the growth condition
		\begin{equation}\label{growth:condition}
			\frac{\varphi(t,x,u,v)}{1+|u|}\in L^\infty\big(\mathbb{R}^+\times\mathbb{R}^d\times\mathbb{R}\times\mathbb{R}^d\big).
		\end{equation}
	\end{itemize}
\end{prop}

\begin{proof}
	We first note that the uniform boundedness and weak relative compactness of $(1+|u|)\widetilde g_\varepsilon$ follow directly from the bounds established in Proposition \ref{prop:relative:entropy:dilated} and the Dunford--Pettis compactness criterion. Specifically, the estimates contained in \eqref{entropy:bound:1:dilated}, \eqref{concentrated:tightness:dilated} and \eqref{fluctuation:H:2:dilated} provide the necessary boundedness, tightness and equi-integrability, respectively, for the application of this criterion.
	
	Next, taking the weak limits of convergent subsequences of fluctuations, as $\varepsilon\to 0$, we find that
	\begin{equation*}
		\begin{aligned}
			\int_{\mathbb{R}^+\times\mathbb{R}^d\times\mathbb{R}^d}
			g_\varepsilon(t,x,v)\varphi(t,x,\omega)dtdxdv
			&=\int_{\mathbb{R}^+\times\mathbb{R}^d\times\mathbb{R}\times\partial B(0,R)}
			\widetilde g_\varepsilon(t,x,u,\omega)\varphi(t,x,\omega)dtdxdud\omega
			\\
			&\to\int_{\mathbb{R}^+\times\mathbb{R}^d\times\mathbb{R}\times\partial B(0,R)}
			\widetilde g(t,x,u,\omega)\varphi(t,x,\omega)dtdxdud\omega,
		\end{aligned}
	\end{equation*}
	for any compactly supported continuous test function $\varphi(t,x,\omega)$. Combining this limit with \eqref{limit:fluctuations}, we conclude, by the arbitrariness of $\varphi$, that \eqref{link:dilated} holds true.
	
	One can also take weak limits against continuous test functions $\varphi(t,x,u,\omega)$ which are compactly supported in $(t,x)$ and satisfy the growth condition
	\begin{equation*}
		\frac{\varphi(t,x,u,\omega)}{1+|u|}\in L^\infty\big(\mathbb{R}^+\times\mathbb{R}^d\times\mathbb{R}\times\partial B(0,R)\big),
	\end{equation*}
	which yields
	\begin{equation*}
		\begin{aligned}
			\int_{\mathbb{R}^+\times\mathbb{R}^d\times\mathbb{R}\times\partial B(0,R)}
			&\widetilde g_\varepsilon(t,x,u,\omega)\varphi(t,x,u,\omega)dtdxdud\omega
			\\
			&\to
			\int_{\mathbb{R}^+\times\mathbb{R}^d\times\mathbb{R}\times\partial B(0,R)}
			\widetilde g(t,x,u,\omega)\varphi(t,x,u,\omega)dtdxdud\omega,
		\end{aligned}
	\end{equation*}
	as $\varepsilon\to 0$. Moreover, considering a continuous test functions $\varphi(t,x,u,v)$, compactly supported in $(t,x)$, with the growth condition \eqref{growth:condition}, we see that
	\begin{equation*}
		\int_{\mathbb{R}^+\times\mathbb{R}^d\times\mathbb{R}\times\partial B(0,R)}
		\widetilde g_\varepsilon(t,x,u,\omega)\big(\varphi(t,x,u,v)-\varphi(t,x,u,\omega)\big)dtdxdud\omega\to 0,
	\end{equation*}
	by tightness of $\widetilde g_\varepsilon$ and uniform continuity of $\varphi$ on compact domains. Therefore, by combining the preceding limits, we conclude that the convergence \eqref{convergence:fluctuations:dilated} also holds true.

	In order to show now that $\widetilde g$ is square-integrable and satisfies the bound \eqref{square:integrable:0:dilated}, we follow and adapt the strategy of proof of \eqref{square:integrable:0}. Specifically, we set
	\begin{equation*}
		z=\delta\varepsilon g_\varepsilon,
		\quad
		a=\delta M_\varepsilon
		\quad\text{and}\quad
		y=\varepsilon^{1-\tau}\lambda \varphi\left(t,x,u,\omega\right),
	\end{equation*}
	for some $\varphi(t,x,u,\omega)\in C_c(\mathbb{R}^+\times\mathbb{R}^d\times\mathbb{R}\times\partial B(0,R))$ and $\lambda>0$, in \eqref{inequality:Young:3}.
	Then, integrating the resulting inequality in $t$, $x$ and $v$, we obtain, for any $0\leq t_1\leq t_2$, that
	\begin{equation*}
		\begin{aligned}
			\delta\lambda & \int_{[t_1,t_2]\times \mathbb{R}^d \times\mathbb{R}\times\partial B(0,R)} \widetilde g_\varepsilon\varphi dtdxdud\omega
			\\
			&=\delta\lambda\int_{[t_1,t_2]\times \mathbb{R}^d \times\mathbb{R}^d} g_\varepsilon\varphi dtdxdv
			\\
			&\leq
			\frac 1{\varepsilon^{2-\tau}}\int_{[t_1,t_2]\times \mathbb{R}^d \times\mathbb{R}^d}h(\delta f_\varepsilon,\delta M_\varepsilon)dtdxdv
			\\
			&\quad +\frac 1{\varepsilon^{2-\tau}}\int_{[t_1,t_2]\times \mathbb{R}^d \times\mathbb{R}^d}\left[\log\big(1+\delta M_\varepsilon(e^{\varepsilon^{1-\tau}\lambda\varphi}-1)\big)-\delta M_\varepsilon\varepsilon^{1-\tau}\lambda\varphi\right]dtdxdv
			\\
			&\leq (t_2-t_1)C^\mathrm{in}+\int_{[t_1,t_2]\times \mathbb{R}^d\times\mathbb{R}\times\partial B(0,R)} I_\varepsilon\big(\lambda \varphi(t,x,u,\omega)\big)
			dtdxdud\omega,
		\end{aligned}
	\end{equation*}
	where
	\begin{equation*}
		I_\varepsilon(y)=\frac 1{2R^{d-1}\varepsilon^{2(1-\tau)}}
		\left[\log\left(1+\frac{e^{\varepsilon^{1-\tau} y}-1}{1+e^{u}}\right)
		-\frac{\varepsilon^{1-\tau} y}{1+e^{u}}\right]
		(R^2+\varepsilon^\tau u)^\frac{d-2}2.
	\end{equation*}
	We may assume that $u>-\frac{R^2}{\varepsilon^\tau}$ in the definition of $I_\varepsilon$, because $\varphi$ is compactly supported and $\varepsilon$ is small. Then, letting $\varepsilon\to 0$ yields that
	\begin{equation}\label{limit:entropy:1:dilated}
		\begin{aligned}
			\delta\lambda & \int_{[t_1,t_2]\times \mathbb{R}^d \times\mathbb{R}\times\partial B(0,R)} \widetilde g(t,x,u,\omega)\varphi(t,x,u,\omega) dtdxdud\omega
			\\
			&\leq
			(t_2-t_1)C^\mathrm{in}
			+\liminf_{\varepsilon\to 0}
			\int_{[t_1,t_2]\times \mathbb{R}^d\times\mathbb{R}\times\partial B(0,R)} I_\varepsilon\big(\lambda \varphi(t,x,u,\omega)\big)
			dtdxdud\omega,
		\end{aligned}
	\end{equation}
	by weak convergence. Note that the main difference between \eqref{limit:entropy:1} and \eqref{limit:entropy:1:dilated} lies in the fact that $\varphi$ is now allowed to depend on the variable $u$.
	
	Next, by virtue of \eqref{dominated:estimate}, we deduce that $I_\varepsilon\big(\lambda \varphi(t,x,u,\omega)\big)$ is dominated by an integrable function, which implies, by the Dominated Convergence Theorem, that
	\begin{equation*}
		\begin{aligned}
			\delta\lambda & \int_{[t_1,t_2]\times \mathbb{R}^d \times\mathbb{R}\times\partial B(0,R)} \widetilde g(t,x,u,\omega)\varphi(t,x,u,\omega) dtdxdud\omega
			\\
			&\leq
			(t_2-t_1)C^\mathrm{in}
			+\frac{\lambda^2}{16R}\int_{[t_1,t_2]\times \mathbb{R}^d\times\mathbb{R}\times\partial B(0,R)}
			\left(\frac{\varphi(t,x,u,\omega)}{\cosh(\frac u2)}\right)^2
			dtdxdud\omega,
		\end{aligned}
	\end{equation*}
	where we employed \eqref{special:limit} to evaluate the pointwise limit of $I_\varepsilon\big(\lambda \varphi(t,x,u,\omega)\big)$.
	
	Finally, by optimizing the value of $\lambda>0$ and substituting $\varphi(t,x,u,\omega)$ with the compactly supported continuous function $\varphi(t,x,u,\omega)\cosh(\frac u2)$, we obtain that
	\begin{equation*}
		\begin{aligned}
			\delta \int_{[t_1,t_2]\times \mathbb{R}^d \times\mathbb{R}\times\partial B(0,R)}
			\widetilde g(t,x,u,\omega)\cosh\left(\frac u2\right)
			\varphi(t,x,u,\omega) &dtdxdud\omega
			\\
			&\leq
			\left(\frac{(t_2-t_1)C^\mathrm{in}}{4R}\right)^\frac 12
			\|\varphi\|_{L^2(dtdxdud\omega)}.
		\end{aligned}
	\end{equation*}
	At last, taking the supremum over all $\varphi$, we arrive at the bound
	\begin{equation*}
		\int_{[t_1,t_2]\times \mathbb{R}^d \times\mathbb{R}\times\partial B(0,R)}
		\left(\widetilde g(t,x,u,\omega)\cosh\left(\frac u2\right)\right)^2
		dtdxdud\omega
		\leq
		\frac{(t_2-t_1)C^\mathrm{in}}{4R\delta^2},
	\end{equation*}
	which, by the arbitrariness of $t_1$ and $t_2$, gives \eqref{square:integrable:0:dilated} and completes the proof of the proposition.
\end{proof}

\section{Relaxation estimates and thermodynamic equilibria}\label{section:relaxation_and_equilibria}

In this section, we develop the analytical tools necessary to characterize the structure of limiting fluctuations. Our primary objective is to derive the equilibrium form of weak limits of fluctuations, as $\varepsilon \to 0$, by exploiting the entropy dissipation bound.

\subsection{Controls from the entropy dissipation bound}

The entropy inequality \eqref{entropy:inequality:scaled} and the initial relative entropy control \eqref{initial:relative:entropy} provide us with the following uniform entropy dissipation bound
\begin{equation}\label{entropy:dissipation:1}
	\frac 1{\varepsilon^{2+\kappa-\gamma}}
	\int_0^tD(f_\varepsilon)(s)ds
	\leq C^\mathrm{in},
\end{equation}
for some $\gamma>0$, $\kappa>0$, $C^\mathrm{in}>0$ and all $t\geq 0$, where the entropy dissipation $D(f_\varepsilon)$ is given in \eqref{entropy:dissipation}.

The following result extracts a simple quadratic control on the renormalized collision integrands
\begin{equation}\label{collision:integrand}
	q_\varepsilon=\frac 1{\delta^\frac 12 \varepsilon^{1+\frac{\kappa-\gamma}2}}\left(\sqrt{\delta f_\varepsilon' \delta f_{\varepsilon *}'(1-\delta f_\varepsilon)(1-\delta f_{\varepsilon *})}
	-
	\sqrt{\delta f_\varepsilon \delta f_{\varepsilon *}(1-\delta f_\varepsilon')(1-\delta f_{\varepsilon *}')}\right)
\end{equation}
from the entropy dissipation bound. The use of renormalized collision integrands is standard in the analysis of hydrodynamic limits (see \cite[Section 5.2]{asr19}, for instance, for an application of a similar method). The ensuing control, when employed in conjunction with the quadratic estimates from Lemma \ref{lemma:squareroot:relative:entropy}, will allow us to establish, in Section \ref{section:relaxation}, the relaxation of fluctuations toward thermodynamic equilibrium, which is a fundamental step in our work.

\begin{lem}\label{lemma:renormalized:collision:bound}
	Consider a family of density distributions $0\leq f_\varepsilon(t,x,v)\leq\delta^{-1}$, with $\varepsilon>0$, such that the entropy dissipation bound \eqref{entropy:dissipation:1} holds uniformly in $\varepsilon$, with some fixed parameters $\gamma>0$ and $\kappa>0$.
	
	Then, as $\varepsilon\to 0$, any subsequence of renormalized collision integrands $q_\varepsilon$ is uniformly bounded in $L^2(b(v-v_*,\sigma)dtdxdvdv_*d\sigma)$.
\end{lem}

\begin{proof}
	The elementary inequality (see \cite[Appendix B]{asr19} for a brief justification)
	\begin{equation*}
		4(\sqrt y - \sqrt z)^2\leq (y-z)\log \left(\frac yz\right),
	\end{equation*}
	which holds for any $y,z>0$, yields that
	\begin{equation*}
		\varepsilon^{2+\kappa-\gamma}
		\int_{\mathbb{R}^d}\int_{\mathbb{R}^d\times\mathbb{R}^d\times\mathbb{S}^{d-1}}
		q_\varepsilon^2 b(v-v_*,\sigma)dvdv_*d\sigma dx
		\leq D(f_\varepsilon).
	\end{equation*}
	It then follows from the entropy dissipation bound \eqref{entropy:dissipation:1} that
	\begin{equation*}
		\int_0^t\int_{\mathbb{R}^d}
		\int_{\mathbb{R}^d\times\mathbb{R}^d\times\mathbb{S}^{d-1}}
		q_\varepsilon(s)^2 b(v-v_*,\sigma)dvdv_*d\sigma dx ds
		\leq C^\mathrm{in},
	\end{equation*}
	for all $t\geq 0$. This establishes the uniform boundedness of renormalized collision integrands and, thus, completes the proof of the lemma.
\end{proof}

In the next lemma, we extract a linearized control from the uniform bound on renormalized collision integrands given in Lemma \ref{lemma:renormalized:collision:bound}. This constitutes the first step in our work where the case $d=2$ needs to be excluded. It is treated separately in our companion article \cite{aa25}.

\begin{lem}\label{lemma:linearized:dissipation}
	Consider a cross-section
	\begin{equation*}
		b(z,\sigma)=b\left(|z|,\frac{z}{|z|}\cdot\sigma\right)\geq 0
	\end{equation*}
	such that
	\begin{equation*}
		b(z,\sigma)\in L^\infty(\mathbb{R}^d\times\mathbb{S}^{d-1})
	\end{equation*}
	and a family of density distributions $0\leq f_\varepsilon(t,x,v)\leq\delta^{-1}$, with $\varepsilon>0$, such that the relative entropy bound \eqref{relative:entropy:1} and the entropy dissipation bound \eqref{entropy:dissipation:1} hold uniformly in $\varepsilon$, with some fixed parameters $\gamma>0$ and $\kappa>0$, where the Fermi--Dirac distribution $M_\varepsilon$ defined in \eqref{normalized:distribution} has a given parameter value $\tau>0$. Further consider the density fluctuations $\phi_\varepsilon(t,x,v)$ and $\psi_\varepsilon(t,x,v)$ defined in \eqref{epsilon:renormalized:fluctuations}.
	For convenience of notation, we also introduce
	\begin{equation}\label{maxwellian:product:notation}
		m_\varepsilon(v,v_*,\sigma)=
		\delta M_\varepsilon \delta M_{\varepsilon *}(1-\delta M_\varepsilon')(1-\delta M_{\varepsilon *}')
		=
		\delta M_\varepsilon' \delta M_{\varepsilon *}'(1-\delta M_\varepsilon)(1-\delta M_{\varepsilon *}).
	\end{equation}
	
	Then, if $d\geq 3$, as $\varepsilon\to 0$, any joint subsequences of renormalized fluctuations $\phi_\varepsilon$ and $\psi_\varepsilon$ satisfy the bounds
	\begin{equation*}
		\begin{aligned}
			(\phi_\varepsilon+\psi_\varepsilon)+(\phi_\varepsilon+\psi_\varepsilon)_*
			-(\phi_\varepsilon+\psi_\varepsilon)'-(\phi_\varepsilon+\psi_\varepsilon)_*'
			&=O(\varepsilon^{1+2\tau-\gamma})_{L^\infty(dt;L^1(bm_\varepsilon dxdvdv_*d\sigma))}
			\\
			&\quad+O(\varepsilon^{\frac{\kappa-\gamma+3\tau}2})_{L^2(dtdx;L^1(bm_\varepsilon dvdv_*d\sigma))}.
		\end{aligned}
	\end{equation*}
\end{lem}

\begin{rem}
	As previously emphasized, in virtue of Proposition \ref{prop:limit:characterization}, if $\mu$ is a limit point of a converging subsequence of fluctuations $g_\varepsilon$, then $\delta M_\varepsilon\phi_\varepsilon$, $\left(1-\delta M_\varepsilon\right)\psi_\varepsilon$ and $\delta M_\varepsilon\left(1-\delta M_\varepsilon\right)\left(\phi_\varepsilon+\psi_\varepsilon\right)$ also converge toward $\mu$, in a suitable sense, as $\varepsilon\to 0$. In particular, recall that $\mu$ is characterized as a singular measure in $v$, which is concentrated on a sphere. Therefore, it is only natural that the uniform bounds from the previous lemma be set in the functional setting of integrable functions in velocity. Moreover, we do not expect that such bounds can be improved to spaces $L^p(bm_\varepsilon dvdv_*d\sigma)$, with $p>1$, in any meaningful way.
\end{rem}

\begin{proof}
	By expanding the renormalized collision integrands $q_\varepsilon$ given in \eqref{collision:integrand} in terms of the renormalized fluctuations $\phi_\varepsilon$ and $\psi_\varepsilon$, we obtain the decomposition
	\begin{equation}\label{linearization:decomposition}
		\begin{aligned}
			\frac{\delta^\frac 12\varepsilon^{1+\frac{\kappa-\gamma}2}q_\varepsilon}{\sqrt{m_\varepsilon}}
			&=
			\left(1+\frac{\delta\varepsilon}2 \phi_\varepsilon'\right)
			\left(1+\frac{\delta\varepsilon}2 \phi_{\varepsilon *}'\right)
			\left(1-\frac{\delta\varepsilon}2 \psi_\varepsilon\right)
			\left(1-\frac{\delta\varepsilon}2 \psi_{\varepsilon *}\right)
			\\
			&\quad-
			\left(1+\frac{\delta\varepsilon}2 \phi_\varepsilon\right)
			\left(1+\frac{\delta\varepsilon}2 \phi_{\varepsilon *}\right)
			\left(1-\frac{\delta\varepsilon}2 \psi_\varepsilon'\right)
			\left(1-\frac{\delta\varepsilon}2 \psi_{\varepsilon *}'\right)
			\\
			&=
			-\frac{\delta\varepsilon}2\mathcal{L}_\varepsilon
			+\left(\frac{\delta\varepsilon}2\right)^2\mathcal{Q}_\varepsilon
			+\left(\frac{\delta\varepsilon}2\right)^3\mathcal{C}_\varepsilon
			+\left(\frac{\delta\varepsilon}2\right)^4\mathcal{T}_\varepsilon,
		\end{aligned}
	\end{equation}
	where
	\begin{equation*}
		\begin{aligned}
			\mathcal{L}_\varepsilon&=
			(\phi_\varepsilon+\psi_\varepsilon)+(\phi_\varepsilon+\psi_\varepsilon)_*
			-(\phi_\varepsilon+\psi_\varepsilon)'-(\phi_\varepsilon+\psi_\varepsilon)_*',
			\\
			\mathcal{Q}_\varepsilon&=
			\phi_{\varepsilon}'\phi_{\varepsilon*}'
			-\phi_{\varepsilon}'\psi_{\varepsilon}
			-\phi_{\varepsilon}'\psi_{\varepsilon*}
			-\phi_{\varepsilon*}'\psi_{\varepsilon}
			-\phi_{\varepsilon*}'\psi_{\varepsilon*}
			+\psi_{\varepsilon}\psi_{\varepsilon*}
			\\
			&\quad
			-\phi_{\varepsilon}\phi_{\varepsilon*}
			+\phi_{\varepsilon}\psi_{\varepsilon}'
			+\phi_{\varepsilon}\psi_{\varepsilon*}'
			+\phi_{\varepsilon*}\psi_{\varepsilon}'
			+\phi_{\varepsilon*}\psi_{\varepsilon*}'
			-\psi_{\varepsilon}'\psi_{\varepsilon*}',
			\\
			\mathcal{C}_\varepsilon&=
			-\phi_{\varepsilon}'\phi_{\varepsilon*}'(\psi_{\varepsilon}+\psi_{\varepsilon*})
			+(\phi_{\varepsilon}'+\phi_{\varepsilon*}')\psi_{\varepsilon}\psi_{\varepsilon*}
			+\phi_{\varepsilon}\phi_{\varepsilon*}(\psi_{\varepsilon}'+\psi_{\varepsilon*}')
			-(\phi_{\varepsilon}+\phi_{\varepsilon*})\psi_{\varepsilon}'\psi_{\varepsilon*}',
			\\
			\mathcal{T}_\varepsilon&=
			\phi_{\varepsilon}'\phi_{\varepsilon*}'\psi_{\varepsilon}\psi_{\varepsilon*}
			-\phi_{\varepsilon}\phi_{\varepsilon*}\psi_{\varepsilon}'\psi_{\varepsilon*}'.
		\end{aligned}
	\end{equation*}
	Our goal is now to control the terms $\mathcal{Q}_\varepsilon$, $\mathcal{C}_\varepsilon$ and $\mathcal{T}_\varepsilon$ individually by employing the bounds from Lemma \ref{lemma:squareroot:relative:entropy}, which, when combined with the bound on $q_\varepsilon$ from Lemma \ref{lemma:renormalized:collision:bound}, will yield a suitable control of the linear term $\mathcal{L}_\varepsilon$.
	More precisely, we are going to show that
	\begin{equation}\label{nonlinear:bound:1}
		\int_{\mathbb{R}^d\times\mathbb{R}^d\times\mathbb{S}^{d-1}}
		\big(
		|\mathcal{Q}_\varepsilon|
		+\varepsilon|\mathcal{C}_\varepsilon|
		+\varepsilon^2|\mathcal{T}_\varepsilon|
		\big)bm_\varepsilon dvdv_*d\sigma
		=O(\varepsilon^{2\tau-\gamma})_{L^\infty(dt;L^1(dx))},
	\end{equation}
	by relying on the assumption that $d\geq 3$.
	
	For mere convenience of notation, we introduce now the functions
	\begin{equation*}
		\Phi_\varepsilon=\sqrt{\delta M_\varepsilon}\phi_\varepsilon
		\qquad\text{and}\qquad
		\Psi_\varepsilon=\sqrt{1-\delta M_\varepsilon}\psi_\varepsilon,
	\end{equation*}
	which we only employ in this proof. In particular, by Corollary \ref{cor:squareroot:relative:entropy}, notice that
	\begin{equation}\label{nonlinear:bound:2}
		\Phi_\varepsilon,\Psi_\varepsilon=O(\varepsilon^{\frac {2-\gamma}p-1})_{L^\infty(dt;L^p(dxdv))},
	\end{equation}
	for all $2\leq p\leq\infty$.
	
	Therefore, in order to show \eqref{nonlinear:bound:1}, we observe, employing the usual collisional symmetries \eqref{change:prepost} and \eqref{change:exchange}, that it is sufficient to establish that
	\begin{equation*}
		\int_{\mathbb{R}^d\times\mathbb{R}^d\times\mathbb{S}^{d-1}}
		\big(|\Phi_{\varepsilon}\Phi_{\varepsilon*}|
		+|\Phi_{\varepsilon}\Psi_{\varepsilon}'|
		+|\Phi_{\varepsilon}\Psi_{\varepsilon*}'|
		+|\Psi_{\varepsilon}\Psi_{\varepsilon*}|\big)
		b \sqrt{m_\varepsilon}
		dvdv_*d\sigma=O(\varepsilon^{2\tau-\gamma})_{L^\infty(dt;L^1(dx))}.
	\end{equation*}	
	Then, further utilizing the classical Young inequality for products, we see that \eqref{nonlinear:bound:1} is a consequence of the asymptotic bound
	\begin{equation*}
		\int_{\mathbb{R}^d\times\mathbb{R}^d\times\mathbb{S}^{d-1}}
		\big(\Phi_{\varepsilon}^2
		+\Psi_{\varepsilon}^2\big)
		b \sqrt{m_\varepsilon}
		dvdv_*d\sigma=O(\varepsilon^{2\tau-\gamma})_{L^\infty(dt;L^1(dx))}.
	\end{equation*}	
	Therefore, in view of \eqref{nonlinear:bound:2}, with $p=2$, we conclude that \eqref{nonlinear:bound:1} is a direct consequence of the control
	\begin{equation}\label{attenuation:coefficient:1}
		\int_{\mathbb{R}^d\times\mathbb{S}^{d-1}}
		b \sqrt{m_\varepsilon}
		dv_*d\sigma=O(\varepsilon^{2\tau})_{L^\infty(dv)}.
	\end{equation}
	The justification of \eqref{attenuation:coefficient:1} is nontrivial and requires great care. In particular, it fails in the case $d=2$. Thus, for the sake of clarity, it is deferred to Lemma \ref{lemma:attenuation:coefficient}, below, where we provide a precise statement with a complete proof of this asymptotic bound. For the moment, admitting that \eqref{attenuation:coefficient:1} holds true, we have completed the proof of \eqref{nonlinear:bound:1}.

	Now, by further combining \eqref{attenuation:coefficient:1} with the fact that
	\begin{equation*}
		m_\varepsilon^2
		=
		\big(\delta M_\varepsilon(1-\delta M_\varepsilon)\big)
		\big(\delta M_\varepsilon(1-\delta M_\varepsilon)\big)_*
		\big(\delta M_\varepsilon(1-\delta M_\varepsilon)\big)'
		\big(\delta M_\varepsilon(1-\delta M_\varepsilon)\big)_*'
	\end{equation*}
	and the basic asymptotic control \eqref{basic:asymptotics}, we see that
	\begin{equation*}
		\int_{\mathbb{R}^d\times\mathbb{R}^d\times\mathbb{S}^{d-1}}
		b m_\varepsilon
		dvdv_*d\sigma=O(\varepsilon^{3\tau})_{L^\infty(dv)}.
	\end{equation*}
	Therefore, by virtue of Lemma \ref{lemma:renormalized:collision:bound}, we conclude that
	\begin{equation*}
		\begin{aligned}
			\int_{\mathbb{R}^d\times\mathbb{R}^d\times\mathbb{S}^{d-1}}
			|q_\varepsilon|b\sqrt{m_\varepsilon}dvdv_*d\sigma
			&\leq
			\left(\int_{\mathbb{R}^d\times\mathbb{R}^d\times\mathbb{S}^{d-1}}
			q_\varepsilon^2 b dvdv_*d\sigma\right)^\frac 12
			\left(\int_{\mathbb{R}^d\times\mathbb{R}^d\times\mathbb{S}^{d-1}}
			bm_\varepsilon dvdv_*d\sigma\right)^\frac 12
			\\
			&=O(\varepsilon^{\frac{3\tau}2})_{L^2(dtdx)},
		\end{aligned}
	\end{equation*}
	in any dimension $d\geq 3$.
	
	Finally, incorporating \eqref{nonlinear:bound:1} and the preceding bound into \eqref{linearization:decomposition}, we deduce that
	\begin{equation*}
		\int_{\mathbb{R}^d\times\mathbb{R}^d\times\mathbb{S}^{d-1}}
		|\mathcal{L}_\varepsilon|
		bm_\varepsilon dvdv_*d\sigma
		=O(\varepsilon^{1+2\tau-\gamma})_{L^\infty(dt;L^1(dx))}
		+O(\varepsilon^{\frac{\kappa-\gamma+3\tau}2})_{L^2(dtdx)},
	\end{equation*}
	for any $d\geq 3$, which completes the proof of the lemma.
\end{proof}

The next lemma provides a precise justification of \eqref{attenuation:coefficient:1}.

\begin{lem}\label{lemma:attenuation:coefficient}
	Let $d\geq 3$ and consider a cross-section
	\begin{equation*}
		b(z,\sigma)=b\left(|z|,\frac{z}{|z|}\cdot\sigma\right)\geq 0
	\end{equation*}
	such that
	\begin{equation*}
		b(z,\sigma)\in L^\infty(\mathbb{R}^d\times\mathbb{S}^{d-1}).
	\end{equation*}
	Then, for any $\alpha>0$, as $\varepsilon\to 0$, it holds that
	\begin{equation*}
		\int_{\mathbb{R}^d\times\mathbb{S}^{d-1}}
		b(v-v_*,\sigma) m_\varepsilon^\alpha
		dv_*d\sigma=O(\varepsilon^{2\tau})_{L^\infty(dv)},
	\end{equation*}
	where $m_\varepsilon$ is given in \eqref{maxwellian:product:notation}.
\end{lem}

\begin{proof}
	The first few steps of this proof are valid in any dimension $d\geq 2$. We will then clearly emphasize, below, the moment when the restriction $d\geq 3$ becomes necessary.
	
	We begin by carefully introducing suitable variables which will allow us to understand the precise asymptotic behavior of $bm_\varepsilon^\alpha$. Specifically, writing
	\begin{equation}\label{attenuation:coefficient:8}
		\begin{aligned}
			u&=\frac{v^2-R^2}{\varepsilon^\tau}\in\left [-\frac{R^2}{\varepsilon^\tau},\infty\right),
			&
			\nu&=\frac{v}{|v|}\in\mathbb{S}^{d-1},
			\\
			u_*&=\frac{v_*^2-R^2}{\varepsilon^\tau}\in\left [-\frac{R^2}{\varepsilon^\tau},\infty\right),
			&
			\nu_*&=\frac{v_*}{|v_*|}\in\mathbb{S}^{d-1},
		\end{aligned}
	\end{equation}
	we compute that
	\begin{equation*}
		\begin{aligned}
			\frac{v'^2-R^2}{\varepsilon^\tau}
			&=\frac{u+u_*}2+\frac{|v+v_*||v-v_*|}{2\varepsilon^\tau}\cos\eta,
			\\
			\frac{v_*'^2-R^2}{\varepsilon^\tau}
			&=\frac{u+u_*}2-\frac{|v+v_*||v-v_*|}{2\varepsilon^\tau}\cos\eta,
		\end{aligned}
	\end{equation*}
	where $\eta\in [0,\pi]$ denotes the angle between $v+v_*$ and $\sigma$.
	In particular, further introducing
	\begin{equation}\label{attenuation:coefficient:9}
		s=\frac{|v+v_*||v-v_*|}{2\varepsilon^\tau}\cos\eta
		\in\left[-\frac{|v+v_*||v-v_*|}{2\varepsilon^\tau},\frac{|v+v_*||v-v_*|}{2\varepsilon^\tau}\right],
	\end{equation}
	it is then readily seen that
	\begin{equation*}
		\begin{aligned}
			m_\varepsilon
			&=
			\sqrt{\big(\delta M_\varepsilon(1-\delta M_\varepsilon)\big)
			\big(\delta M_\varepsilon(1-\delta M_\varepsilon)\big)_*
			\big(\delta M_\varepsilon(1-\delta M_\varepsilon)\big)'
			\big(\delta M_\varepsilon(1-\delta M_\varepsilon)\big)_*'}
			\\
			&=\left[16\cosh\left(\frac u2\right)
			\cosh\left(\frac {u_*}2\right)
			\cosh\left(\frac {u+u_*}4+\frac s2\right)
			\cosh\left(\frac {u+u_*}4-\frac s2\right)\right]^{-1}
			\\
			&=\left[8\cosh\left(\frac u2\right)
			\cosh\left(\frac {u_*}2\right)
			\left(\cosh\left(\frac {u+u_*}2\right)
			+\cosh s \right)\right]^{-1}.
		\end{aligned}
	\end{equation*}
	Therefore, employing the above variables, we obtain that
	\begin{equation*}
		\begin{aligned}
			\int_{\mathbb{R}^d\times\mathbb{S}^{d-1}}
			&b(v-v_*,\sigma) m_\varepsilon^\alpha
			dv_*d\sigma
			\\
			&=\frac{\varepsilon^\tau}{2^{1+3\alpha}}
			\int_{-\frac{R^2}{\varepsilon^\tau}}^\infty
			\int_{\mathbb{S}^{d-1}}
			\int_{\mathbb{S}^{d-1}}
			\frac{b(v-v_*,\sigma)(R^2+\varepsilon^\tau u_*)^\frac{d-2}2}
			{\cosh^\alpha\left(\frac u2\right)
			\cosh^\alpha\left(\frac {u_*}2\right)
			\left(\cosh\left(\frac{u+u_*}2\right)+\cosh s\right)^\alpha}
			d\sigma d\nu_* du_*.
		\end{aligned}
	\end{equation*}
	At this stage, it is readily seen that
	\begin{equation*}
		\begin{aligned}
			\mathds{1}_{\{-\frac{R^2}{\varepsilon^\tau}\leq u < -\frac{R^2}{2\varepsilon^\tau}\}}
			&\int_{\mathbb{R}^d\times\mathbb{S}^{d-1}}
			b(v-v_*,\sigma) m_\varepsilon^\alpha
			dv_*d\sigma
			\\
			&\lesssim \varepsilon^\tau\|b\|_{L^\infty}
			\cosh^{-\alpha}\left(\frac {R^2}{4\varepsilon^\tau}\right)
			\int_{-\frac{R^2}{\varepsilon^\tau}}^\infty
			\cosh^{-\frac \alpha 2}\left(\frac {u_*}2\right)
			du_*
			=O(\varepsilon^{2\tau})_{L^\infty(dv)}.
		\end{aligned}
	\end{equation*}
	Thus, there only remains to consider values $u \geq -\frac{R^2}{2\varepsilon^\tau}$, which will be important toward the end of the proof, below.

	Next, we employ the spherical decomposition
	\begin{equation}\label{attenuation:coefficient:10}
		\sigma=\frac{v+v_*}{|v+v_*|}\cos\eta+\widetilde\sigma\sin\eta,
		\qquad \text{with }\widetilde\sigma\in\mathbb{S}^{d-2}\perp \frac{v+v_*}2,
	\end{equation}
	where $\mathbb{S}^{d-2}\perp \frac{v+v_*}2$ denotes the $(d-2)$-dimensional unit sphere in $\mathbb{R}^d$, centered at the origin, normal to $\frac{v+v_*}2$, embedded with its natural surface measure, i.e.,
	\begin{equation*}
		\left(\mathbb{S}^{d-2}\perp \frac{v+v_*}2\right)
		=
		\left\{
		\sigma\in\mathbb{R}^d : |\sigma|=1,\ \frac{v+v_*}2\cdot\sigma=0
		\right\}.
	\end{equation*}
	Then, noticing that
	\begin{equation*}
		d\sigma=\sin^{d-2}\eta d\eta d\widetilde\sigma
		=\frac{2\varepsilon^\tau \left(\left( |v+v_*||v-v_*| \right)^2-(2\varepsilon^\tau s)^2\right)^\frac{d-3}2}
		{\left(|v+v_*||v-v_*|\right)^{d-2}}
		dsd\widetilde\sigma,
	\end{equation*}
	we deduce that
	\begin{equation}\label{attenuation:coefficient:4}
		\begin{aligned}
			\int_{\mathbb{R}^d\times\mathbb{S}^{d-1}}
			&b(v-v_*,\sigma) m_\varepsilon^\alpha
			dv_*d\sigma
			\\
			&=\frac{\varepsilon^{2\tau}}{8^{\alpha}}
			\int_{-\frac{R^2}{\varepsilon^\tau}}^\infty
			\int_{\mathbb{S}^{d-1}}
			\int_{-\frac{|v+v_*||v-v_*|}{2\varepsilon^\tau}}^{\frac{|v+v_*||v-v_*|}{2\varepsilon^\tau}}
			\frac{\left(\int_{\mathbb{S}^{d-2}\perp\frac{v+v_*}{2}}
			b(v-v_*,\sigma)d\widetilde\sigma \right)
			(R^2+\varepsilon^\tau u_*)^\frac{d-2}2}
			{\cosh^\alpha\left(\frac u2\right)
			\cosh^\alpha\left(\frac {u_*}2\right)
			\left(\cosh\left(\frac{u+u_*}2\right)+\cosh s\right)^\alpha}
			\\
			&\quad\times\frac{ \left(\left( |v+v_*||v-v_*| \right)^2-(2\varepsilon^\tau s)^2\right)^\frac{d-3}2}
			{\left(|v+v_*||v-v_*|\right)^{d-2}}
			ds d\nu_* du_*.
		\end{aligned}
	\end{equation}
	It is to be emphasized that the preceding decomposition also holds in the two-dimensional case $d=2$, where $\mathbb{S}^{0}\perp\frac{v+v_*}{2}$ is made up of the only two unit vectors which are orthogonal to $\frac{v+v_*}{2}$ and $|\mathbb{S}^{0}\perp\frac{v+v_*}{2}|=2$.
	
	We have now reached the point where the case $d=2$ needs to be considered separately. Thus, from now on, we assume that $d\geq 3$. Recalling that the cross-section $b$ is uniformly bounded pointwise, it follows that
	\begin{equation*}
		\begin{aligned}
			\varepsilon^{-2\tau}\int_{\mathbb{R}^d\times\mathbb{S}^{d-1}}
			& b(v-v_*,\sigma)m_\varepsilon^\alpha
			dv_*d\sigma
			\\
			&\lesssim
			\int_{-\frac{R^2}{\varepsilon^\tau}}^\infty
			\int_{\mathbb{S}^{d-1}}
			\int_{-\frac{|v+v_*||v-v_*|}{2\varepsilon^\tau}}^{\frac{|v+v_*||v-v_*|}{2\varepsilon^\tau}}
			\frac{\cosh^{-\alpha}\left(\frac u2\right)
			\cosh^{-\frac\alpha 2}\left(\frac {u_*}2\right)
			\cosh^{-\alpha} (s)}
			{|v+v_*||v-v_*|}
			ds d\nu_* du_*
			\\
			&\lesssim
			\cosh^{-\alpha}\left(\frac u2\right)
			\int_{-\frac{R^2}{\varepsilon^\tau}}^\infty
			\cosh^{-\frac\alpha2}\left(\frac {u_*}2\right)
			\left(\int_{\mathbb{S}^{d-1}}
			\frac{1}
			{|v+v_*||v-v_*|}
			d\nu_*\right)
			du_*.
		\end{aligned}
	\end{equation*}
	Here, denoting the angle between $v$ and $v_*$ by $\zeta \in [0,\pi]$, we further notice that
	\begin{equation}\label{attenuation:coefficient:5}
		\begin{aligned}
			\int_{\mathbb{S}^{d-1}}
			\frac{1}{|v+v_*||v-v_*|}
			d\nu_*
			&=
			\int_{\mathbb{S}^{d-1}}
			\left((v^2+v_*^2)^2-(2v\cdot v_*)^2\right)^{-\frac 12}
			d\nu_*
			\\
			&=|\mathbb{S}^{d-2}|
			\int_0^\pi
			\left((v^2+v_*^2)^2-4v^2v_*^2\cos^2\zeta\right)^{-\frac 12}
			\sin^{d-2}\zeta d\zeta
			\\
			&=\frac{|\mathbb{S}^{d-2}|}{2|v||v_*|}
			\int_{-1}^1
			\left(\left(\frac{v^2+v_*^2}{2|v||v_*|}\right)^2-t^2\right)^{-\frac 12}
			(1-t^2)^{\frac{d-3}2} dt
			\\
			&\lesssim \left(|v||v_*|(v^2+v_*^2)\right)^{-\frac 12}
			\int_{0}^1
			\left(\left(\frac{v^2+v_*^2}{2|v||v_*|}-1\right)+t\right)^{-\frac 12}
			t^{\frac{d-3}2} dt
			\\
			&\lesssim
			\frac 1{v^2+v_*^2},
		\end{aligned}
	\end{equation}
	where we used that
	\begin{equation*}
		\int_0^1 \left(\lambda+ t\right)^{-\frac 12} t^{\frac{d-3}2} dt
		=\lambda^{\frac{d-2}2}\int_0^{\lambda^{-1}} \left(1+t\right)^{-\frac 12} t^{\frac{d-3}2} dt\lesssim
		(1+\lambda)^{-\frac 12},
	\end{equation*}
	for all $\lambda\in (0,\infty)$. All in all, combining the previous estimates, we deduce that
	\begin{equation*}
		\varepsilon^{-2\tau}\int_{\mathbb{R}^d\times\mathbb{S}^{d-1}}
		b(v-v_*,\sigma) m_\varepsilon^\alpha
		dv_*d\sigma
		\lesssim
		\int_{-\frac{R^2}{\varepsilon^\tau}}^\infty
		\frac{\cosh^{-\alpha}\left(\frac u2\right)
		\cosh^{-\frac \alpha2}\left(\frac {u_*}2\right)}
		{2R^2+\varepsilon^\tau(u+u_*)}
		du_*.
	\end{equation*}

	Finally, in order to control the last integral, above, we recall that we only need to consider values in the domain $\{v^2 \geq \frac{R^2}{2}\}=\{u \geq-\frac{R^2}{2\varepsilon^\tau}\}$ and then deduce that
	\begin{equation*}
		\begin{aligned}
			\mathds{1}_{\{u \geq -\frac{R^2}{2\varepsilon^\tau}\}}
			\int_{-\frac{R^2}{\varepsilon^\tau}}^\infty
			\frac{\cosh^{-\alpha}\left(\frac u2\right)
			\cosh^{-\frac\alpha 2}\left(\frac {u_*}2\right)}
			{2R^2+\varepsilon^\tau(u+u_*)}
			du_*
			&\leq\frac 2{R^2}
			\int_{-\frac{R^2}{\varepsilon^\tau}}^\infty
			\cosh^{-\alpha}\left(\frac u2\right)
			\cosh^{-\frac\alpha 2}\left(\frac {u_*}2\right)
			du_*
			\\
			&\lesssim\cosh^{-\alpha}\left(\frac u2\right).
		\end{aligned}
	\end{equation*}
	In conclusion, we have now obtained that
	\begin{equation*}
		\mathds{1}_{\{u \geq -\frac{R^2}{2\varepsilon^\tau}\}}\int_{\mathbb{R}^d\times\mathbb{S}^{d-1}}
		b(v-v_*,\sigma) \sqrt{m_\varepsilon}
		dv_*d\sigma
		\lesssim \varepsilon^{2\tau}\cosh^{-\alpha}\left(\frac u2\right),
	\end{equation*}
	which completes the proof of the lemma.
\end{proof}

\subsection{Collisional symmetries on the sphere}

Prior to addressing thermodynamic equilibria of density fluctuations in the hydrodynamic regime $\varepsilon\to 0$, we establish, in the next lemma, an important technical result which builds upon the method of proof of Lemma \ref{lemma:attenuation:coefficient}. It shows, in the case $d\geq 3$, that the classical volume-preserving pre-post-collisional change of variables \eqref{change:prepost} induces volume-preserving collisional symmetries
\begin{equation*}
	(\omega,\omega_*,\sigma)\mapsto \left(\omega',\omega_*',\frac{\omega-\omega_*}{|\omega-\omega_*|}\right)
	=\left(\frac{\omega+\omega_*}{2}+\frac{|\omega-\omega_*|}{2}\sigma,\frac{\omega+\omega_*}{2}-\frac{|\omega-\omega_*|}{2}\sigma,\frac{\omega-\omega_*}{|\omega-\omega_*|}\right),
\end{equation*}
when restricted to the invariant bundle
\begin{equation*}
	(\omega,\omega_*,\sigma)
	\in \partial B(0,R)\times\partial B(0,R)\times\left(\mathbb{S}^{d-2}\perp\frac{\omega+\omega_*}2\right)
	\subset \mathbb{R}^d\times\mathbb{R}^d\times\mathbb{S}^{d-1},
\end{equation*}
where $R>0$ and $\mathbb{S}^{d-2}\perp \frac{\omega+\omega_*}2$ denotes the $(d-2)$-dimensional unit sphere in $\mathbb{R}^d$, centered at the origin, normal to $\frac{\omega+\omega_*}2$, embedded with its natural surface measure.

\begin{lem}\label{lemma:collisional:symmetries}
	Let $d\geq 3$. Then, for all $R>0$, it holds that
	\begin{equation*}
		\begin{aligned}
			\int_{\partial B(0,R)\times\partial B(0,R)\times\left(\mathbb{S}^{d-2}\perp\frac{\omega+\omega_*}2\right)}
			&\chi(\omega,\omega_*,\omega',\omega_*')
			d\omega d\omega_* d\sigma
			\\
			&=
			\int_{\partial B(0,R)\times\partial B(0,R)\times\left(\mathbb{S}^{d-2}\perp\frac{\omega+\omega_*}2\right)}
			\chi\left(\omega',\omega_*',\omega,\omega_*\right)
			d\omega d\omega_* d\sigma,
		\end{aligned}
	\end{equation*}
	for any $\chi\in C\left(\partial B(0,R)^4\right)$, where
	\begin{equation*}
		\omega'=\frac{\omega+\omega_*}{2}+\frac{|\omega-\omega_*|}{2}\sigma
		\in\partial B(0,R)
	\end{equation*}
	and
	\begin{equation*}
		\omega_*'=\frac{\omega+\omega_*}{2}-\frac{|\omega-\omega_*|}{2}\sigma
		\in\partial B(0,R).
	\end{equation*}
\end{lem}

\begin{proof}
	We begin by extending $\chi$ into a bounded measurable function $\widetilde\chi$ on $(\mathbb{R}^d)^4$, with compact support, by setting
	\begin{equation*}
		\begin{aligned}
			\widetilde\chi(v,v_*,w,w_*)
			&=
			\chi\left(R\frac{v}{|v|},R\frac{v_*}{|v_*|},R\frac{w}{|w|},R\frac{w_*}{|w_*|}\right)
			\\
			&\quad\times\mathds{1}_{\left\{
			|v^2-R^2|\leq \varepsilon^\tau,\ |v_*^2-R^2|\leq \varepsilon^\tau,\ |w^2-R^2|\leq \varepsilon^\tau,\ |w_*^2-R^2|\leq \varepsilon^\tau
			\right\}},
		\end{aligned}
	\end{equation*}
	where $\varepsilon^\tau$, with $\varepsilon>0$ and $\tau>0$, is to be interpreted as a small parameter which will eventually tend to zero. Since we are not currently considering a hydrodynamic regime, this parameter could be replaced by any other symbol and the exponent $\tau$ does not have any particular significance. However, we prefer to stick to the notation $\varepsilon^\tau$ for consistency with the proof of Lemma \ref{lemma:attenuation:coefficient}.
	
	In particular, we can employ the variables introduced in \eqref{attenuation:coefficient:8}, \eqref{attenuation:coefficient:9} and \eqref{attenuation:coefficient:10}, in the proof of Lemma \ref{lemma:attenuation:coefficient}, to write that
	\begin{equation*}
		\widetilde\chi(v,v_*,v',v_*')
		=
		\chi\left(R\nu,R\nu_*,R\frac{v'}{|v'|},R\frac{v_*'}{|v_*'|}\right)
		\mathds{1}_{\left\{
		|u|\leq 1,\ |u_*|\leq 1,\ \left|\frac{u+u_*}2+s\right|\leq 1,\ \left|\frac{u+u_*}2-s\right|\leq 1
		\right\}},
	\end{equation*}
	and, repeating the steps leading to \eqref{attenuation:coefficient:4}, to deduce that
	\begin{equation*}
		\begin{aligned}
			&\int_{\mathbb{R}^d\times\mathbb{R}^d\times\mathbb{S}^{d-1}}
			\widetilde\chi(v,v_*,v',v_*')dvdv_*d\sigma
			\\
			&=\frac{\varepsilon^{2\tau}}4
			\int_{-\frac{R^2}{\varepsilon^\tau}}^\infty
			\int_{\mathbb{S}^{d-1}}
			\int_{-\frac{R^2}{\varepsilon^\tau}}^\infty
			\int_{\mathbb{S}^{d-1}}
			\int_{\mathbb{S}^{d-1}}
			\widetilde\chi(v,v_*,v',v_*')
			(R^2+\varepsilon^\tau u)^\frac{d-2}2
			(R^2+\varepsilon^\tau u_*)^\frac{d-2}2
			d\sigma d\nu du d\nu_* du_*
			\\
			&=\frac{\varepsilon^{3\tau}}2
			\int_{-1}^1
			\int_{\mathbb{S}^{d-1}}
			\int_{-1}^1
			\int_{\mathbb{S}^{d-1}}
			\int_\mathbb{R}
			\left(\int_{\mathbb{S}^{d-2}\perp\frac{v+v_*}{2}}\widetilde\chi(v,v_*,v',v_*')d\widetilde\sigma\right)
			\\
			&\quad\times\frac{ \left(\left( |v+v_*||v-v_*| \right)^2-(2\varepsilon^\tau s)^2\right)^\frac{d-3}2_+
			(R^2+\varepsilon^\tau u)^\frac{d-2}2
			(R^2+\varepsilon^\tau u_*)^\frac{d-2}2}
			{\left(|v+v_*||v-v_*|\right)^{d-2}}
			dsd\nu du d\nu_* du_*,
		\end{aligned}
	\end{equation*}
	where, for any $\lambda\in\mathbb{R}$, we have denoted $(z)_+^\lambda=z^\lambda$, if $z>0$, and $(z)_+^\lambda=0$, if $z\leq 0$.
	
	We are now going to let $\varepsilon$ tend to zero in the last integral above.
	To that end, recall that $\widetilde\chi$ is uniformly bounded and observe that
	\begin{equation}\label{new:singularity}
		\frac{1}{|v+v_*|}\leq
		\frac{1}{(|v||v_*|)^\frac 12|\nu+\nu_*|}
		\qquad\text{and}\qquad
		\frac{1}{|v-v_*|}\leq
		\frac{1}{(|v||v_*|)^\frac 12|\nu-\nu_*|},
	\end{equation}
	which, employing \eqref{attenuation:coefficient:5}, allows us to deduce that the integrand above is dominated by an integrable function independent of $\varepsilon$.
	
	Therefore, by the Dominated Convergence Theorem, we find that
	\begin{equation*}
		\begin{aligned}
			\lim_{\varepsilon\to 0}&\varepsilon^{-3\tau}
			\int_{\mathbb{R}^d\times\mathbb{R}^d\times\mathbb{S}^{d-1}}
			\widetilde\chi(v,v_*,v',v_*')dvdv_*d\sigma
			\\
			&=
			\int_{-1}^1
			\int_{\mathbb{S}^{d-1}}
			\int_{-1}^1
			\int_{\mathbb{S}^{d-1}}
			\int_\mathbb{R}
			H(u,\nu,u_*,\nu_*,s)
			\frac{
			R^{2(d-3)}\mathds{1}_{\left\{
			\left|\frac{u+u_*}2+s\right|\leq 1,\ \left|\frac{u+u_*}2-s\right|\leq 1
		\right\}}}
			{2|\nu+\nu_*||\nu-\nu_*|}
			dsd\nu du d\nu_* du_*,
		\end{aligned}
	\end{equation*}
	where $H$ stands for the pointwise limit
	\begin{equation*}
		H(u,\nu,u_*,\nu_*,s)
		=\lim_{\varepsilon\to 0}
		\int_{\mathbb{S}^{d-2}\perp\frac{v+v_*}{2}}
		\chi\left(R\nu,R\nu_*,R\frac{v'}{|v'|},R\frac{v_*'}{|v_*'|}\right)
		d\widetilde\sigma,
	\end{equation*}
	for almost every $(u,\nu,u_*,\nu_*,s)$.
	
	In order to evaluate $H$ explicitly, we assume that $\nu+\nu_*\neq 0$ and we observe that the pointwise limits
	\begin{equation}\label{pointwise:limits}
		\begin{aligned}
			\lim_{\varepsilon\to 0}v&=\lim_{\varepsilon\to 0}\left(R^2+\varepsilon^\tau u\right)^{\frac 12}\nu=R\nu,
			\\
			\lim_{\varepsilon\to 0}v_*&=\lim_{\varepsilon\to 0}\left(R^2+\varepsilon^\tau u_*\right)^{\frac 12}\nu_*=R\nu_*,
			\\
			\lim_{\varepsilon\to 0} |v'|&=\lim_{\varepsilon\to 0}\left(R^2+\varepsilon^\tau\left(\frac{u+u_*}2+s\right)\right)^\frac 12=R,
			\\
			\lim_{\varepsilon\to 0} |v_*'|&=\lim_{\varepsilon\to 0}\left(R^2+\varepsilon^\tau\left(\frac{u+u_*}2-s\right)\right)^\frac 12=R,
		\end{aligned}
	\end{equation}
	are all independent of $\widetilde\sigma$. Furthermore, parametrizing $\mathbb{S}^{d-2}\perp\frac{v+v_*}{2}$ by
	\begin{equation}\label{parametrization:sphere}
		\widetilde\sigma=\mathcal{R}_{\frac{v+v_*}2}\overline\sigma,
	\end{equation}
	where $\overline\sigma \in\mathbb{S}^{d-2}\perp e_d$, with $e_d=(0,\ldots,0,1)\in\mathbb{R}^d$, and $\mathcal{R}_{V}$, for any $V\in\mathbb{R}^d\setminus\{0\}$, is the unique rotation of $\mathbb{R}^d$ which takes $e_d$ to $\frac{V}{|V|}$ while leaving the $(d-2)$-dimensional subspace of $\mathbb{R}^d$ orthogonal to $e_d$ and $V$ invariant, we infer, by continuity of $\chi$, that
	\begin{equation*}
		\begin{aligned}
			H
			&=\lim_{\varepsilon\to 0}
			\int_{\mathbb{S}^{d-2}\perp e_d}
			\chi\left(R\nu,R\nu_*,R\frac{v'}{|v'|},R\frac{v_*'}{|v_*'|}\right)
			d\overline\sigma
			\\
			&=\int_{\mathbb{S}^{d-2}\perp e_d}
			\chi\left(R\nu,R\nu_*,
			R\left(\frac{\nu+\nu_*}{2}+\frac{|\nu-\nu_*|}{2}
			\mathcal{R}_{\frac{\nu+\nu_*}2}\overline \sigma\right),
			R\left(\frac{\nu+\nu_*}{2}-\frac{|\nu-\nu_*|}{2}
			\mathcal{R}_{\frac{\nu+\nu_*}2}\overline \sigma\right)
			\right)
			d\overline\sigma
			\\
			&=\int_{\mathbb{S}^{d-2}\perp \frac{\nu+\nu_*}2}
			\chi\left(R\nu,R\nu_*,
			R\left(\frac{\nu+\nu_*}{2}+\frac{|\nu-\nu_*|}{2}
			\widetilde \sigma\right),
			R\left(\frac{\nu+\nu_*}{2}-\frac{|\nu-\nu_*|}{2}
			\widetilde \sigma\right)
			\right)
			d\widetilde\sigma,
		\end{aligned}
	\end{equation*}
	which provides us with an explicit expression for $H$.
	
	All in all, we conclude, with the straightforward computation
	\begin{equation*}
		\int_{-1}^1\int_{-1}^1\int_\mathbb{R}
		\mathds{1}_{\left\{
		\left|\frac{u+u_*}2+s\right|\leq 1,\ \left|\frac{u+u_*}2-s\right|\leq 1
		\right\}}
		dsdu_*du
		=\int_{-1}^1\int_{-1}^1(2-|u+u_*|)du_*du=\frac{16}3,
	\end{equation*}
	that
	\begin{equation*}
		\begin{aligned}
			\lim_{\varepsilon\to 0}&\varepsilon^{-3\tau}
			\int_{\mathbb{R}^d\times\mathbb{R}^d\times\mathbb{S}^{d-1}}
			\widetilde\chi(v,v_*,v',v_*')dvdv_*d\sigma
			\\
			&=
			\frac{8}{3R^2}
			\int_{\partial B(0,R)}
			\int_{\partial B(0,R)}
			\int_{\mathbb{S}^{d-2}\perp \frac{\omega+\omega_*}2}
			\frac{\chi\left(\omega,\omega_*,
			\frac{\omega+\omega_*}{2}+\frac{|\omega-\omega_*|}{2}
			\widetilde \sigma,
			\frac{\omega+\omega_*}{2}-\frac{|\omega-\omega_*|}{2}
			\widetilde \sigma
			\right)}
			{|\omega+\omega_*||\omega-\omega_*|}
			d\widetilde\sigma d\omega d\omega_*.
		\end{aligned}
	\end{equation*}
	
	Now, notice that a similar argument yields the limit
	\begin{equation*}
		\begin{aligned}
			\lim_{\varepsilon\to 0}&\varepsilon^{-3\tau}
			\int_{\mathbb{R}^d\times\mathbb{R}^d\times\mathbb{S}^{d-1}}
			\widetilde\chi(v',v_*',v,v_*)dvdv_*d\sigma
			\\
			&=
			\frac{8}{3R^2}
			\int_{\partial B(0,R)}
			\int_{\partial B(0,R)}
			\int_{\mathbb{S}^{d-2}\perp \frac{\omega+\omega_*}2}
			\frac{\chi\left(
			\frac{\omega+\omega_*}{2}+\frac{|\omega-\omega_*|}{2}
			\widetilde \sigma,
			\frac{\omega+\omega_*}{2}-\frac{|\omega-\omega_*|}{2}
			\widetilde \sigma,
			\omega,\omega_*
			\right)}
			{|\omega+\omega_*||\omega-\omega_*|}
			d\widetilde\sigma d\omega d\omega_*.
		\end{aligned}
	\end{equation*}
	Therefore, by the pre-post-collisional symmetries in the whole space, since it holds that
	\begin{equation*}
		\int_{\mathbb{R}^d\times\mathbb{R}^d\times\mathbb{S}^{d-1}}
		\widetilde\chi(v,v_*,v',v_*')dvdv_*d\sigma
		=
		\int_{\mathbb{R}^d\times\mathbb{R}^d\times\mathbb{S}^{d-1}}
		\widetilde\chi(v',v_*',v,v_*)dvdv_*d\sigma,
	\end{equation*}
	 we conclude that
	\begin{equation*}
		\begin{aligned}
			\int_{\partial B(0,R)}
			&\int_{\partial B(0,R)}
			\int_{\mathbb{S}^{d-2}\perp \frac{\omega+\omega_*}2}
			\frac{\chi\left(\omega,\omega_*,
			\frac{\omega+\omega_*}{2}+\frac{|\omega-\omega_*|}{2}
			\widetilde \sigma,
			\frac{\omega+\omega_*}{2}-\frac{|\omega-\omega_*|}{2}
			\widetilde \sigma
			\right)}
			{|\omega+\omega_*||\omega-\omega_*|}
			d\widetilde\sigma d\omega d\omega_*
			\\
			&=
			\int_{\partial B(0,R)}
			\int_{\partial B(0,R)}
			\int_{\mathbb{S}^{d-2}\perp \frac{\omega+\omega_*}2}
			\frac{\chi\left(
			\frac{\omega+\omega_*}{2}+\frac{|\omega-\omega_*|}{2}
			\widetilde \sigma,
			\frac{\omega+\omega_*}{2}-\frac{|\omega-\omega_*|}{2}
			\widetilde \sigma,
			\omega,\omega_*
			\right)}
			{|\omega+\omega_*||\omega-\omega_*|}
			d\widetilde\sigma d\omega d\omega_*.
		\end{aligned}
	\end{equation*}
	Finally, replacing $\chi$ by the continuous function $|\omega+\omega_*||\omega-\omega_*|\chi$ concludes the proof of the lemma.
\end{proof}

\subsection{Relaxation toward thermodynamic equilibrium}\label{section:relaxation}

In Proposition \ref{prop:limit:characterization}, we showed that density fluctuations which satisfy a suitable entropy bound, relative to a normalized Fermi--Dirac distribution at low temperatures, will concentrate on a sphere in the velocity variable. By further exploiting the entropy dissipation bound, which, loosely speaking, controls the distance of densities to the space of statistical equilibria, we are now in a position to describe how the limiting fluctuations reach thermodynamic equilibrium, in the limit $\varepsilon\to 0$, by characterizing their velocity distribution on the sphere. The following result is thus instrumental in understanding hydrodynamic regimes near Fermi ground states.

\begin{prop}\label{prop:instrumental}
	Let $d\geq 3$. Consider a cross-section
	\begin{equation*}
		b(z,\sigma)=b\left(|z|,\frac{z}{|z|}\cdot\sigma\right)\geq 0
	\end{equation*}
	such that
	\begin{equation*}
		b(z,\sigma)\in L^\infty(\mathbb{R}^d\times\mathbb{S}^{d-1})\cap C(B(0,2R)\times\mathbb{S}^{d-1})
	\end{equation*}
	and
	\begin{equation*}
		b(z,\sigma)>0\quad\text{on }B(0,2R)\times\mathbb{S}^{d-1},
	\end{equation*}
	and a family of density distributions $0\leq f_\varepsilon(t,x,v)\leq\delta^{-1}$, with $\varepsilon>0$, such that the relative entropy bound \eqref{relative:entropy:1} and the entropy dissipation bound \eqref{entropy:dissipation:1} hold uniformly in $\varepsilon$, with some fixed parameters $\gamma>0$ and $\kappa>0$, where the Fermi--Dirac distribution $M_\varepsilon$ defined in \eqref{normalized:distribution} has a given parameter value $\tau>0$.
	Further suppose that
	\begin{equation*}
		0<\gamma=\tau< 1
		\qquad\text{and}\qquad
		\kappa>2\tau,
	\end{equation*}
	and consider the density fluctuations $g_\varepsilon(t,x,v)$ defined in \eqref{epsilon:fluctuations}.
	
	In accordance with Proposition \ref{prop:limit:characterization}, as $\varepsilon\to 0$, up to extraction of a subsequence, the family of fluctuations $g_\varepsilon$ converges in the weak* topology of $\mathcal{M}_\mathrm{loc}\big(\mathbb{R}^+\times\mathbb{R}^d\times\mathbb{R}^d\big)$ toward a limit point
	\begin{equation*}
		\mu(t,x,v)=g\left(t,x,R\frac{v}{|v|}\right)dt\otimes dx\otimes \delta_{\partial B(0,R)}(v),
	\end{equation*}
	where the density $g(t,x,\omega)$ belongs to $L^\infty\big(dt;L^2\big(\mathbb{R}^d\times\partial B(0,R)\big)\big)$.
	
	Then, the limiting density can be further characterized as
	\begin{equation*}
		g(t,x,\omega)=\rho(t,x)+U(t,x)\cdot\omega,
	\end{equation*}
	where the density $\rho(t,x)$ and the velocity field $U(t,x)$ belong to $L^\infty\big(dt;L^2\big(\mathbb{R}^d\big)\big)$.
\end{prop}

\begin{proof}
	Consider the renormalized fluctuations $\phi_\varepsilon(t,x,v)$ and $\psi_\varepsilon(t,x,v)$ defined in \eqref{epsilon:renormalized:fluctuations}. For mere convenience of notation, we also introduce the renormalized density fluctuation
	\begin{equation}\label{renormalized:fluctuation}
		h_\varepsilon=\delta M_\varepsilon\left(1-\delta M_\varepsilon\right)
		\left(\phi_\varepsilon+\psi_\varepsilon\right).
	\end{equation}
	Now, in view of the decomposition \eqref{renormalized:decomposition} and the bound \eqref{bound:extra}, recall that
	\begin{equation}\label{fluctuation:expansion}
		g_\varepsilon= h_\varepsilon+O(\varepsilon^{1-\tau})_{L^\infty (dt;L^1(dxdv))}
	\end{equation}
	and
	\begin{equation}\label{distance:densities}
		(\delta M_\varepsilon(1-\delta M_\varepsilon))^{\alpha-\frac 12}
		h_\varepsilon=O\left(\varepsilon^{\tau\left(\frac 1q-1\right)}\right)_{L^\infty (dt;L^2(dx;L^q(dv)))},
	\end{equation}
	for all $1\leq q\leq 2$ and $\alpha> 0$, or $q=2$ and $\alpha=0$.
	In particular, this implies that both families of fluctuations $g_\varepsilon$ and $h_\varepsilon$ converge in the weak* topology of locally finite Radon measures toward the same limit point $\mu$.
	Specifically, for any compactly supported continuous test function $\varphi(t,x,v)$, it holds that
	\begin{equation*}
		\int_{\mathbb{R}^+\times\mathbb{R}^d\times\mathbb{R}^d}
		h_\varepsilon(t,x,v)\varphi(t,x,v)dtdxdv
		\to
		\int_{\mathbb{R}^+\times\mathbb{R}^d\times\partial B(0,R)}
		g(t,x,\omega)\varphi(t,x,\omega)dtdxd\omega,
	\end{equation*}
	as $\varepsilon\to 0$.
	Therefore, in order to characterize $\mu$ and $g$, it will be convenient to work on $h_\varepsilon$, rather than $g_\varepsilon$, which will allow us to employ Lemma \ref{lemma:linearized:dissipation}.
	
	Next, for any $\alpha>1$ and any compactly supported continuous test function $\Theta(t,x,u,v)$, with $(t,x,u,v)\in\mathbb{R}^+\times\mathbb{R}^d\times\mathbb{R}\times\mathbb{R}^d$,
	we consider the weak form of the linearized collision operator
	\begin{equation*}
		\mathcal{I}_\varepsilon(t,x)=
		\varepsilon^{-2\tau}\int_{\mathbb{R}^d\times\mathbb{R}^d\times\mathbb{S}^{d-1}}
		\left(\pi_{\varepsilon}+\pi_{\varepsilon*}-\pi_{\varepsilon}'-\pi_{\varepsilon*}'\right)
		\Theta\left(t,x,\frac{v^2-R^2}{\varepsilon^\tau},v\right)
		bm_\varepsilon^\alpha dvdv_*d\sigma,
	\end{equation*}
	where $m_\varepsilon$ is given in \eqref{maxwellian:product:notation} and $\pi_\varepsilon$ denotes
	\begin{equation}\label{pi:fluctuation}
		\pi_\varepsilon=\frac{h_\varepsilon}{\delta M_\varepsilon\left(1-\delta M_\varepsilon\right)}=
		\phi_\varepsilon+\psi_\varepsilon.
	\end{equation}
	The asymptotic behavior of $\mathcal{I}_\varepsilon$, as $\varepsilon\to 0$, will give us crucial information on the convergence of $g_\varepsilon$ and $h_\varepsilon$ toward thermodynamic equilibrium.
	In particular, by Lemma \ref{lemma:linearized:dissipation}, one observes that
	\begin{equation}\label{asymptotic:I}
		\mathcal{I}_\varepsilon(t,x)=
		O(\varepsilon^{1-\tau})_{L^\infty(dt;L^1(dx))}
		+O(\varepsilon^{\frac{\kappa}2-\tau})_{L^2(dtdx)},
	\end{equation}
	as $\varepsilon\to 0$.

	In order to further study the limit of $\mathcal{I}_\varepsilon$, we introduce the quantity
	\begin{equation*}
		\begin{aligned}
			\chi_\varepsilon(t,x,v,v_*,\sigma)
			&=\Theta+\Theta_*-\Theta'-\Theta_*'
			\\
			&=\Theta\left(t,x,\frac{v^2-R^2}{\varepsilon^\tau},v\right)
			+\Theta\left(t,x,\frac{v_*^2-R^2}{\varepsilon^\tau},v_*\right)
			\\
			&\quad -\Theta\left(t,x,\frac{v'^2-R^2}{\varepsilon^\tau},v'\right)
			-\Theta\left(t,x,\frac{v_*'^2-R^2}{\varepsilon^\tau},v_*'\right),
		\end{aligned}
	\end{equation*}
	so that $\mathcal{I}_\varepsilon$ can be recast, exploiting the collisional symmetries, in the convenient dual form
	\begin{equation*}
		\mathcal{I}_\varepsilon(t,x)=
		\varepsilon^{-2\tau}\int_{\mathbb{R}^d\times\mathbb{R}^d\times\mathbb{S}^{d-1}}
		\pi_\varepsilon\chi_\varepsilon
		bm_\varepsilon^\alpha dvdv_*d\sigma.
	\end{equation*}
	Since we are assuming that $d\geq 3$, note that
	\begin{equation*}
		\varepsilon^{-2\tau}\int_{\mathbb{R}^d\times\mathbb{R}^d\times\mathbb{S}^{d-1}}
		\left(\Theta+\Theta_*-\Theta'-\Theta_*'\right)
		bm_\varepsilon^{\alpha-1} dvdv_*d\sigma
	\end{equation*}
	remains bounded, uniformly in $\varepsilon$, by virtue of Lemma \ref{lemma:attenuation:coefficient}.
	
	Then, we further decompose $\mathcal{I}_\varepsilon$ into
	\begin{equation*}
		\mathcal{I}_\varepsilon=\mathcal{I}_\varepsilon^1+\mathcal{I}_\varepsilon^2,
	\end{equation*}
	where
	\begin{equation*}
		\begin{aligned}
			\mathcal{I}_\varepsilon^1(t,x)&=
			\varepsilon^{-2\tau}\int_{\mathbb{R}^d\times\mathbb{R}^d\times\mathbb{S}^{d-1}}
			\pi_\varepsilon\chi_\varepsilon
			bm_\varepsilon^\alpha \mathds{1}_{\{v^2\geq \frac {R^2}2,v_*^2\geq \frac {R^2}2\}} dvdv_*d\sigma,
			\\
			\mathcal{I}_\varepsilon^2(t,x)&=
			\varepsilon^{-2\tau}\int_{\mathbb{R}^d\times\mathbb{R}^d\times\mathbb{S}^{d-1}}
			\pi_\varepsilon\chi_\varepsilon
			bm_\varepsilon^\alpha
			\left(\mathds{1}_{\{v^2< \frac {R^2}2\}}
			+\mathds{1}_{\{v^2\geq \frac {R^2}2,v_*^2< \frac {R^2}2\}}
			\right)
			dvdv_*d\sigma.
		\end{aligned}
	\end{equation*}
	By the fact that
	\begin{equation*}
		m_\varepsilon
		\leq
		\left[16\cosh\left(\frac{v^2-R^2}{2\varepsilon^\tau}\right)\cosh\left(\frac{v_*^2-R^2}{2\varepsilon^\tau}\right)\right]^{-1},
	\end{equation*}
	it is then readily seen, for any $1<\alpha'<\alpha$, that
	\begin{equation*}
		\left\|\mathcal{I}_\varepsilon^2\right\|_{L^\infty(dt;L^2(dx))}
		\lesssim
		\big\|(\delta M_\varepsilon(1-\delta M_\varepsilon))^{\frac {\alpha-\alpha'+1}2}\pi_\varepsilon\big\|_{L^\infty(dt;L^2(dx;L^1(dv)))}
		\frac{\|\Theta\|_{L^\infty}\|b\|_{L^\infty}}{\varepsilon^{\tau}\cosh^{\alpha'-1}\left(\frac{R^2}{4\varepsilon^\tau}\right)},
	\end{equation*}
	which implies, by \eqref{distance:densities}, that $\mathcal{I}_\varepsilon^2$ vanishes exponentially in $L^\infty(dt;L^2(dx))$. In other words, relaxing the exponential convergence, one can write that
	\begin{equation}\label{vanishing:I2}
		\mathcal{I}_\varepsilon^2(t,x)=O(\varepsilon^\beta)_{L^\infty(dt;L^2(dx))},
	\end{equation}
	for any $\beta>0$.

	In order to analyze the asymptotic behavior of $\mathcal{I}_\varepsilon^1$, we consider the suitable variables $u$, $u_*$, $\nu$, $\nu_*$, $s$ and $\widetilde \sigma$ introduced in \eqref{attenuation:coefficient:8}, \eqref{attenuation:coefficient:9} and \eqref{attenuation:coefficient:10}. Then, repeating the steps leading to \eqref{attenuation:coefficient:4}, we obtain the representation
	\begin{equation*}
		\begin{aligned}
			\mathcal{I}_\varepsilon^1 &=
			\int_{\mathbb{R}^d}
			\int_{-\frac{R^2}{\varepsilon^\tau}}^\infty
			\int_{\mathbb{S}^{d-1}}
			\int_{-\frac{|v+v_*||v-v_*|}{2\varepsilon^\tau}}^{\frac{|v+v_*||v-v_*|}{2\varepsilon^\tau}}
			\frac{\pi_\varepsilon\left(\int_{\mathbb{S}^{d-2}\perp\frac{v+v_*}{2}}
			\chi_\varepsilon(v,v_*,\sigma)b(v-v_*,\sigma)d\widetilde\sigma \right)
			(R^2+\varepsilon^\tau u_*)^\frac{d-2}2}
			{8^{\alpha}\cosh^\alpha\left(\frac u2\right)
			\cosh^\alpha\left(\frac {u_*}2\right)
			\left(\cosh\left(\frac{u+u_*}2\right)+\cosh s\right)^\alpha}
			\\
			&\quad\times\frac{ \left(\left( |v+v_*||v-v_*| \right)^2-(2\varepsilon^\tau s)^2\right)^\frac{d-3}2}
			{\left(|v+v_*||v-v_*|\right)^{d-2}}
			\mathds{1}_{\{v^2\geq \frac {R^2}2,v_*^2\geq \frac {R^2}2\}}ds d\nu_* du_*dv
			\\
			&=
			\int_{\mathbb{R}}
			\int_{\mathbb{S}^{d-1}}
			\int_{\mathbb{R}}
			\int_{\mathbb{S}^{d-1}}
			\int_\mathbb{R}
			\frac{\widetilde \pi_\varepsilon\left(\int_{\mathbb{S}^{d-2}\perp\frac{v+v_*}{2}}
			\chi_\varepsilon b d\widetilde\sigma \right)
			(R^2+\varepsilon^\tau u_*)_+^\frac{d-2}2}
			{8^{\alpha}\cosh^\alpha\left(\frac u2\right)
			\cosh^\alpha\left(\frac {u_*}2\right)
			\left(\cosh\left(\frac{u+u_*}2\right)+\cosh s\right)^\alpha}
			\\
			&\quad\times\frac{ \left(\left( |v+v_*||v-v_*| \right)^2-(2\varepsilon^\tau s)^2\right)^\frac{d-3}2_+}
			{\left(|v+v_*||v-v_*|\right)^{d-2}}
			\mathds{1}_{\{v^2\geq \frac {R^2}2,v_*^2\geq \frac {R^2}2\}}ds d\nu_* du_*d\nu du,
		\end{aligned}
	\end{equation*}
	where, for any $\lambda\in\mathbb{R}$, we have denoted $(z)_+^\lambda=z^\lambda$, if $z>0$, and $(z)_+^\lambda=0$, if $z\leq 0$, and introduced the function
	\begin{equation}\label{pi:fluctuation:dilation}
		\widetilde \pi_\varepsilon(t,x,u,\nu)
		= \frac{\varepsilon^\tau}2 \pi_\varepsilon(t,x,v)\left(R^2+\varepsilon^\tau u\right)_+^{\frac{d-2}{2}},
	\end{equation}
	which is defined for all $(t,x,u,\nu)\in \mathbb{R}^+\times\mathbb{R}^d\times\mathbb{R}\times\mathbb{S}^{d-1}$.

	Some work is now required to identify the weak limit of $\mathcal{I}_\varepsilon^1$. To that end, on the one hand, for any compactly supported continuous test function $\varphi(t,x,v)$, we notice that
	\begin{equation*}
		\begin{aligned}
			\int_{\mathbb{R}^+\times\mathbb{R}^d\times\mathbb{R}\times\mathbb{S}^{d-1}}
			\frac{\widetilde \pi_\varepsilon(t,x,u,\nu)}{4\cosh^2(\frac u2)}
			\varphi(t,x,v)dtdxdud\nu
			&=
			\int_{\mathbb{R}^+\times\mathbb{R}^d\times\mathbb{R}^d}
			h_\varepsilon(t,x,v)\varphi(t,x,v)dtdxdv
			\\
			&\to
			\int_{\mathbb{R}^+\times\mathbb{R}^d\times\partial B(0,R)}
			g(t,x,\omega)\varphi(t,x,\omega)dtdxd\omega,
		\end{aligned}
	\end{equation*}
	as $\varepsilon\to 0$.

	On the other hand, by virtue of \eqref{distance:densities}, for any $\beta>0$, one sees that
	\begin{equation*}
		\int_{\mathbb{R}\times\mathbb{S}^{d-1}}
		\big|\widetilde \pi_\varepsilon\big|\left(\delta M_\varepsilon\left(1-\delta M_\varepsilon\right)\right)^{\beta+\frac 12}
		dud\nu
		=
		\int_{\mathbb{R}^{d}}
		| h_\varepsilon|
		\left(\delta M_\varepsilon\left(1-\delta M_\varepsilon\right)\right)^{\beta-\frac 12}
		dv
		=O(1)_{L^\infty(dt;L^2(dx))}
	\end{equation*}
	and
	\begin{equation*}
		\begin{aligned}
			\int_{\mathbb{R}\times\mathbb{S}^{d-1}}
			&\big|\widetilde \pi_\varepsilon\big|^2
			\left(\delta M_\varepsilon\left(1-\delta M_\varepsilon\right)\right)^{2\beta+1}
			dud\nu
			\\
			&\lesssim
			\varepsilon^{\tau}
			\int_{\mathbb{R}\times\mathbb{S}^{d-1}}
			\left(\left|\sqrt{\delta M_\varepsilon}\phi_\varepsilon\right|
			+\left|\sqrt{1-\delta M_\varepsilon}\psi_\varepsilon\right|\right)^2
			\frac{\varepsilon^\tau}2 \left(R^2+\varepsilon^\tau u\right)_+^{\frac{d-2}{2}}
			dud\nu
			\\
			&=
			\varepsilon^{\tau}
			\int_{\mathbb{R}^{d}}
			\left(\left|\sqrt{\delta M_\varepsilon}\phi_\varepsilon\right|
			+\left|\sqrt{1-\delta M_\varepsilon}\psi_\varepsilon\right|\right)^2
			dv
			=O(1)_{L^\infty(dt;L^1(dx))},
		\end{aligned}
	\end{equation*}
	whereby $\widetilde \pi_\varepsilon \left(\delta M_\varepsilon\left(1-\delta M_\varepsilon\right)\right)^{\beta+\frac 12}$ is uniformly bounded in $L^\infty(dt;L^2(dx;L^1\cap L^2(dud\nu))$. Thus, up to extraction of a subsequence, there exists $\widetilde \pi(t,x,u,\nu)$ such that
	\begin{equation}\label{bound:pi}
		\frac{\widetilde \pi(t,x,u,\nu)}{\left(4\cosh^2(\frac u2)\right)^{\beta+\frac 12}}
		\in L^\infty(dt;L^2(dx;L^1\cap L^2(dud\nu))
	\end{equation}
	and
	\begin{equation}\label{convergence:pi}
		\widetilde \pi_\varepsilon \left(\delta M_\varepsilon\left(1-\delta M_\varepsilon\right)\right)^{\beta+\frac 12}
		\rightharpoonup^*
		\frac{\widetilde \pi(t,x,u,\nu)}{\left(4\cosh^2(\frac u2)\right)^{\beta+\frac 12}},
	\end{equation}
	as $\varepsilon\to 0$, for all $\beta>0$, at least in the weak* topology of $L^\infty(dt;L^2(dxdud\nu))$. Therefore, observing that
	\begin{equation*}
		\begin{aligned}
			\int_{\mathbb{R}^+\times\mathbb{R}^d\times\mathbb{R}\times\mathbb{S}^{d-1}}&
			\big|\widetilde \pi_\varepsilon(t,x,u,\nu)\left(\delta M_\varepsilon\left(1-\delta M_\varepsilon\right)\right)
			\left(\varphi(t,x,v)-\varphi(t,x,R\nu)\right)\big|dtdxdud\nu
			\\
			&=\int_{\mathbb{R}^+\times\mathbb{R}^d\times\mathbb{R}\times\mathbb{S}^{d-1}}
			\Big|\frac{\widetilde \pi_\varepsilon(t,x,u,\nu)}{4\cosh^2(\frac u2)}
			\left(\varphi(t,x,v)-\varphi(t,x,R\nu)\right)\Big|dtdxdud\nu
			\\
			&\leq
			\left\|\widetilde \pi_\varepsilon \left(\delta M_\varepsilon\left(1-\delta M_\varepsilon\right)\right)^{\beta+\frac 12}\right\|_{L^\infty(dt;L^2(dxdud\nu))}
			\\
			&\quad\times\left\|
			\left(\varphi(t,x,v)-\varphi(t,x,R\nu)\right)
			\left(\delta M_\varepsilon\left(1-\delta M_\varepsilon\right)\right)^{\frac 12-\beta}\right\|_{L^1(dt;L^2(dxdud\nu))},
		\end{aligned}
	\end{equation*}
	and, if $0<\beta<\frac 12$, by the Dominated Convergence Theorem, that
	\begin{equation*}
		\lim_{\varepsilon\to 0}\left\|
		\left(\varphi(t,x,v)-\varphi(t,x,R\nu)\right)
		\left(\delta M_\varepsilon\left(1-\delta M_\varepsilon\right)\right)^{\frac 12-\beta}\right\|_{L^1(dt;L^2(dxdud\nu))}
		=0,
	\end{equation*}
	we conclude that
	\begin{equation*}
		\begin{aligned}
			\int_{\mathbb{R}^+\times\mathbb{R}^d\times\mathbb{R}\times\mathbb{S}^{d-1}}
			&\frac{\widetilde \pi_\varepsilon(t,x,u,\nu)}{4\cosh^2(\frac u2)}\varphi(t,x,v)dtdxdud\nu
			\\
			&\to
			\int_{\mathbb{R}^+\times\mathbb{R}^d\times\mathbb{R}\times\mathbb{S}^{d-1}}
			\frac{\widetilde \pi(t,x,u,\nu)}{4\cosh^2(\frac u2)}
			\varphi(t,x,R\nu)dtdxdud\nu.
		\end{aligned}
	\end{equation*}

	We have thus evaluated the weak* limit of $\widetilde \pi_\varepsilon$ in two different ways, which, when compared, lead to the conclusion that
	\begin{equation*}
		\int_{\mathbb{R}^+\times\mathbb{R}^d\times\partial B(0,R)}
		g(t,x,\omega)\varphi(t,x,\omega)dtdxd\omega
		=
		\int_{\mathbb{R}^+\times\mathbb{R}^d\times\mathbb{R}\times\mathbb{S}^{d-1}}
		\frac{\widetilde \pi(t,x,u,\nu)}{4\cosh^2(\frac u2)}\varphi(t,x,R\nu)dtdxdud\nu.
	\end{equation*}
	By the arbitrariness of $\varphi$, we then deduce that
	\begin{equation}\label{velocity:distributions}
		\int_{\mathbb{R}}\frac{\widetilde \pi(t,x,u,\nu)}{4\cosh^2(\frac u2)}du=R^{d-1}g(t,x,R\nu).
	\end{equation}

	The use of the limiting scaled density $\widetilde \pi$ will now allow us to evaluate the limit of $\mathcal{I}_\varepsilon^1$.
	To that end, we write that
	\begin{equation*}
		\mathcal{I}_\varepsilon^1(t,x)=
		\int_{\mathbb{R}\times\mathbb{S}^{d-1}}
		\widetilde \pi_\varepsilon(t,x,u,\nu)
		H_\varepsilon(t,x,u,\nu)
		dud\nu,
	\end{equation*}
	where $H_\varepsilon$ stands for
	\begin{equation*}
		\begin{aligned}
			H_\varepsilon(t,x,u,\nu)
			&=
			\int_{\mathbb{R}\times\mathbb{S}^{d-1}\times\mathbb{R}}
			\frac{\left(\int_{\mathbb{S}^{d-2}\perp\frac{v+v_*}{2}}
			\chi_\varepsilon b d\widetilde\sigma \right)
			(R^2+\varepsilon^\tau u_*)_+^\frac{d-2}2}
			{8^{\alpha}\cosh^\alpha\left(\frac u2\right)
			\cosh^\alpha\left(\frac {u_*}2\right)
			\left(\cosh\left(\frac{u+u_*}2\right)+\cosh s\right)^\alpha}
			\\
			&\quad\times\frac{ \left(\left( |v+v_*||v-v_*| \right)^2-(2\varepsilon^\tau s)^2\right)^\frac{d-3}2_+}
			{\left(|v+v_*||v-v_*|\right)^{d-2}}
			\mathds{1}_{\{v^2\geq \frac {R^2}2,v_*^2\geq \frac {R^2}2\}}
			du_* d\nu_* ds.
		\end{aligned}
	\end{equation*}
	We claim now that
	\begin{equation}\label{limit:I:0}
		\int_{\mathbb{R}^+\times\mathbb{R}^d}\mathcal{I}_\varepsilon^1(t,x) dtdx
		\to
		\int_{\mathbb{R}^+\times\mathbb{R}^d\times\mathbb{R}\times\mathbb{S}^{d-1}}
		\widetilde \pi(t,x,u,\nu)
		H(t,x,u,\nu)
		dtdxdud\nu,
	\end{equation}
	as $\varepsilon\to 0$, with
	\begin{equation*}
		H(t,x,u,\nu)
		=
		\int_{\mathbb{R}\times\mathbb{S}^{d-1}\times\mathbb{R}}
		\frac{
		\left(\int_{\mathbb{S}^{d-2}\perp\frac{\nu+\nu_*}{2}}
		\chi_0
		b\big(R(\nu-\nu_*),\widetilde\sigma\big) d\widetilde\sigma
		\right)
		R^{d-4}du_* d\nu_* ds}
		{8^{\alpha}\cosh^\alpha\left(\frac u2\right)
		\cosh^\alpha\left(\frac {u_*}2\right)
		\left(\cosh\left(\frac{u+u_*}2\right)+\cosh s\right)^\alpha
		|\nu+\nu_*||\nu-\nu_*|}
	\end{equation*}
	and
	\begin{equation*}
		\begin{aligned}
			\chi_0(t,x,u,\nu,u_*,\nu_*,s,\widetilde \sigma)
			&=\Theta+\Theta_*-\Theta'-\Theta_*'
			\\
			&=
			\Theta(t,x,u,R\nu)+\Theta(t,x,u_*,R\nu_*)
			\\
			&\quad -\Theta\left(t,x,u',R\nu'\right)-\Theta\left(t,x,u_*',R\nu_*'\right),
		\end{aligned}
	\end{equation*}
	where we have denoted that
	\begin{equation*}
		u'=\frac{u+u_*}2+s,
		\qquad
		u_*'=\frac{u+u_*}2-s,
	\end{equation*}
	and
	\begin{equation*}
		\nu'=\frac{\nu+\nu_*}2+\frac{|\nu-\nu_*|}2\widetilde\sigma,
		\qquad
		\nu_*'=\frac{\nu+\nu_*}2-\frac{|\nu-\nu_*|}2\widetilde\sigma.
	\end{equation*}

	In particular, observing that $H(t,x,u,\nu)\cosh^\alpha\left(\frac u2\right)$ belongs to $L^\infty(dtdxdud\nu)$ and that it is compactly supported in $(t,x)$, we notice, by virtue of the weak* convergence of $\widetilde \pi_\varepsilon$ established in \eqref{convergence:pi}, that
	\begin{equation*}
		\begin{aligned}
			\int_{\mathbb{R}^+\times\mathbb{R}^d\times\mathbb{R}\times\mathbb{S}^{d-1}}
			\widetilde \pi_\varepsilon
			H
			dtdxdud\nu
			&=
			\int_{\mathbb{R}^+\times\mathbb{R}^d\times\mathbb{R}\times\mathbb{S}^{d-1}}
			\frac{\widetilde \pi_\varepsilon}{\cosh^{\alpha'}\left(\frac u2\right)}
			H\cosh^{\alpha'}\left(\frac u2\right)
			dtdxdud\nu
			\\
			&\to
			\int_{\mathbb{R}^+\times\mathbb{R}^d\times\mathbb{R}\times\mathbb{S}^{d-1}}
			\widetilde \pi
			H
			dtdxdud\nu,
		\end{aligned}
	\end{equation*}
	where $\alpha'$ denotes any fixed value in the range $(1,\alpha)$.
	Therefore, we see that \eqref{limit:I:0} will follow as a direct consequence of the control
	\begin{equation}\label{limit:I:1}
		\left\|(H_\varepsilon-H)\cosh^{\alpha'}\left(\frac u2\right)\right\|_{L^2(dtdxdud\nu)}\to 0,
	\end{equation}
	as $\varepsilon\to 0$.
	
	In order to justify the validity of \eqref{limit:I:1}, we first note, employing \eqref{new:singularity} and the fact that the domain of integration in $H_\varepsilon$ is restricted to $\{v^2\geq \frac {R^2}2,v_*^2\geq \frac {R^2}2\}$, that $(H_\varepsilon-H)\cosh^{\alpha'}\left(\frac u2\right)$ is dominated by the square-integrable function
	\begin{equation*}
		\frac{\|\Theta(t,x,u,v)\|_{L^\infty(dudv)}\|b\|_{L^\infty}}
		{\cosh^{\alpha-\alpha'}\left(\frac u2\right)}
		\int_{\mathbb{S}^{d-1}} \frac 1{|\nu+\nu_*||\nu-\nu_*|} d\nu_*
		\in L^2(dtdxdud\nu).
	\end{equation*}
	It is to be emphasized, in view of \eqref{attenuation:coefficient:5} and since $d\geq 3$, that the singularity $(|\nu+\nu_*||\nu-\nu_*|)^{-1}$ is integrable in $\nu_*$, uniformly in $\nu$.
	Therefore, by the Dominated Convergence Theorem, we see that, in order to show \eqref{limit:I:1}, it is now sufficient to establish the pointwise limit
	\begin{equation}\label{limit:dominated:1}
		\lim_{\varepsilon\to 0} H_\varepsilon(t,x,u,\nu)=H(t,x,u,\nu),
	\end{equation}
	for almost every $(t,x,u,\nu)\in \mathbb{R}^+\times\mathbb{R}^d\times\mathbb{R}\times\mathbb{S}^{d-1}$.
	
	To that end, we consider the parametrization of $\mathbb{S}^{d-2}\cap\frac{v+v_*}2$ given in \eqref{parametrization:sphere} to deduce, for any given $(t,x,u,\nu,u_*,\nu_*,s)$ such that $\nu+\nu_*\neq 0$, by continuity of $\Theta$ and $b$, that
	\begin{equation*}
		\begin{aligned}
			\lim_{\varepsilon\to 0}&
			\int_{\mathbb{S}^{d-2}\perp\frac{v+v_*}{2}}
			\chi_\varepsilon(t,x,v,v_*,\sigma) b(v-v_*,\sigma) d\widetilde\sigma
			\\
			&=\lim_{\varepsilon\to 0}
			\int_{\mathbb{S}^{d-2}\perp e_d}
			\chi_\varepsilon(t,x,v,v_*,\sigma) b(v-v_*,\sigma)
			d\overline\sigma
			\\
			&=
			\int_{\mathbb{S}^{d-2}\perp e_d}
			\Bigg(\Theta\left(t,x,u,R\nu\right)
			+\Theta\left(t,x,u_*,R\nu_*\right)
			\\
			&\quad-\Theta\left(t,x,u',R\left(\frac{\nu+\nu_*}{2}+\frac{|\nu-\nu_*|}{2}
			\mathcal{R}_{\frac{\nu+\nu_*}2}\overline \sigma\right)\right)
			\\
			&\quad-\Theta\left(t,x,u_*',R\left(\frac{\nu+\nu_*}{2}-\frac{|\nu-\nu_*|}{2}
			\mathcal{R}_{\frac{\nu+\nu_*}2}\overline \sigma\right)\right)\Bigg)
			b\big(R(\nu-\nu_*),\mathcal{R}_{\frac{\nu+\nu_*}2}\overline \sigma\big)
			d\overline\sigma
			\\
			&=
			\int_{\mathbb{S}^{d-2}\perp\frac{\nu+\nu_*}{2}}
			\chi_0(t,x,u,\nu,u_*,\nu_*,s,\widetilde\sigma)
			b\big(R(\nu-\nu_*),\widetilde\sigma\big) d\widetilde\sigma,
		\end{aligned}
	\end{equation*}
	where we exploited the pointwise limits \eqref{pointwise:limits}.
	Thus, fixing the variables $(t,x,u,\nu)$ and observing that the integrand of $H_\varepsilon$ is dominated by the function
	\begin{equation*}
		\frac{\|\Theta\|_{L^\infty}\|b\|_{L^\infty}}
		{\cosh^{\frac\alpha 2}\left(\frac {u_*}2\right)\cosh^{\alpha} (s)
		|\nu+\nu_*||\nu-\nu_*|},
	\end{equation*}
	which is integrable in $(u_*,\nu_*,s)$, we deduce, by the Dominated Convergence Theorem, again, that \eqref{limit:dominated:1} holds true. As previously emphasized, this in turn establishes the validity of \eqref{limit:I:0} and \eqref{limit:I:1}.

	All in all, combining the asymptotic controls of $\mathcal{I}_\varepsilon^1$ and $\mathcal{I}_\varepsilon^2$ given in \eqref{limit:I:0} and \eqref{vanishing:I2}, respectively, we conclude that
	\begin{equation*}
		\int_{\mathbb{R}^+\times\mathbb{R}^d}\mathcal{I}_\varepsilon(t,x) dtdx
		\to
		\int_{\mathbb{R}^+\times\mathbb{R}^d\times\mathbb{R}\times\mathbb{S}^{d-1}}
		\widetilde \pi(t,x,u,\nu)
		H(t,x,u,\nu)
		dtdxdud\nu,
	\end{equation*}
	as $\varepsilon\to 0$. In particular, further taking into account the asymptotic behavior of $\mathcal{I}_\varepsilon$ described in \eqref{asymptotic:I}, we finally deduce that
	\begin{equation*}
		\int_{\mathbb{R}^+\times\mathbb{R}^d\times\mathbb{R}\times\mathbb{S}^{d-1}}
		\widetilde \pi(t,x,u,\nu)
		H(t,x,u,\nu)
		dtdxdud\nu=0.
	\end{equation*}
	
	Next, we exploit the collisional symmetries from Lemma \ref{lemma:collisional:symmetries} in combination with the fact that the map
	\begin{equation*}
		(u,u_*,s)\mapsto (u',u_*',\frac{u-u_*}2)=
		\left(\frac{u+u_*}2+s,\frac{u+u_*}2-s,\frac{u-u_*}2\right)
	\end{equation*}
	is a volume-preserving involution on $\mathbb{R}^3$ to arrive at the result
	\begin{equation*}
		B\big(\widetilde \pi(t,x,u,\nu),\Theta(t,x,u,R\nu)\big)=0,
	\end{equation*}
	where $B$ is the bilinear form defined by
	\begin{equation*}
		\begin{aligned}
			&B(p,q)=
			\\
			&\quad\int
			\frac{
			\int_{\mathbb{S}^{d-2}\perp\frac{\nu+\nu_*}{2}}
			(p+p_*-p'-p_*')
			(q+q_*-q'-q_*')
			b\big(R(\nu-\nu_*),\widetilde\sigma\big) d\widetilde\sigma
			}
			{\cosh^\alpha\big(\frac u2\big)
			\cosh^\alpha\big(\frac {u_*}2\big)
			\cosh^\alpha\big(\frac {u'}2\big)
			\cosh^\alpha\big(\frac {u_*'}2\big)
			|\nu+\nu_*||\nu-\nu_*|}
			dtdxdud\nu du_* d\nu_* ds,
		\end{aligned}
	\end{equation*}
	for suitable functions $p(t,x,u,\nu)$ and $q(t,x,u,\nu)$. It can be shown that $B$ is bounded on $L^2$ in the sense that
	\begin{equation*}
		|B(p,q)|
		\lesssim \|b\|_{L^\infty}
		\left(
		\int \frac{p(t,x,u,\nu)^2}{\cosh^\alpha\big(\frac u2\big)}dtdxdud\nu
		\right)^\frac 12
		\left(
		\int \frac{q(t,x,u,\nu)^2}{\cosh^\alpha\big(\frac u2\big)}dtdxdud\nu
		\right)^\frac 12.
	\end{equation*}
	Therefore, a standard approximation argument allows us to infer that
	\begin{equation*}
		B\big(\widetilde \pi,\widetilde \pi \big)=0,
	\end{equation*}
	whereby, recalling that the cross-section $b(z,\sigma)$ is positive on $B(0,2R)\times\mathbb{S}^{d-1}$, we obtain the crucial functional relation
	\begin{equation}\label{functional:0}
		\widetilde\pi (u,\nu)+\widetilde\pi (u_*,\nu_*)-\widetilde\pi (u',\nu')-\widetilde\pi (u_*',\nu_*')=0,
	\end{equation}
	for almost all $(u,u_*,s,\nu,\nu_*,\widetilde\sigma)$, which expresses that $\widetilde\pi$ is a collision invariant, where we omit the dependence of $\widetilde\pi$ on $t$ and $x$.

	We have now reached the last stage of our proof, which consists in extracting suitable information from \eqref{functional:0} to characterize all collision invariants. To that end, we begin by fixing $u=u_*$ and $s=0$ in \eqref{functional:0}, which reduces it to the equation
	\begin{equation}\label{functional:1}
		\widetilde\pi (u,\nu)+\widetilde\pi (u,\nu_*)-\widetilde\pi (u,\nu')-\widetilde\pi (u,\nu_*')=0,
	\end{equation}
	for almost all $(\nu,\nu_*,\widetilde\sigma)$.
	A thorough analysis of collision invariants on the sphere satisfying \eqref{functional:1} is available in our companion article \cite{aa24}, where we have shown that the set of solutions to \eqref{functional:1} is comprised of the linear span of $\{1,\nu_1,\ldots,\nu_d\}$. It therefore follows that
	\begin{equation*}
		\widetilde \pi (t,x,u,\nu)=A(t,x,u)+C(t,x,u)\cdot \nu,
	\end{equation*}
	for some suitable coefficients $A\in\mathbb{R}$ and $C\in\mathbb{R}^d$.
	
	Then, going back to the original functional relation \eqref{functional:0} and setting $\nu=\nu_*$ therein, we deduce, omitting the dependence of functions on $t$ and $x$, that
	\begin{equation*}
		A(u)+A(u_*)=A\left(\frac{u+u_*}2+s\right)+A\left(\frac{u+u_*}2-s\right),
	\end{equation*}
	for all real values $u$, $u_*$ and $s$. In particular, fixing $u_*=0$ and introducing the auxiliary coefficients $\widetilde A(u)=A(u)-A(0)$, it is readily seen that $\widetilde A$ verifies Cauchy's functional relation, that is,
	\begin{equation*}
		\widetilde A(y+z)=\widetilde A(y)+\widetilde A(z),
	\end{equation*}
	for all $y,z\in\mathbb{R}$. In particular, by virtue of classical results on Cauchy's relation (see \cite{aa24} for a presentation of such results), we conclude that $\widetilde A(t,x,u)=A_1(t,x)+A_2(t,x)u$, for some suitable coefficients $A_1,A_2\in\mathbb{R}$.
	
	Finally, in order to describe $C(u)$, setting $u=u_*$ and $\widetilde\sigma$ in \eqref{functional:0} so that $\nu=\nu'$ and $\nu_*=\nu_*'$, we observe that
	\begin{equation*}
		C(u)=C(u+s),
	\end{equation*}
	for all real values $u$ and $s$, which implies that $C(t,x,u)$ must be constant in $u$.
	
	In conclusion, in view of \eqref{bound:pi}, we have now shown that
	\begin{equation}\label{characterization:limit:pi}
		\widetilde \pi(t,x,u,\nu)=A_1(t,x)+A_2(t,x)u+C(t,x)\cdot\nu,
	\end{equation}
	where the coefficients $A_1$, $A_2$ and $C$ belong to $L^\infty(dt;L^2(dx))$,
	which fully characterizes the velocity distribution of collision invariants satisfying \eqref{functional:0}.
	
	At last, recalling that $\widetilde\pi(t,x,u,\nu)$ is related to $g(t,x,\omega)$ through the identity \eqref{velocity:distributions}, a direct computation shows that
	\begin{equation*}
		g(t,x,\omega)=\rho(t,x)+U(t,x)\cdot\omega,
	\end{equation*}
	for some suitable choice of coefficients $\rho(t,x)$ and $U(t,x)$, and that
	\begin{equation*}
		\int_{\partial B(0,R)}g(t,x,\omega)^2d\omega=R^{d-1}|\mathbb{S}^{d-1}|\rho(t,x)^2+\frac{R^{d+1}|\mathbb{S}^{d-1}|}{d}U(t,x)^2
		\in L^\infty(dt;L^1(dx)),
	\end{equation*}
	which completes the proof of the proposition.
\end{proof}

\subsection{Relaxation of dilated fluctuations}

The proof of Proposition \ref{prop:instrumental} also yields a characterization of the thermodynamic equilibria for the dilated fluctuations $\widetilde g_\varepsilon$ in the hydrodynamic regime as $\varepsilon\to 0$. We provide a precise statement and justification of this result below.

\begin{prop}\label{prop:instrumental:dilated}
	Let $d\geq 3$. Consider a cross-section
	\begin{equation*}
		b(z,\sigma)=b\left(|z|,\frac{z}{|z|}\cdot\sigma\right)\geq 0
	\end{equation*}
	such that
	\begin{equation*}
		b(z,\sigma)\in L^\infty(\mathbb{R}^d\times\mathbb{S}^{d-1})\cap C(B(0,2R)\times\mathbb{S}^{d-1})
	\end{equation*}
	and
	\begin{equation*}
		b(z,\sigma)>0\quad\text{on }B(0,2R)\times\mathbb{S}^{d-1},
	\end{equation*}
	and a family of density distributions $0\leq f_\varepsilon(t,x,v)\leq\delta^{-1}$, with $\varepsilon>0$, such that the relative entropy bound \eqref{relative:entropy:1} and the entropy dissipation bound \eqref{entropy:dissipation:1} hold uniformly in $\varepsilon$, with some fixed parameters $\gamma>0$ and $\kappa>0$, where the Fermi--Dirac distribution $M_\varepsilon$ defined in \eqref{normalized:distribution} has a given parameter value $\tau>0$.
	Further suppose that
	\begin{equation*}
		0<\gamma=\tau< 1
		\qquad\text{and}\qquad
		\kappa>2\tau,
	\end{equation*}
	and consider the density fluctuations $g_\varepsilon(t,x,v)$ defined in \eqref{epsilon:fluctuations}, and its dilations $\widetilde g_\varepsilon (t,x,u,\omega)$ given in \eqref{dilated:fluctuations}.
	
	In accordance with Proposition \ref{prop:limit:characterization:dilated}, as $\varepsilon\to 0$, up to extraction of a subsequence, the family of dilated fluctuations $\widetilde g_\varepsilon$ converges in the weak topology of $L^1_\mathrm{loc}\big(dtdx;L^1\big(dud\omega\big)\big)$ toward a limit point
	\begin{equation*}
		\widetilde g\in L^\infty\left(dt;L^2\left(\cosh^2\left(\frac u2\right)dxdud\omega\right)\right).
	\end{equation*}
	
	Then, the limiting dilated density can be further characterized as
	\begin{equation*}
		\widetilde g(t,x,u,\omega)=\frac{\rho(t,x)+U(t,x)\cdot\omega+E(t,x)u}{4\cosh^2(\frac u2)},
	\end{equation*}
	where the density $\rho(t,x)$, the velocity field $U(t,x)$ and the energy density $E(t,x)$ belong to $L^\infty\big(dt;L^2\big(\mathbb{R}^d\big)\big)$.
\end{prop}

\begin{proof}
	First, we consider the renormalized density fluctuation $h_\varepsilon(t,x,v)$ introduced in \eqref{renormalized:fluctuation}, in the context of the proof of Proposition \ref{prop:instrumental}, and define its dilation
	\begin{equation*}
		\widetilde h_\varepsilon (t,x,u,\omega)=
		\left\{
		\begin{aligned}
			&\frac{\varepsilon^\tau|v|^{d-2}}{2R^{d-1}}
			h_\varepsilon(t,x,v) &&\text{if }u\geq -\frac{R^2}{\varepsilon^\tau},
			\\
			&0 &&\text{if }u< -\frac{R^2}{\varepsilon^\tau},
		\end{aligned}
		\right.
	\end{equation*}
	where the variables $v$, $u$ and $\omega$ are related through \eqref{spectrum:variables}. Then, recalling the change of variable formula \eqref{dilation:change:variable}, it follows from \eqref{fluctuation:expansion} that
	\begin{equation*}
		\widetilde g_\varepsilon= \widetilde h_\varepsilon+O(\varepsilon^{1-\tau})_{L^\infty (dt;L^1(dxdud\omega))},
	\end{equation*}
	whereby $\widetilde g$ is also a limit point of $\widetilde h_\varepsilon$ in the weak topology of $L^1_\mathrm{loc}\big(dtdx;L^1\big(dud\omega\big)\big)$.
	
	We also consider the fluctuation term $\pi_\varepsilon(t,x,v)$ and its dilated version $\widetilde \pi_\varepsilon(t,x,u,\nu)$, which are defined in \eqref{pi:fluctuation} and \eqref{pi:fluctuation:dilation}, respectively. In particular, we have that
	\begin{equation*}
		\widetilde h_\varepsilon(t,x,u,\omega)=\frac{\widetilde \pi_\varepsilon(t,x,u,\nu)}{4R^{d-1}\cosh^2\left(\frac u2\right)}.
	\end{equation*}
	Observe that the variables $\nu$ and $\omega$ satisfy that $\omega=R\nu$.
	
	It is then established in \eqref{bound:pi}, \eqref{convergence:pi} and \eqref{characterization:limit:pi}, up to extraction of subsequences, that
	\begin{equation*}
		\frac{\widetilde \pi_\varepsilon(t,x,u,\nu)}{4\cosh^2\left(\frac u2\right)}
		\rightharpoonup^*
		\frac{A(t,x)+B(t,x)u+C(t,x)\cdot\nu}{4\cosh^2(\frac u2)},
	\end{equation*}
	as $\varepsilon\to 0$, at least in the weak* topology of $L^\infty(dt;L^2(dxdud\nu))$, where the coefficients $A$, $B$ and $C$ belong to $L^\infty(dt;L^2(dx))$.
	
	All in all, we conclude that the limit point $\widetilde g$ can be written as
	\begin{equation*}
		\widetilde g(t,x,u,\omega)=\frac{\rho(t,x)+U(t,x)\cdot\omega+E(t,x)u}{4\cosh^2(\frac u2)},
	\end{equation*}
	for some suitable choice of coefficients $\rho(t,x)$, $U(t,x)$ and $E(t,x)$. Moreover, employing that
	\begin{equation*}
		\int_{\mathbb{R}}\frac{1}{4\cosh^2\left(\frac u2\right)}du=1
		\qquad\text{and}\qquad
		\int_{\mathbb{R}}\frac{u^2}{4\cosh^2\left(\frac u2\right)}du=\frac{\pi^2}3,
	\end{equation*}
	we readily compute that
	\begin{equation*}
		\begin{aligned}
			&\int_{\mathbb{R}\times\partial B(0,R)}
			\widetilde g(t,x,u,\omega)^2
			4\cosh^2\left(\frac u2\right)dud\omega
			\\
			&\quad =R^{d-1}|\mathbb{S}^{d-1}|\rho(t,x)^2+\frac{R^{d+1}|\mathbb{S}^{d-1}|}{d}U(t,x)^2
			+\frac{\pi^2}{3}R^{d-1}|\mathbb{S}^{d-1}|E(t,x)^2
			\in L^\infty(dt;L^1(dx)),
		\end{aligned}
	\end{equation*}
	which completes the proof of the proposition.
\end{proof}

\section{Removal of conservation defects}\label{removal:defects}

Proving Theorem \ref{thm:main:2} requires demonstrating that the conservation defects vanish in the low-temperature hydrodynamic regime $\varepsilon\to 0$. This crucial, technical result is presented in the following lemma.

It is important to note, however, that this result requires the imposition of the additional parameter restriction $\kappa+\tau<2$. Further observe that, as $\tau\to 0$, this restriction formally becomes $\kappa<2$, which matches the parameter range for the classical acoustic limit of the Boltzmann equation discussed in Section \ref{acoustic:limit}.

\begin{lem}\label{lemma:removal}
	In the proof of Theorem \ref{thm:main:2}, the conservation defects
	\begin{equation*}
		D_\varepsilon(t,x)=
		\frac{1}{\varepsilon^{\kappa+1}}\int_{\mathbb{R}^d} \frac{v^2-R^2}{\varepsilon^\tau}Q_\mathrm{FD}(f_\varepsilon)(t,x,v)\chi(v)dv
	\end{equation*}
	converge to zero in the strong topology of $L^1_{\mathrm{loc}}(dtdx)$, provided that the constant parameter $\alpha$ in \eqref{cutoff:function:0} is sufficiently large.
\end{lem}

\begin{proof}
	The control of conservation defects relies primarily on two estimates. First, on the bound
	\begin{equation}\label{defect:1}
		q_\varepsilon=O(1)_{L^2(bdtdxdvdv_*d\sigma)},
	\end{equation}
	established in Lemma \ref{lemma:renormalized:collision:bound} as a consequence of the entropy dissipation bound, where the renormalized collision integrands $q_\varepsilon$ are defined in \eqref{collision:integrand}. Second, on the asymptotic control \eqref{asymptotic:control} based on Lemma \ref{lemma:squareroot:relative:entropy}, which we reproduce here as
	\begin{equation}\label{defect:2}
		g_\varepsilon= \sqrt{\delta M_\varepsilon(1-\delta M_\varepsilon)}
		O(\varepsilon^{-\frac{\tau}2})_{L^\infty (dt;L^2(dxdv))}
		+O(\varepsilon^{1-\tau})_{L^\infty (dt;L^1(dxdv))}
	\end{equation}
	and follows from the control of the relative entropy.
	
	Next, we decompose $D_\varepsilon$ into
	\begin{equation*}
		D_\varepsilon=\sum_{i=1}^6D_\varepsilon^i,
	\end{equation*}
	with
	\begin{equation*}
		\begin{aligned}
			D_\varepsilon^1&=\delta\varepsilon^{1-\tau}
			\int_{\mathbb{R}^d\times\mathbb{R}^d\times\mathbb{S}^{d-1}}
			q_\varepsilon^2 u\chi bdvdv_*d\sigma,
			\\
			D_\varepsilon^2&=\frac{2\delta^\frac 12}{\varepsilon^{\frac{\kappa+\tau}{2}}}
			\int_{\mathbb{R}^d\times\mathbb{R}^d\times\mathbb{S}^{d-1}}
			q_\varepsilon
			\sqrt{G_\varepsilon}u\chi(1-\chi_*) bdvdv_*d\sigma,
			\\
			D_\varepsilon^3&=\frac{2\delta^\frac 12}{\varepsilon^{\frac{\kappa+\tau}{2}}}
			\int_{\mathbb{R}^d\times\mathbb{R}^d\times\mathbb{S}^{d-1}}
			q_\varepsilon
			\sqrt{G_\varepsilon}u\chi\chi_*\big(1-\chi'\big) bdvdv_*d\sigma,
			\\
			D_\varepsilon^4&=\frac{2\delta^\frac 12}{\varepsilon^{\frac{\kappa+\tau}{2}}}
			\int_{\mathbb{R}^d\times\mathbb{R}^d\times\mathbb{S}^{d-1}}
			q_\varepsilon
			\sqrt{G_\varepsilon}u\chi\chi_*\chi'\big(1-\chi_*'\big) bdvdv_*d\sigma,
			\\
			D_\varepsilon^5&=\frac{1}{\varepsilon^{\kappa+1}}
			\int_{\mathbb{R}^d\times\mathbb{R}^d\times\mathbb{S}^{d-1}}\big(G_\varepsilon'-G_\varepsilon\big)u\chi\chi_*\chi'\chi_*' bdvdv_*d\sigma,
			\\
			D_\varepsilon^6&=-\delta\varepsilon^{1-\tau}
			\int_{\mathbb{R}^d\times\mathbb{R}^d\times\mathbb{S}^{d-1}}
			q_\varepsilon^2 u\chi\chi_*\chi'\chi_*' bdvdv_*d\sigma,
		\end{aligned}
	\end{equation*}
	where we used the notation
	\begin{equation*}
		\begin{aligned}
			G_\varepsilon&=ff_*(1-\delta f')(1-\delta f_*'),
			\\
			G_\varepsilon'&=f'f_*'(1-\delta f)(1-\delta f_*),
		\end{aligned}
	\end{equation*}
	for the sake of convenience.
	
	It is then readily seen that $D_\varepsilon^5=0$, by a straightforward application of collisional symmetries \eqref{change:prepost} and \eqref{change:exchange}. Furthermore, in view of \eqref{defect:1}, we find that
	\begin{equation*}
		D_\varepsilon^1,D_\varepsilon^6=O(\varepsilon^{1-\tau}|\log\varepsilon|)_{L^1(dtdx)}.
	\end{equation*}
	However, the control of $D_\varepsilon^2$, $D_\varepsilon^3$ and $D_\varepsilon^4$ is more involved.
	
	Due to their similitude, the terms $D_\varepsilon^2$ and $D_\varepsilon^4$ are handled together by first writing the decompositions
	\begin{equation*}
		D_\varepsilon^2=\sum_{i=1}^3 D_\varepsilon^{2,i}
		\quad\text{and}\quad
		D_\varepsilon^4=\sum_{i=1}^3 D_\varepsilon^{4,i},
	\end{equation*}
	where
	\begin{equation*}
		\begin{aligned}
			D_\varepsilon^{2,1}&=\frac{2\delta^\frac 12}{\varepsilon^{\frac{\kappa+\tau}{2}}}
			\int_{\mathbb{R}^d\times\mathbb{R}^d\times\mathbb{S}^{d-1}}
			q_\varepsilon
			\sqrt{G_\varepsilon}u\chi(1-\chi_*)\mathds{1}_{\{|v_*|> R\}} bdvdv_*d\sigma,
			\\
			D_\varepsilon^{2,2}&=\frac{2\delta^\frac 12}{\varepsilon^{\frac{\kappa+\tau}{2}}}
			\int_{\mathbb{R}^d\times\mathbb{R}^d\times\mathbb{S}^{d-1}}
			q_\varepsilon
			\sqrt{G_\varepsilon'}u\chi(1-\chi_*)\mathds{1}_{\{|v_*|< R\}} bdvdv_*d\sigma,
			\\
			D_\varepsilon^{2,3}&= -2\delta\varepsilon^{1-\tau}
			\int_{\mathbb{R}^d\times\mathbb{R}^d\times\mathbb{S}^{d-1}}
			q_\varepsilon^2u\chi(1-\chi_*)\mathds{1}_{\{|v_*|< R\}} bdvdv_*d\sigma,
		\end{aligned}
	\end{equation*}
	and
	\begin{equation*}
		\begin{aligned}
			D_\varepsilon^{4,1}&=\frac{2\delta^\frac 12}{\varepsilon^{\frac{\kappa+\tau}{2}}}
			\int_{\mathbb{R}^d\times\mathbb{R}^d\times\mathbb{S}^{d-1}}
			q_\varepsilon
			\sqrt{G_\varepsilon'}u\chi\chi_*\chi'\big(1-\chi_*'\big)\mathds{1}_{\{|v_*'|>R\}} bdvdv_*d\sigma,
			\\
			D_\varepsilon^{4,2}&=\frac{2\delta^\frac 12}{\varepsilon^{\frac{\kappa+\tau}{2}}}
			\int_{\mathbb{R}^d\times\mathbb{R}^d\times\mathbb{S}^{d-1}}
			q_\varepsilon
			\sqrt{G_\varepsilon}u\chi\chi_*\chi'\big(1-\chi_*'\big)\mathds{1}_{\{|v_*'|<R\}} bdvdv_*d\sigma,
			\\
			D_\varepsilon^{4,3}&=-2\delta \varepsilon^{1-\tau}
			\int_{\mathbb{R}^d\times\mathbb{R}^d\times\mathbb{S}^{d-1}}
			q_\varepsilon^2
			u\chi\chi_*\chi'\big(1-\chi_*'\big)\mathds{1}_{\{|v_*'|>R\}} bdvdv_*d\sigma.
		\end{aligned}
	\end{equation*}

	It then follows from \eqref{defect:1} that
	\begin{equation*}
		D_\varepsilon^{2,3},D_\varepsilon^{4,3}=O(\varepsilon^{1-\tau}|\log\varepsilon|)_{L^1(dtdx)}.
	\end{equation*}
	Moreover, for any compact set $K\subset\mathbb{R}^+\times\mathbb{R}^d$, we obtain, since $f_\varepsilon$ is uniformly bounded pointwise, that
	\begin{equation*}
		\begin{aligned}
			\int_K\big(|D_\varepsilon^{2,1}|+|D_\varepsilon^{4,1}|\big)dtdx
			&\lesssim
			\frac{|\log\varepsilon|}{\varepsilon^{\frac{\kappa+\tau}{2}}}
			\left(\int_{K\times\mathbb{R}^d\times\mathbb{R}^d\times\mathbb{S}^{d-1}}
			G_\varepsilon \chi(1-\chi_*)\mathds{1}_{\{|v_*|> R\}} dtdxdvdv_*d\sigma\right)^\frac 12
			\\
			&\lesssim
			\frac{|\log\varepsilon|}{\varepsilon^{\frac{\kappa+\tau}{2}}}
			\left(\int_{\mathbb{R}^d}\chi dv\right)^\frac 12
			\left(\int_{K\times\mathbb{R}^d}
			f_{\varepsilon} (1-\chi)\mathds{1}_{\{|v|> R\}} dtdxdv\right)^\frac 12,
		\end{aligned}
	\end{equation*}
	and
	\begin{equation*}
		\begin{aligned}
			\int_K\big(|D_\varepsilon^{2,2}|+|D_\varepsilon^{4,2}|\big)dtdx
			&\lesssim
			\frac{|\log\varepsilon|}{\varepsilon^{\frac{\kappa+\tau}{2}}}
			\left(\int_{K\times\mathbb{R}^d\times\mathbb{R}^d\times\mathbb{S}^{d-1}}
			G_\varepsilon' \chi(1-\chi_*)\mathds{1}_{\{|v_*|< R\}} dtdxdvdv_*d\sigma\right)^\frac 12
			\\
			&\lesssim
			\frac{|\log\varepsilon|}{\varepsilon^{\frac{\kappa+\tau}{2}}}
			\left(\int_{\mathbb{R}^d}\chi dv\right)^\frac 12
			\left(\int_{K\times\mathbb{R}^d}
			(1-\delta f_\varepsilon) (1-\chi)\mathds{1}_{\{|v|< R\}} dtdxdv\right)^\frac 12.
		\end{aligned}
	\end{equation*}
	In particular, employing \eqref{defect:2} and recalling that
	\begin{equation*}
		\chi(v)=\mathds{1}_{\left\{\frac{|v^2-R^2|}{\varepsilon^\tau}< -\log\varepsilon^\alpha\right\}}
		=\mathds{1}_{\left\{|u|< -\log\varepsilon^\alpha\right\}},
	\end{equation*}
	the preceding estimates lead to
	\begin{equation}\label{defect:3}
		\begin{aligned}
			\int_K\big(|D_\varepsilon^{2,1}|+|D_\varepsilon^{4,1}|\big)dtdx
			&\lesssim
			\frac{|\log\varepsilon|^\frac 32}{\varepsilon^{\frac{\kappa}{2}}}
			\left(\int_{K\times\mathbb{R}^d}
			f_{\varepsilon} (1-\chi)\mathds{1}_{\{|v|> R\}} dtdxdv\right)^\frac 12
			\\
			&\lesssim
			\frac{|\log\varepsilon|^\frac 32}{\varepsilon^{\frac{\kappa}{2}}}
			\left(\int_{K\times\mathbb{R}^d}
			\big(M_{\varepsilon}+\varepsilon|g_{\varepsilon}|\big) \mathds{1}_{\left\{\frac{v^2-R^2}{\varepsilon^\tau}\geq -\log\varepsilon^\alpha\right\}} dtdxdv\right)^\frac 12
			\\
			&\lesssim
			\frac{|\log\varepsilon|^\frac 32}{\varepsilon^{\frac{\kappa}{2}}}
			\left(\varepsilon^\frac\alpha 2
			+\varepsilon^{1+\frac\alpha 4}+\varepsilon^{2-\tau}\right)^\frac 12
			\lesssim |\log\varepsilon|^\frac 32\varepsilon^{1-\frac{\kappa+\tau}2}
		\end{aligned}
	\end{equation}
	and
	\begin{equation}\label{defect:4}
		\begin{aligned}
			\int_K\big(|D_\varepsilon^{2,2}|+|D_\varepsilon^{4,2}|\big)&dtdx
			\\
			&\lesssim
			\frac{|\log\varepsilon|^\frac 32}{\varepsilon^{\frac{\kappa}{2}}}
			\left(\int_{K\times\mathbb{R}^d}
			(1-\delta f_\varepsilon) (1-\chi)\mathds{1}_{\{|v|< R\}} dtdxdv\right)^\frac 12
			\\
			&\lesssim
			\frac{|\log\varepsilon|^\frac 32}{\varepsilon^{\frac{\kappa}{2}}}
			\left(\int_{K\times\mathbb{R}^d}
			\big((1-\delta M_{\varepsilon})+\varepsilon|g_{\varepsilon}|\big) \mathds{1}_{\left\{\frac{v^2-R^2}{\varepsilon^\tau}\leq \log\varepsilon^\alpha\right\}} dtdxdv\right)^\frac 12
			\\
			&\lesssim
			\frac{|\log\varepsilon|^\frac 32}{\varepsilon^{\frac{\kappa}{2}}}
			\left(\varepsilon^\frac\alpha 2
			+\varepsilon^{1+\frac\alpha 4}+\varepsilon^{2-\tau}\right)^\frac 12
			\lesssim |\log\varepsilon|^\frac 32\varepsilon^{1-\frac{\kappa+\tau}2},
		\end{aligned}
	\end{equation}
	provided the constant $\alpha$ is sufficiently large. All in all, gathering the preceding bounds, we deduce that
	\begin{equation*}
		D_\varepsilon^2,D_\varepsilon^4=O(\varepsilon^{1-\frac{\kappa+\tau}2}|\log\varepsilon|^\frac 32)_{L^1_\mathrm{loc}(dtdx)}+O(\varepsilon^{1-\tau}|\log\varepsilon|)_{L^1(dtdx)},
	\end{equation*}
	which establishes the vanishing of $D_\varepsilon^2$ and $D_\varepsilon^4$, as $\varepsilon\to 0$, provided that $\kappa+\tau<2$.
	
	The control of the last term $D_\varepsilon^3$ is analogous to the preceding estimates. However, it requires a few extra technical steps. Thus, we begin with the technical estimate, for any $v_*\in\mathbb{R}^d$,
	\begin{equation}\label{technical:1}
		\begin{aligned}
			\int_{\mathbb{R}^d\times\mathbb{S}^{d-1}}\chi'\chi_*' dvd\sigma&=\int_{\mathbb{R}^d\times\mathbb{S}^{d-1}}\chi\left(v_*+\frac V2+\frac{|V|}2\sigma\right)\chi\left(v_*+\frac V2-\frac{|V|}2\sigma\right) dVd\sigma
			\\
			&=2\int_{\{\sigma\cdot V>0\}}\chi\left(v_*+\frac V2+\frac{|V|}2\sigma\right)\chi\left(v_*+\frac V2-\frac{|V|}2\sigma\right) dVd\sigma
			\\
			&=2\int_{\{\sqrt 2 \sigma\cdot W>|W|\}}\chi\left(v_*+W\right)\chi\left(v_*+W-\frac{W^2}{\sigma\cdot W}\sigma\right) \frac{2^{d-1}W^2}{(\sigma\cdot W)^2}dWd\sigma
			\\
			&\leq 2^{d+1}\int_{\mathbb{R}^d}\chi(W)dW=O(\varepsilon^\tau|\log\varepsilon|),
		\end{aligned}
	\end{equation}
	where we employed the change of variable
	\begin{equation*}
		W=\frac V2+\frac{|V|}2\sigma\in\mathbb{R}^d\cap\{\sigma\cdot W>0\},\qquad
		V=2W-\frac{W^2}{\sigma\cdot W}\sigma\in\mathbb{R}^d\cap\{\sigma\cdot V>-|V|\},
	\end{equation*}
	with Jacobian determinant
	\begin{equation*}
		dV=\frac{2^{d-1}W^2}{(\sigma\cdot W)^2}dW.
	\end{equation*}
	Note that this change of variable formula can be readily established through a straightforward use of spherical coordinates. The asymptotic control \eqref{technical:1} will come in handy later on in the proof.
	
	Then, as before, we proceed by decomposing $D_\varepsilon^3$ into
	\begin{equation*}
		D_\varepsilon^3=\sum_{i=1}^3 D_\varepsilon^{3,i},
	\end{equation*}
	where
	\begin{equation*}
		\begin{aligned}
			D_\varepsilon^{3,1}&=\frac{2\delta^\frac 12}{\varepsilon^{\frac{\kappa+\tau}{2}}}
			\int_{\mathbb{R}^d\times\mathbb{R}^d\times\mathbb{S}^{d-1}}
			q_\varepsilon
			\sqrt{G_\varepsilon'}u\chi\chi_*\big(1-\chi'\big)\mathds{1}_{\{|v'|>R\}} bdvdv_*d\sigma,
			\\
			D_\varepsilon^{3,2}&=\frac{2\delta^\frac 12}{\varepsilon^{\frac{\kappa+\tau}{2}}}
			\int_{\mathbb{R}^d\times\mathbb{R}^d\times\mathbb{S}^{d-1}}
			q_\varepsilon
			\sqrt{G_\varepsilon}u\chi\chi_*\big(1-\chi'\big)\mathds{1}_{\{|v'|<R\}} bdvdv_*d\sigma,
			\\
			D_\varepsilon^{3,3}&=-2\delta \varepsilon^{1-\tau}
			\int_{\mathbb{R}^d\times\mathbb{R}^d\times\mathbb{S}^{d-1}}
			q_\varepsilon^2
			u\chi\chi_*\big(1-\chi'\big)\mathds{1}_{\{|v'|>R\}} bdvdv_*d\sigma.
		\end{aligned}
	\end{equation*}
	It is then readily seen from \eqref{defect:1} that
	\begin{equation*}
		D_\varepsilon^{3,3}=O(\varepsilon^{1-\tau}|\log\varepsilon|)_{L^1(dtdx)}.
	\end{equation*}
	In order to control $D_\varepsilon^{3,1}$ and $D_\varepsilon^{3,2}$, we repeat the estimates leading to \eqref{defect:3} and \eqref{defect:4}, and obtain, provided $\alpha$ is large enough, that
	\begin{equation*}
		\begin{aligned}
			\int_K|D_\varepsilon^{3,1}|dtdx
			&\lesssim
			\frac{|\log\varepsilon|}{\varepsilon^{\frac{\kappa+\tau}{2}}}
			\left(\sup_{v_*\in\mathbb{R}^d}\int_{\mathbb{R}^d\times\mathbb{S}^{d-1}}\chi'\chi_*' dvd\sigma\right)^\frac 12
			\\
			&\quad\times\left(\int_{K\times\mathbb{R}^d}
			f_{\varepsilon} (1-\chi)\mathds{1}_{\{|v|> R\}} dtdxdv\right)^\frac 12
			\\
			&\lesssim |\log\varepsilon|\varepsilon^{1-\frac{\kappa}2-\tau}
			\left(\sup_{v_*\in\mathbb{R}^d}\int_{\mathbb{R}^d\times\mathbb{S}^{d-1}}\chi'\chi_*' dvd\sigma\right)^\frac 12
		\end{aligned}
	\end{equation*}
	and
	\begin{equation*}
		\begin{aligned}
			\int_K|D_\varepsilon^{3,2}|dtdx
			&\lesssim
			\frac{|\log\varepsilon|}{\varepsilon^{\frac{\kappa+\tau}{2}}}
			\left(\sup_{v_*\in\mathbb{R}^d}\int_{\mathbb{R}^d\times\mathbb{S}^{d-1}}\chi'\chi_*' dvd\sigma\right)^\frac 12
			\\
			&\quad\times\left(\int_{K\times\mathbb{R}^d}
			(1-\delta f_\varepsilon) (1-\chi)\mathds{1}_{\{|v|< R\}} dtdxdv\right)^\frac 12
			\\
			&\lesssim |\log\varepsilon|\varepsilon^{1-\frac{\kappa}2-\tau}
			\left(\sup_{v_*\in\mathbb{R}^d}\int_{\mathbb{R}^d\times\mathbb{S}^{d-1}}\chi'\chi_*' dvd\sigma\right)^\frac 12.
		\end{aligned}
	\end{equation*}
	Therefore, by making use of \eqref{technical:1}, we conclude that	
	\begin{equation*}
		D_\varepsilon^{3,1},D_\varepsilon^{3,2}=O(\varepsilon^{1-\frac{\kappa+\tau}2}|\log\varepsilon|^\frac 32)_{L^1_\mathrm{loc}(dtdx)},
	\end{equation*}
	whence
	\begin{equation*}
		D_\varepsilon^3=O(\varepsilon^{1-\frac{\kappa+\tau}2}|\log\varepsilon|^\frac 32)_{L^1_\mathrm{loc}(dtdx)}+O(\varepsilon^{1-\tau}|\log\varepsilon|)_{L^1(dtdx)},
	\end{equation*}
	which shows the vanishing of $D_\varepsilon^3$, provided that $\kappa+\tau<2$.
	
	Finally, by gathering the preceding asymptotic estimates for each $D_\varepsilon^i$, we deduce that
	\begin{equation*}
		D_\varepsilon=O(\varepsilon^{1-\frac{\kappa+\tau}2}|\log\varepsilon|^\frac 32)_{L^1_\mathrm{loc}(dtdx)}+O(\varepsilon^{1-\tau}|\log\varepsilon|)_{L^1(dtdx)},
	\end{equation*}
	which, since $\kappa+\tau<2$, completes the proof that conservation defects vanish in the strong topology of integrable functions.
\end{proof}

\bibliographystyle{alpha}
\bibliography{quantum}

\begin{thebibliography}{Zak15b}

\bibitem[AA24]{aa24}
Benjamin Anwasia and Diogo Ars\'enio.
\newblock Quantized collision invariants on the sphere.
\newblock {\em Commun. Math.}, 32(3):229--239, 2024.

\bibitem[AA25]{aa25}
Benjamin Anwasia and Diogo Ars\'enio.
\newblock Hydrodynamics of degenerate {F}ermi gases on {F}ermi circles.
\newblock 2025.
\newblock In preparation.

\bibitem[Ars12]{a12}
Diogo Ars\'enio.
\newblock From {B}oltzmann's equation to the incompressible
  {N}avier-{S}tokes-{F}ourier system with long-range interactions.
\newblock {\em Arch. Ration. Mech. Anal.}, 206(2):367--488, 2012.

\bibitem[ASR19]{asr19}
Diogo Ars\'{e}nio and Laure Saint-Raymond.
\newblock {\em From the {V}lasov-{M}axwell-{B}oltzmann system to incompressible
  viscous electro-magneto-hydrodynamics. {V}ol. 1}.
\newblock EMS Monographs in Mathematics. European Mathematical Society (EMS),
  Z\"{u}rich, 2019.

\bibitem[BG49]{bg49}
D.~Bohm and E.~P. Gross.
\newblock Theory of plasma oscillations. a. origin of medium-like behavior.
\newblock {\em Phys. Rev.}, 75:1851--1864, Jun 1949.

\bibitem[BGL91]{bgl91}
Claude Bardos, Fran\c{c}ois Golse, and C.~David Levermore.
\newblock Fluid dynamic limits of kinetic equations. {I}. {F}ormal derivations.
\newblock {\em J. Statist. Phys.}, 63(1-2):323--344, 1991.

\bibitem[BGL93]{bgl93}
Claude Bardos, Fran\c{c}ois Golse, and C.~David Levermore.
\newblock Fluid dynamic limits of kinetic equations. {II}. {C}onvergence proofs
  for the {B}oltzmann equation.
\newblock {\em Comm. Pure Appl. Math.}, 46(5):667--753, 1993.

\bibitem[BGL98]{bgl98}
Claude Bardos, Fran\c{c}ois Golse, and C.~David Levermore.
\newblock Acoustic and {S}tokes limits for the {B}oltzmann equation.
\newblock {\em C. R. Acad. Sci. Paris S\'er. I Math.}, 327(3):323--328, 1998.

\bibitem[BGL00]{bgl00}
Claude Bardos, Fran\c{c}ois Golse, and C.~David Levermore.
\newblock The acoustic limit for the {B}oltzmann equation.
\newblock {\em Arch. Ration. Mech. Anal.}, 153(3):177--204, 2000.

\bibitem[Bri15]{b15}
Marc Briant.
\newblock From the {B}oltzmann equation to the incompressible {N}avier-{S}tokes
  equations on the torus: a quantitative error estimate.
\newblock {\em J. Differential Equations}, 259(11):6072--6141, 2015.

\bibitem[CC90]{cc90}
Sydney Chapman and Thomas~George Cowling.
\newblock {\em The mathematical theory of nonuniform gases}.
\newblock Cambridge Mathematical Library. Cambridge University Press,
  Cambridge, third edition, 1990.
\newblock An account of the kinetic theory of viscosity, thermal conduction and
  diffusion in gases, In co-operation with D. Burnett, With a foreword by Carlo
  Cercignani.

\bibitem[CIP94]{CIP94}
Carlo Cercignani, Reinhard Illner, and Mario Pulvirenti.
\newblock {\em The mathematical theory of dilute gases}, volume 106 of {\em
  Applied Mathematical Sciences}.
\newblock Springer-Verlag, New York, 1994.

\bibitem[DL89]{DiPernaLions89}
Ronald~J. DiPerna and Pierre-Louis Lions.
\newblock On the {C}auchy problem for {B}oltzmann equations: global existence
  and weak stability.
\newblock {\em Ann. of Math. (2)}, 130(2):321--366, 1989.

\bibitem[DL91]{dl91}
R.~J. DiPerna and P.-L. Lions.
\newblock Global solutions of {B}oltzmann's equation and the entropy
  inequality.
\newblock {\em Arch. Rational Mech. Anal.}, 114(1):47--55, 1991.

\bibitem[Dol94]{d94}
Jean Dolbeault.
\newblock Kinetic models and quantum effects: a modified {B}oltzmann equation
  for {F}ermi-{D}irac particles.
\newblock {\em Arch. Rational Mech. Anal.}, 127(2):101--131, 1994.

\bibitem[GL02]{gl02}
Fran\c{c}ois Golse and C.~David Levermore.
\newblock Stokes-{F}ourier and acoustic limits for the {B}oltzmann equation:
  convergence proofs.
\newblock {\em Comm. Pure Appl. Math.}, 55(3):336--393, 2002.

\bibitem[GSR04]{gsr04}
Fran\c{c}ois Golse and Laure Saint-Raymond.
\newblock The {N}avier-{S}tokes limit of the {B}oltzmann equation for bounded
  collision kernels.
\newblock {\em Invent. Math.}, 155(1):81--161, 2004.

\bibitem[Han25]{h25}
Bowen Han.
\newblock The incompressible {N}avier-{S}tokes-{F}ourier-{P}oisson limit of the
  {V}lasov-{P}oisson-{B}oltzmann-{F}ermi-{D}irac equation.
\newblock {\em J. Differential Equations}, 426:104--185, 2025.

\bibitem[JLM10]{jlm10}
Ning Jiang, C.~David Levermore, and Nader Masmoudi.
\newblock Remarks on the acoustic limit for the {B}oltzmann equation.
\newblock {\em Comm. Partial Differential Equations}, 35(9):1590--1609, 2010.

\bibitem[JXZ22]{jxz22}
Ning Jiang, Linjie Xiong, and Kai Zhou.
\newblock The incompressible {N}avier-{S}tokes-{F}ourier limit from
  {B}oltzmann-{F}ermi-{D}irac equation.
\newblock {\em J. Differential Equations}, 308:77--129, 2022.

\bibitem[JZ24a]{jz24jde}
Ning Jiang and Kai Zhou.
\newblock The acoustic limit from the {B}oltzmann equation with {F}ermi-{D}irac
  statistics.
\newblock {\em J. Differential Equations}, 398:344--372, 2024.

\bibitem[JZ24b]{jz24cmp}
Ning Jiang and Kai Zhou.
\newblock The compressible {E}uler and acoustic limits from quantum {B}oltzmann
  equation with {F}ermi-{D}irac statistics.
\newblock {\em Comm. Math. Phys.}, 405(2):Paper No. 23, 58, 2024.

\bibitem[Kit04]{Kittel2004}
Charles Kittel.
\newblock {\em Introduction to Solid State Physics}.
\newblock John Wiley \& Sons, 8th edition, 2004.

\bibitem[Lio94]{Lions94}
Pierre-Louis Lions.
\newblock Compactness in {B}oltzmann's equation via {F}ourier integral
  operators and applications. {III}.
\newblock {\em J. Math. Kyoto Univ.}, 34(3):539--584, 1994.

\bibitem[LL80]{LandauVol5}
Lev~Davidovich Landau and Evgeny~Mikhailovich Lifshitz.
\newblock {\em Statistical Physics, Part 1}, volume~5 of {\em Course of
  Theoretical Physics}.
\newblock Butterworth--Heineman, 1980.

\bibitem[LM10]{lm10}
C.~David Levermore and Nader Masmoudi.
\newblock From the {B}oltzmann equation to an incompressible
  {N}avier-{S}tokes-{F}ourier system.
\newblock {\em Arch. Ration. Mech. Anal.}, 196(3):753--809, 2010.

\bibitem[LP80]{LandauVol9}
Evgeny~Mikhailovich Lifshitz and Lev~Petrovich Pitaevskii.
\newblock {\em Statistical Physics, Part 2}, volume~9 of {\em Course of
  Theoretical Physics}.
\newblock Butterworth--Heineman, 1980.

\bibitem[LP81]{LandauVol10}
Evgeny~Mikhailovich Lifshitz and Lev~Petrovich Pitaevskii.
\newblock {\em Physical Kinetics}, volume~10 of {\em Course of Theoretical
  Physics}.
\newblock Pergamon Press, 1981.

\bibitem[PB21]{Pathria2021StatisticalMechanics}
Raj~Kumar Pathria and Paul~Dudley Beale.
\newblock {\em Statistical Mechanics}.
\newblock Academic Press, 4th edition, 2021.

\bibitem[Per89]{p89}
B.~Perthame.
\newblock Global existence to the {BGK} model of {B}oltzmann equation.
\newblock {\em J. Differential Equations}, 82(1):191--205, 1989.

\bibitem[SE10]{Shukla:2010}
P.~K. Shukla and B.~Eliasson.
\newblock Nonlinear aspects of quantum plasma physics.
\newblock {\em Phys. Usp.}, 53(1):51--76, 2010.

\bibitem[SR09]{SR09}
Laure Saint-Raymond.
\newblock {\em Hydrodynamic limits of the {B}oltzmann equation}, volume 1971 of
  {\em Lecture Notes in Mathematics}.
\newblock Springer-Verlag, Berlin, 2009.

\bibitem[TL29]{tl29}
Lewi Tonks and Irving Langmuir.
\newblock Oscillations in ionized gases.
\newblock {\em Phys. Rev.}, 33:195--210, Feb 1929.

\bibitem[Zak15a]{z15}
Timofey Zakrevskiy.
\newblock The {E}uler limit for kinetic models with {F}ermi-{D}irac statistics.
\newblock {\em Asymptot. Anal.}, 95(1-2):59--77, 2015.

\bibitem[Zak15b]{z15thesis}
Timofey Zakrevskiy.
\newblock {\em Kinetic models in the near-equilibrium regime}.
\newblock 2015.
\newblock Thesis (Ph.D.)--\'Ecole Polytechnique.

\bibitem[ZJ12]{Zhu_2012}
Jun Zhu and Peiyong Ji.
\newblock Dispersion relation and landau damping of waves in high-energy
  density plasmas.
\newblock {\em Plasma Physics and Controlled Fusion}, 54(6):065004, may 2012.

\end{thebibliography}

\end{document}